\documentclass[ejs]{imsart}

\RequirePackage{amsthm,amsmath,amsfonts,amssymb}
\RequirePackage[authoryear]{natbib} 
\RequirePackage[colorlinks,citecolor=blue,urlcolor=blue]{hyperref}
\RequirePackage{mathtools}
\RequirePackage{dsfont}
\RequirePackage{ulem}
\usepackage{xcolor}

\startlocaldefs

\theoremstyle{plain}
\newtheorem{theorem}{{Theorem}}[section] 
\newtheorem*{theorem*}{{Theorem}}
\newtheorem{proposition}[theorem]{Proposition}
\newtheorem*{proposition*}{Proposition}
\newtheorem{corollary}[theorem]{Corollary}
\newtheorem*{corollary*}{Corollary}
\newtheorem{lemma}[theorem]{Lemma}
\newtheorem*{lemma*}{Lemma}
\newtheorem{assumption}[theorem]{Assumption}
\newtheorem*{assumption*}{Assumption}
\newtheorem{definition}[theorem]{Definition}
\newtheorem*{definition*}{Definition}

\theoremstyle{remark}

\newtheorem*{notation*}{Notation}
\newtheorem*{remark*}{Remark}

\newcommand{\Cov}{\mathrm{Cov}}
\newcommand{\Trace}{\mathrm{Tr}}

\endlocaldefs

\begin{document}

\begin{frontmatter}

\title{$U$-statistics on bipartite exchangeable networks}
\runtitle{$U$-statistics on bipartite exchangeable networks}

\begin{aug}

\author[A]{\fnms{T\^am} \snm{Le Minh}\ead[label=e1]{tam.le-minh@inrae.fr}}

\address[A]{Universit\'e Paris-Saclay, AgroParisTech, INRAE, UMR MIA Paris-Saclay, 91120, Palaiseau, France, \printead{e1}}
\end{aug}

\begin{abstract}
Bipartite networks with exchangeable nodes can be represented by row-column exchangeable matrices. A quadruplet is a submatrix of size $2 \times 2$. A quadruplet $U$-statistic is the average of a function on a quadruplet over all the quadruplets of a matrix. We prove several asymptotic results for quadruplet $U$-statistics on row-column exchangeable matrices, including a weak convergence result in the general case and a central limit theorem when the matrix is also dissociated. These results are applied to statistical inference in network analysis. We suggest a method to perform parameter estimation, network comparison and motifs count for a particular family of row-column exchangeable network models: the bipartite expected degree distribution (BEDD) models. These applications are illustrated by simulations.
\end{abstract}

\begin{keyword}[class=MSC2020]
\kwd[Primary ]{60F05}
\kwd[; secondary ]{60G09}
\kwd{62G20}
\end{keyword}

\begin{keyword}
\kwd{$U$-statistics}
\kwd{row-column exchangeability}
\kwd{central limit theorem}
\kwd{network inference}
\end{keyword}

\end{frontmatter}

\section{Introduction \label{sec:intro}}

\paragraph{RCE matrices} Networks arise naturally when considering interaction data. The nodes of a network represent the entities of a system and an edge between two nodes represents the interaction between the associated entities. The network is bipartite when there are two different sets of nodes, and edges only link nodes of different types. A natural representation for a bipartite network is its rectangular adjacency matrix. The rows and columns of an adjacency matrix $Y$ represent the two different types of nodes and each entry $Y_{ij}$ encodes the interaction between the nodes associated to row $i$ and column $j$, e.g. for binary networks, $Y_{ij} = 1$ if $i$ and $j$ interact and $Y_{ij} = 0$ else, or for weighted networks, $Y_{ij}$ is the weight of the edge linking $i$ and $j$.

Many probabilistic network models assume that the network units, either edges or nodes, are exchangeable, i.e. are invariant by permutation. In the adjacency matrix of a bipartite network, the edge-exchangeability corresponds to the exchangeability of all its entries (full exchangeability), such as in \cite{adamczak2016circular}, while the node-exchangeability refers to the exchangeability of its rows and columns, such as in \cite{aldous1981representations}.

Recent developments have been made for edge-exchangeable models \citep{cai2016edge, williamson2016nonparametric}, but node-exchangeable models have a longer history for both unipartite and bipartite networks and encompasses families of models such as the stochastic block model \citep{holland1983stochastic, snijders1997estimation}, the latent block model \citep{govaert2003clustering}, the latent position model \citep{hoff2002latent} or the random dot product graph model \citep{young2007random}. Implicitly, the exchangeability of the network units is associated with a sampling assumption. The choice of whether considering the exchangeability for edges or nodes depends on what is assumed to be sampled to observe the networks, whether it be edges or nodes \citep{crane2018edge}.  

Let us observe a bipartite network represented by a finite submatrix of size $m \times n$. We assume that the nodes of the same type are infinitely exchangeable, which means that the row elements and the column elements of the adjacency matrix are separately invariant under infinite permutation. The infinite exchangeability assumption is equivalent to considering that this observed network is made of the first $m$ rows and $n$ columns of an infinite adjacency matrix, whose rows and columns are exchangeable. This assumption is similar to \cite{orbanz2014bayesian} for unipartite networks and provides a consistent framework to analyze networks of different sizes. It can be used with many random network models, including the ones listed above (stochastic block model, latent block model, latent position model and random dot product graph model). 

However, it has to be distinguished from the finitely exchangeable case. Finite exchangeability does not imply infinite exchangeability, for example, if the observed network consists of the first $m$ rows and $n$ columns of a larger adjacency matrix but of finite size. In that case, we say that the (finitely) exchangeable sequences of nodes are not infinitely extendible \citep{konstantopoulos2019extendibility, mai2020infinite}. The finitely exchangeable case for networks has been notably studied by \cite{lauritzen2018random}, but is out of scope of our paper. From here, the concept of exchangeability will always refer to infinite exchangeability, unless explicitely specified.

Thus, the exchangeability property of our infinite adjacency matrices is called row-column exchangeability. Let $\mathbb{S}_\infty$ be the group of finite permutations over $\mathbb{N}$. An infinite matrix $Y$ is said to be row-column exchangeable (RCE) if for any couple $\Phi = (\sigma_1,\sigma_2) \in \mathbb{S}_\infty^2$, 
\begin{equation*}
    \Phi Y \overset{\mathcal{D}}{=} Y,
\end{equation*}
where $\Phi Y := (Y_{\sigma_1(i)\sigma_2(j)})_{i \ge 1, j \ge 1}$. 

\paragraph{U-statistics} $U$-statistics form a large class of statistics with interesting properties for many purposes such as estimation and hypothesis testing. We are interested in using them to analyze RCE networks. 

Given a sequence of random variables  $(Y_1, Y_2, ..., Y_n)$ numbered with a unique index, a $U$-statistic is defined as the following average
\begin{equation}
    U^h_n = \binom{n}{k}^{-1} \sum_{1 \le i_1 < i_2 < ... < i_k \le n} h(Y_{i_1}, Y_{i_2}, ..., Y_{i_k}),
    \label{eq:ustatuni}
\end{equation}
where $h$ is a symmetric function of size $k$ referred to as the kernel.

The case where the $(Y_i)_{i \ge 1}$ are i.i.d. is well-studied. \cite{halmos1946theory} established the optimality of $U$-statistics as unbiased estimators and \cite{hoeffding1948class} derived a central limit theorem (CLT), which ensures their asymptotic normality provided $\mathbb{E}[h(Y_{i_1}, Y_{i_2}, ..., Y_{i_k})^2] < \infty$. For dependent cases, results exist for several dependency strucures, for example \cite{nandi1963properties, zhao1990normal} for finitely exchangeable variables, \cite{reitzner2013central} for Poisson point processes or \cite{duchemin2020concentration, duchemin2022three} for Markov chains.

In the infinitely exchangeable case, it is particularly convenient to view $h(Y_{i_1}, Y_{i_2}, ..., Y_{i_k})$ as an array of random variables $(X_{\boldsymbol i})_{\boldsymbol i}$ indexed by $k$-tuples ${\boldsymbol i} = (i_1, i_2, ..., i_k)$ where $X_{\boldsymbol i} = X_{(i_1, i_2, ..., i_k)}$. With this notation, it becomes clear that the $U$-statistic defined by~\eqref{eq:ustatuni} is the sum of the corresponding entries of the $k$-dimensional array $X$. But $(Y_1, Y_2, ...)$ being exchangeable implies that the array $X$ is jointly exchangeable, i.e. it is invariant by the action of joint permutations on each of its indices, for any sequence of $k$-tuples $(\boldsymbol i, \boldsymbol j, ...) \in \mathbb{N}^k$ and for any finite permutation $\sigma$ $\in \mathbb{S}_\infty$,
\begin{equation}
    (X_{\boldsymbol i}, X_{\boldsymbol j}, ...) \overset{\mathcal{D}}{=} (X_{(\sigma(i_1), \sigma(i_2), ..., \sigma(i_k))}, X_{(\sigma(j_1), \sigma(j_2), ..., \sigma(j_k))}, ...).
    \label{eq:joint_exch}
\end{equation}
\cite{eagleson1978limit} proved a CLT for sums jointly exchangeable arrays, which applies to $U$-statistics of exchangeable sequences. Many other asymptotic results for $U$-statistics of exchangeable sequences were derived afterwards, such as a Berry-Esseen bound \citep{van1984berry} and a law of iterated logarithm \citep{scott1985law}.

In relation with the existing literature, we add here the definition of separate exchangeability. $(X_{\boldsymbol i})_{\boldsymbol i}$ is said to be separately exchangeable if for any sequence of $k$-tuples $(\boldsymbol i, \boldsymbol j, ...) \in \mathbb{N}^k$ and for any permutations $\sigma_1, \sigma_2, ..., \sigma_k$ of $\mathbb{S}_\infty$,
\begin{equation}
    (X_{\boldsymbol i}, X_{\boldsymbol j}, ...) \overset{\mathcal{D}}{=} (X_{(\sigma_1(i_1), \sigma_2(i_2), ..., \sigma_k(i_k))}, X_{(\sigma_1(j_1), \sigma_2(j_2), ..., \sigma_k(j_k))}, ...).
    \label{eq:separate_exch}
\end{equation}

\paragraph{$U$-statistics for RCE matrices} Our contribution applies to $U$-statistics based on submatrices of size $2 \times 2$, that we call quadruplets, of an (infinite) RCE matrix $Y$
\begin{equation*}
    Y_{(i_1,i_2;j_1,j_2)} := \left(
    \begin{matrix}
    Y_{i_1j_1} & Y_{i_1j_2} \\
    Y_{i_2j_1} & Y_{i_2j_2}
    \end{matrix}
    \right).
\end{equation*}
Their kernels are real functions $h$ over quadruplets. To mimic the kernel symmetry in the unidimensional case, we assume that they present the following symmetry property: for any matrix $Y$, 
\begin{equation}
    h(Y_{(1,2;1,2)}) = h(Y_{(2,1;1,2)}) = h(Y_{(1,2;2,1)}) = h(Y_{(2,1;2,1)}).
    \label{eq:symmetry}
\end{equation}
This assumption can be made without loss of generality for $U$-statistics, since any quadruplet function $k$ can be made symmetric considering 
\begin{equation*}
h(Y_{(1,2;1,2)}) = \frac{1}{4} \left(k(Y_{(1,2;1,2)}) + k(Y_{(2,1;1,2)}) + k(Y_{(1,2;2,1)}) + k(Y_{(2,1;2,1)})\right)
\end{equation*}
and $\mathbb{E}[h(Y_{(1,2;1,2)})] = \mathbb{E}[k(Y_{(1,2;1,2)})]$.

Applied to an observed network represented by the first $m$ rows and $n$ columns of $Y$, a quadruplet $U$-statistic is then defined by 
\begin{equation}
    U^h_{m,n} = \binom{m}{2}^{-1} \binom{n}{2}^{-1} \sum_{\substack{1 \le i_1 < i_2 \le m\\1 \le j_1 < j_2 \le n}} h(Y_{(i_1,i_2;j_1,j_2)}),
    \label{eq:ustats}
\end{equation}
where $\binom{m}{2}$ is the number of $2$-combinations from $m$ elements. For clarity, we define the $4$-dimensional array $X$ using the following notation
\begin{equation*}
    X_{\{i_1, i_2; j_1, j_2\}} := h(Y_{(i_1, i_2; j_1, j_2)})
\end{equation*}
which means that a $U$-statistic is the mean of the first $\binom{m}{2} \times \binom{n}{2}$ entries of $X$. However, $U$-statistics of jointly exchangeable arrays deal with the mean of the first $\binom{m}{4}$ entries of the array $X$. Therefore, Theorem 4 of~\cite{eagleson1978limit} applies to $U$-statistics of square matrices, but not generally to the case of bipartite networks, where row and column nodes are distinct by nature. In particular, $U$-statistics for RCE matrices allow row and column indices to overlap and most importantly, $m$ to be different from $n$. Instead, the invariance structure of $X$ is a special case of $\pi$-exchangeability \citep{kallenberg1999multivariate}, where for any two permutations $\sigma_1$ and $\sigma_2$ of $\mathbb{S}_\infty$, we have 
\begin{equation}
    (X_{\{i_1,i_2;j_1,j_2\}})_{\substack{\{i_1 , i_2\} \subset \mathbb{N}\\ \{j_1 , j_2\} \subset \mathbb{N}}} \overset{\mathcal{D}}{=} (X_{\{\sigma_1(i_1),\sigma_1(i_2);\sigma_2(j_1),\sigma_2(j_2)\}})_{\substack{\{i_1 , i_2\} \subset \mathbb{N}\\ \{j_1 , j_2\} \subset \mathbb{N}}}.
    \label{eq:pi_exch}
\end{equation}
Therefore, $X$ is not separately exchangeable because, compared to~\eqref{eq:separate_exch}, the same permutation $\sigma_1$ has to be applied on both row indices $i_1$ and $i_2$ and the same permutation $\sigma_2$ on both column indices $j_1$ and $j_2$.

Lemma 12 of \cite{kallenberg1999multivariate} establishes a strong law of large numbers for $\pi$-exchangeable variables, which applies to our $U$-statistics. Our aim is to establish a weak convergence theorem similar to Theorem 4 of~\cite{eagleson1978limit}. In the recent literature, two related results were obtained by \cite{austern2022limit} and \cite{davezies2021empirical}. \cite{austern2022limit} explains how a result from~\cite{lindenstrauss1999pointwise} can be translated to a strong law of large numbers for sums of exchangeable arrays and their Theorem 17 is analogous to Theorem 4 of~\cite{eagleson1978limit} but it is obtained using Stein's method. \cite{davezies2021empirical} adopted a functional point of view. Their Theorem 2.1 is a Donsker-type version of the same result on jointly exchangeable arrays and Theorem 3.4 is an extension to separately exchangeable arrays. Because $U$-statistics of jointly exchangeable arrays are not suited to bipartite networks and because our arrays are not separately exchangeable, these results do not apply to our $U$-statistics of RCE matrices, as defined in~\eqref{eq:ustats} where~\eqref{eq:pi_exch} is satisfied. 

To generalize these results to our case, our proof relies on the convergence of sums of backward martingales (Theorem 1 of~\citealp{eagleson1978limit}). We derive a CLT result in the case where the RCE matrix $Y$ is dissociated, i.e. any of its submatrices with disjoint indexing sets are independent. This CLT excludes the degenerate case (i.e. when the convergence rate to the limiting distribution is greater than $\sqrt{N}$, see Section~\ref{sec:disc}) through a clear assumption on the asymptotic variance. In the degenerate case, the convergence result of \cite{austern2022limit} does not lead to a CLT neither, and \cite{davezies2021empirical} proved a different convergence theorem. We offer a discussion on the degenerate case and its implications in Section~\ref{sec:disc}. Finally, we recall that the backward martingale approach also yields Kallenberg's strong law of large numbers~\citep{kallenberg1999multivariate}.

In the last part of this work, we will put a special emphasis on the statistical analysis of bipartite networks. We introduce two versions of a RCE matrix model, the Bipartite Expected Degree Distribution (BEDD) model and we explain how our theorems apply to both of them. We suggest a method to perform statistical inference on these models using quadruplet $U$-statistics through several examples and we discuss how one can extend it.

\paragraph{Outline} The weak convergence theorem in the general case and the CLT in the dissociated case are presented and proven in Section~\ref{sec:main}. We shed further light on the difference between the dissociated and the non-dissociated cases using the Aldous-Hoover representation theorem. Section~\ref{sec:app} illustrates our results with applications to statistical network analysis using a RCE model and several examples of inference tasks. 

\section{Main result \label{sec:main}}

\subsection{Asymptotic framework}

Our results apply in an asymptotic framework where the numbers of rows and columns of the submatrix used in the calculation of the $U$-statistic grow at the same rate, i.e. ${m}/({m+n}) \rightarrow c \in ]0,1[$. To simplify the proofs, we allow only one row or one column to be added to the submatrix. Now, we build a sequence of dimensions $(m_N, n_N)_{N \ge 1}$ for the submatrix satisfying these conditions.

\begin{definition}[Sequence of dimensions] \label{def:mn}
    Let $c$ be an irrational number such that $0<c<1$. For all $N \in \mathbb{N}$, we define $m_N = 2 + \lfloor c(N+1) \rfloor$ and $n_N = 2 + \lfloor (1-c)(N+1) \rfloor$, where $\lfloor \cdot \rfloor$ is the floor function.
    \label{def:sequences}
\end{definition}

\begin{proposition}
    $m_N$ and $n_N$ satisfy: 
    \begin{enumerate}
        \item $\frac{m_N}{m_N + n_N} \xrightarrow[N \rightarrow \infty]{} c$,
        \item $m_N + n_N = 4 + N$, for all $N \in \mathbb{N}$.
    \end{enumerate}
    \label{prop:sequences}
\end{proposition}

\begin{corollary}
    At each iteration $N \in \mathbb{N}^*$, one and only one of these two propositions is true:
    \begin{enumerate}
        \item $m_N = m_{N-1} + 1$ and $n_N = n_{N-1}$,
        \item $n_N = n_{N-1} + 1$ and $m_N = m_{N-1}$.
    \end{enumerate}
    \label{cor:partition}
\end{corollary}

Throughout the paper, only sequences satisfying Definition~\ref{def:sequences} are considered. Such sequences $m_N$ and $n_N$ satisfy the desired growth conditions (proof given in Appendix~\ref{app:mn}). We investigate the asymptotic behaviour of $U^h_{m,n}$ through $U^h_N$ defined as follows.
\begin{definition}
    With $(m_N, n_N)_{N \ge 1}$ introduced by Definition~\ref{def:mn} and $U^h_{m,n}$ specified by equation~\eqref{eq:ustats}, we set the sequence of $U$-statistics $(U^h_{N})_{N \ge 1}$ such that for all $n \in \mathbb{N}$, $U^h_{N} = U^h_{m_N, n_N}$.
    \label{def:uhn}
\end{definition}

\subsection{Theorems}

We establish the following results on the asymptotic behaviour of $U$-statistics over RCE matrices.

\begin{theorem}[Main theorem]
    Let $Y$ be a RCE matrix. Let $h$ be a quadruplet kernel such that $\mathbb{E}[h(Y_{(1,2;1,2)})^2] < \infty$. Let $(U^h_{N})_{N \ge 1}$ be the sequence of $U$-statistics associated with $h$ defined in Definition~\ref{def:uhn}. Let $\mathcal{F}_N = \sigma\big((U^h_{k,l}, k \ge m_N, l \ge n_N)\big)$ and $\mathcal{F}_\infty := \bigcap_{N = 1}^{\infty} \mathcal{F}_N$. Set $U^h_\infty = \mathbb{E}[h(Y_{(1,2;1,2)})|\mathcal{F}_\infty]$ and 
    \begin{equation*}
    \begin{split}
        V =&~ \frac{4}{c}\Cov\big(h(Y_{(1,2;1,2)}),h(Y_{(1,3;3,4)}) \big|\mathcal{F}_\infty \big) + \frac{4}{1-c}\Cov\big(h(Y_{(1,2;1,2)}),h(Y_{(3,4;1,3)})\big|\mathcal{F}_\infty \big).
    \end{split}
    \end{equation*}
    If $\mathbb{P}(V > 0) > 0$, then
    \begin{equation*}
        \sqrt{N}(U^h_{N}-U^h_\infty) \xrightarrow[N \rightarrow \infty]{\mathcal{D}} W,
    \end{equation*}
    where $W$ is a random variable with characteristic function $\phi(t) = \mathbb{E}[\exp(-\frac{1}{2} t^2 V)]$.
    \label{th:mixture_theorem}
\end{theorem}

Theorem~\ref{th:mixture_theorem} states that the limiting distribution of $\sqrt{N}(U^h_N - U^h_\infty)$ is a mixture of Gaussians. $V$ consists of two terms corresponding to the covariance of the kernel taken on two quadruplets sharing one row or one column, conditional on $\mathcal{F}_\infty$. The condition $\mathbb{P}(V > 0) > 0$ is used to avoid the case $V = 0$ almost surely, which is a degenerate case discussed in Section~\ref{sec:disc}. Since $V$ is not constant in general, the limit distribution is an infinite mixture of Gaussians. This expression is analogous to $\eta^2$ in Theorem 4 of~\citet{eagleson1978limit} for jointly exchangeable arrays and the covariance kernel in Theorems 2.1 and 3.4 of~\citet{davezies2021empirical} for jointly and separately exchangeable arrays. We see that if $V$ is constant, then the asymptotic distribution is a simple Gaussian. One may observe that $U^h_N$ and $\sum_{\phi \in \mathbb{S}^2_N} f(\phi Y)$ studied in Corollary 19 of \citet{austern2022limit} are related as $U^h_N = (m_N!)^{-2} \sum_{\phi \in \mathbb{S}^2_{m_N}} f(\phi Y)$ when $m_N = n_N$ and $f(Y) = h(Y_{(1,2;1,2)})$. Still the convergence rates given in Theorem~\ref{th:mixture_theorem} and Corollary 19 of \citet{austern2022limit} (the proof of which is not given in the paper) are inconsistent. From what we understand, Corollary 19 actually corresponds to a degenerate case ($V=0$) in Theorem~\ref{th:mixture_theorem}. Next we identify a class of models where the limiting distribution of $\sqrt{N}(U^h_N - U^h_\infty)$ is a simple Gaussian. 

\begin{definition}
  $Y$ is a dissociated matrix if and only if $(Y_{ij})_{1 \le i \le m, 1 \le j \le n}$ is independent of $(Y_{ij})_{i > m, j > n}$, for all $m$ and $n$.
\end{definition}

In other words, $Y$ is dissociated if submatrices that are not sharing any row or column are independent. Now we claim the following extension to Theorem~\ref{th:mixture_theorem} for dissociated RCE matrices.

\begin{theorem}
    In addition to the hypotheses of Theorem~\ref{th:mixture_theorem}, if $Y$ is dissociated, then $U^h_\infty$ and $V$ are constant and \begin{equation*}
        \sqrt{N}(U^h_{N}-U^h_\infty) \xrightarrow[N \rightarrow \infty]{\mathcal{D}} \mathcal{N}(0, V),
    \end{equation*}
    More precisely,
    \begin{enumerate}
        \item $U^h_\infty = \mathbb{E}[h(Y_{(1,2;1,2)})]$,
        \item $V = \frac{4}{c}\Cov\big(h(Y_{(1,2;1,2)}),h(Y_{(1,3;3,4)})\big) + \frac{4}{1-c}\Cov\big(h(Y_{(1,2;1,2)}),h(Y_{(3,4;1,3)})\big)$.
    \end{enumerate}
    \label{th:gaussian_theorem}
\end{theorem}

This result can be more directly exploited for statistical applications, as the limiting distribution is much more simple. Another important result is the joint asymptotic normality of $U$-statistics, which holds as long as Theorem~\ref{th:gaussian_theorem} applies to each kernel separately and they are linearly independent. 

\begin{theorem}
\label{th:multivariate}
  Let $Y$ be a RCE dissociated matrix. Let $(h_1, h_2, ..., h_n)$ be a vector of quadruplet kernels such that 
  \begin{enumerate}
      \item Theorem~\ref{th:gaussian_theorem} applies for each kernel, i.e. $\mathbb{E}[h_k(Y_{(1,2;1,2)})^2] < \infty$ and $U^{h_k}_\infty$ and $V^{h_k}$ are as defined in Theorem~\ref{th:gaussian_theorem} for each kernel $h_k$, $1 \le k \le n$,
      \item for $t \in \mathbb{R}^n$, $\sum_{k=1}^n t_k h_k \equiv 0$ if and only if $t = (0, ..., 0)$.
  \end{enumerate} 
  Then
  \begin{equation*}
      \sqrt{N}\left(\begin{pmatrix} U^{h_1}_N  \\ U^{h_2}_N \\ ... \\ U^{h_n}_N 
    \end{pmatrix}-\begin{pmatrix} U^{h_1}_\infty  \\ U^{h_2}_\infty \\ ... \\ U^{h_n}_\infty 
    \end{pmatrix}\right) \xrightarrow[N \rightarrow \infty]{\mathcal{D}} \mathcal{N}(0, \Sigma),
  \end{equation*}
  with
  \begin{equation*}
      \Sigma = \left(C^{h_i, h_j} \right)_{1 \le i,j \le n},
  \end{equation*}
  where $C^{h_k,h_\ell} = \lim_{N \rightarrow +\infty} N \Cov(U^{h_k}_N, U^{h_\ell}_N)$ for all $1 \le k,\ell \le n$ (and $C^{h_k,h_k} = V^{h_k}$).
\end{theorem}

This theorem allows us to obtain the asymptotic normality of linear combinations of $U$-statistics and more interestingly, the asymptotic normality of differentiable functions of $U$-statistics (see Section~\ref{subsec:est}).

\begin{remark*}
  Lemma 12 of \cite{kallenberg1999multivariate} already provided a strong law of large numbers for $\pi$-exchangeable arrays, of which quadruplet kernels are a subcase. The following theorem rephrases Kallenberg's law of large numbers for quadruplet $U$-statistics and gives an additional precision in the dissociated case for which we provide an alternative proof, as it is a natural consequence of our proof of Theorems~\ref{th:mixture_theorem} and~\ref{th:gaussian_theorem}.
\end{remark*}

\begin{theorem}
    Let $Y$ be a RCE matrix. Let $h$ be a quadruplet kernel. Let $(U^h_{N})_{N \ge 1}$ the sequence of $U$-statistics associated with $h$ defined in Definition~\ref{def:uhn}. Let $\mathcal{F}_N = \sigma\big((U^h_{k,l}, k \ge m_N, l \ge n_N)\big)$ and $\mathcal{F}_\infty := \bigcap_{N = 1}^{\infty} \mathcal{F}_N$. We have
    \begin{equation*}
        U^h_N \xrightarrow[N \rightarrow \infty]{a.s.} \mathbb{E}[h(Y_{(1,2;1,2)})|\mathcal{F}_\infty].
    \end{equation*}
    Furthermore, if $Y$ is dissociated, then $\mathbb{E}[h(Y_{(1,2;1,2)})|\mathcal{F}_\infty] = \mathbb{E}[h(Y_{(1,2;1,2)})]$.
    \label{th:kallenberg_slln}
\end{theorem}

\subsection{The Aldous-Hoover theorem} 
\label{subsec:aldous_hoover}

We shall explain Theorems~\ref{th:mixture_theorem} and~\ref{th:gaussian_theorem} in the light of the Aldous-Hoover representation theorem. Theorem 1.4 of~\cite{aldous1981representations} states that for any RCE matrix $Y$, there exists a real function $f$ such that if we denote $Y^*_{ij} = f(\alpha, \xi_i, \eta_j, \zeta_{ij})$, for $1 \le i,j < \infty$, where the $\alpha$, $\xi_i$, $\eta_j$ and $\zeta_{ij}$ are i.i.d. random variables with uniform distribution over $[0,1]$, then
\begin{equation}
    Y \overset{\mathcal{D}}{=} Y^*.
    \label{eq:aldous}
\end{equation}

It is possible to identify the role of each of the random variables involved in the representation theorem. We notice that each $Y_{ij}$ is determined by $\alpha, \xi_i, \eta_j$ and $\zeta_{ij}$. $\zeta_{ij}$ is entry-specific while $\xi_i$ is shared by all the entries involving the row $i$ and $\eta_j$ by the ones involving the column $j$. Therefore, the $\xi_i$ and $\eta_j$ represent the contribution of each individual of type 1 and type 2 of the network, i.e. each row and column of the matrix. These contributions are i.i.d., which makes the network exchangeable. Finally, $\alpha$ is global to the whole network and shared by all entries.

Proposition 3.3 of~\cite{aldous1981representations} states that if $Y$ is dissociated, then $Y^*$ can be written without $\alpha$, i.e. it is of the form $Y^*_{ij} = f(\xi_i, \eta_j, \zeta_{ij})$, for $1 \le i,j < \infty$. In this case, because the $\xi_i$, $\eta_j$ and $\zeta_{ij}$ are i.i.d., averaging with the $U$-statistic over an increasing number of nodes nullifies the contribution of each individual interaction ($\zeta_{ij}$) and node ($\xi_i$ and $\eta_j$). In the general case, i.e. when $Y$ is not dissociated, then conditionally on $\alpha$, $Y$ is dissociated. It is easy to see that the mixture of Gaussians from Theorem~\ref{th:mixture_theorem} results from this conditioning.

In practice, dissociated exchangeable random graph models are widely spread. Notably, a RCE model is dissociated if and only if it can be written as a $W$-graph (or graphon), i.e. it is defined by a distribution $\mathcal{W}$ depending on two parameters in $[0,1]$ such that for $1 \le i,j < \infty$:
\begin{equation*}
  \begin{split}
    \xi_i, \eta_j &\overset{i.i.d.}{\sim} \mathcal{U}[0,1] \\
    Y_{ij}~|~\xi_i, \eta_j &\sim \mathcal{W}(\xi_i, \eta_j),
  \end{split}
  \label{eq:graphon}
\end{equation*}
see \cite{diaconis2008graph} for binary bipartite graphs, \cite{lovasz2010limits} for an extension to weighted graphs but in a unipartite setup. In this definition, it is easy to recognize the variables from the representation theorem of Aldous-Hoover. We simply identify the $\xi_i$ and $\eta_j$, then it suffices to take $\phi^{-1}_{\xi_i, \eta_j}$ the inverse distribution function of $\mathcal{W}(\xi_i, \eta_j)$ to see that defining the dissociated RCE matrix $Y^*$ such that $Y^*_{ij} = f(\xi_i, \eta_j, \zeta_{ij}) := \phi^{-1}_{\xi_i, \eta_j}(\zeta_{ij})$ fulfills $Y^* \overset{\mathcal{D}}{=} Y$.

\subsection{Proof of Theorem~\ref{th:mixture_theorem}}

To prove Theorem~\ref{th:mixture_theorem}, we adapt the proof of \cite{eagleson1978limit} establishing the asymptotic normality of sums of backward martingale differences. The definition of a backward martingale is reminded in Appendix~\ref{app:backward}.

\begin{theorem}[\citealp{eagleson1978limit}]
    Let $(M_n, \mathcal{F}_n)_{n \ge 1}$ be a square-integrable reverse martingale, $V$ a $\mathcal{F}$-measurable, a.s. finite, positive random variable. Denote $M_\infty := \mathbb{E}[M_1 | \mathcal{F}_\infty]$ where $\mathcal{F}_\infty := \bigcap_{n = 1}^{\infty} \mathcal{F}_n$. Set $Z_{nk} := \sqrt{n}(M_k - M_{k+1})$. If:
    \begin{enumerate}
        \item $\sum_{k=n}^\infty \mathbb{E}[Z_{nk}^2 | \mathcal{F}_{k+1}] \xrightarrow[n \rightarrow \infty]{\mathbb{P}} V$ \text{(asymptotic variance)},
        \item for all $\epsilon > 0$, $\sum_{k=n}^\infty \mathbb{E}[Z_{nk}^2 \mathds{1}_{\{ |Z_{nk}| > \epsilon\}} | \mathcal{F}_{k+1}] \xrightarrow[n \rightarrow \infty]{\mathbb{P}} 0$ \text{(conditional Lindeberg condition)},
    \end{enumerate}
    then $\sum_{k=n}^\infty Z_{nk} = \sqrt{n}(M_n - M_\infty) \xrightarrow[n \rightarrow \infty]{\mathcal{D}} W$, where $W$ is a random variable with characteristic function $\phi(t) = \mathbb{E}[\exp(-\frac{1}{2} t^2 V)]$.
    \label{th:eagleson}
\end{theorem}

\begin{proof}[Proof of Theorem~\ref{th:mixture_theorem}]
    The three steps to apply Theorem~\ref{th:eagleson} to $(M_N)_{N \geq 1} = (U^h_N)_{N \geq 1}$ are to show that it is a backward martingale for a well chosen filtration and that it fulfills conditions $1$ and $2$. The expression of $V$ is made explicit along the way. More precisely, 
    \begin{enumerate}
        \item first, defining $\mathcal{F}_N = \sigma\big((U^h_{k,l}, k \ge m_N, l \ge n_N)\big)$, Proposition~\ref{prop:martingale} states that $(U^h_N, \mathcal{F}_N)_{N \ge 1}$ is indeed a square-integrable reverse martingale ;
        \item then, Proposition~\ref{prop:asymptotic_variance} implies that $\sum_{K=N}^\infty \mathbb{E}[Z_{NK}^2 | \mathcal{F}_{K+1}]$, where $Z_{NK} := \sqrt{N} (U_K-U_{K+1})$, does converge to a random variable $V$ with the desired expression ;
        \item finally, the conditional Lindeberg condition is ensured by Proposition~\ref{prop:lindeberg_condition}, since  from it, we deduce that for all $\epsilon > 0$, $\sum_{K=N}^\infty \mathbb{E}[Z_{NK}^2 \mathds{1}_{\{ |Z_{NK}| > \epsilon\}} | \mathcal{F}_{K+1}] \xrightarrow[N \rightarrow \infty]{\mathbb{P}} 0$.
    \end{enumerate}
    Hence, if $V$ is positive, Theorem~\ref{th:eagleson} can be applied to $U^h_N$ and we obtain that $\sqrt{N}(U^h_N - U^h_\infty) \xrightarrow[N \rightarrow \infty]{\mathcal{D}} W$, where $W$ is a random variable with characteristic function $\phi(t) = \mathbb{E}[\exp(-\frac{1}{2} t^2 V)]$. The proofs of Propositions~\ref{prop:martingale}, ~\ref{prop:asymptotic_variance} and ~\ref{prop:lindeberg_condition}  are provided in Appendices~\ref{app:square},~\ref{app:variance} and~\ref{app:lindeberg} respectively.
\end{proof}

\subsection{Proof of Theorem~\ref{th:gaussian_theorem}}

The proof of Theorem~\ref{th:gaussian_theorem} relies on a Hewitt-Savage type zero-one law for events that are permutable in our row-column setup. Therefore, it is useful to define first what a row-column permutable event is. We remind the Aldous-Hoover representation theorem for dissociated RCE matrices as stated earlier: if $Y$ is a dissociated RCE matrix, then its distribution can be written with $(\xi_i)_{1 \le i < m_N}$, $(\eta_j)_{1 \le j < n_N}$ and $(\zeta_{ij})_{1 \le i \le m_N, 1 \le j \le n_N}$ arrays of i.i.d. random variables.

Then let us consider such arrays of i.i.d. random variables $(\xi_i)_{1 \le i < m_N}$, $(\eta_j)_{1 \le j < n_N}$ and $(\zeta_{ij})_{1 \le i \le m_N, 1 \le j \le n_N}$. If we were to consider events depending only on them, there is no loss of generality in using the product probability space $(\Omega_N, \mathcal{A}_N, \mathbb{P}_N)$, where
\begin{equation*}
    \begin{split}
        \Omega_N &= \left\{ (\omega^\xi, \omega^\eta, \omega^\zeta) : \omega^\xi \in \mathbb{R}^{m_N}, \omega^\eta \in \mathbb{R}^{n_N}, \omega^\zeta \in \mathbb{R}^{m_N n_N} \right\} = \mathbb{R}^{m_N + n_N + m_N n_N}, \\
        \mathcal{A}_N &= \mathcal{B}(\mathbb{R})^{m_N + n_N + m_N n_N}, \\
        \mathbb{P}_N &= \mu^{m_N + n_N + m_N n_N}.
    \end{split}
\end{equation*}
We then define the action of a row-column permutation on an element of $\Omega_N$.
\begin{definition}
  Let $\Phi = (\sigma_1, \sigma_2) \in \mathbb{S}_{m_N} \times \mathbb{S}_{n_N}$. The action of $\Phi$ on $\omega \in \Omega_N$ is defined by
  \begin{equation*}
      \Phi \omega = \big(\sigma_1 \omega^\xi, \sigma_2 \omega^\eta, (\sigma_1,\sigma_2) \omega^\zeta\big)
  \end{equation*}
  where $\sigma_1 \omega^\xi = (\omega^\xi_{\sigma_1(i)})_{1 \le i < m_N}$, $\sigma_2 \omega^\eta = (\omega^\eta_{\sigma_2(j)})_{1 \le j < n_N}$ and $(\sigma_1,\sigma_2) \omega^\zeta = (\omega^\zeta_{\sigma_1(i) \sigma_2(j)})_{1 \le i < m_N, 1 \le j < n_N}$
\end{definition}

\begin{definition}
  Let $A \in \mathcal{A}_N$. $A$ is invariant by the action of $\mathbb{S}_{m_N} \times \mathbb{S}_{n_N}$ if and only if for all $\Phi \in \mathbb{S}_{m_N} \times \mathbb{S}_{n_N}$, $\Phi^{-1} A = A$, i.e.
  \begin{equation*}
      \left\{ \omega : \Phi \omega \in A \right\} = \left\{ \omega : \omega \in A \right\}.
  \end{equation*}
\end{definition}

\begin{notation*}
  In this section, we denote by $\mathcal{E}_N$ the collection of events of $\mathcal{A}_N$ that are invariant by row-column permutations of size $m_N \times n_N$, i.e. $\Phi \in \mathbb{S}_{m_N} \times \mathbb{S}_{n_N}$. We denote $\mathcal{E}_\infty := \bigcap_{n = 1}^{\infty} \mathcal{E}_N$, which is the collection of events that are invariant by permutations of size $m_N \times n_N$, for all $N$.
\end{notation*}

The following theorem is an extension of the Hewitt-Savage zero-one law to the row-column setup. 
\begin{theorem}
  For all $A \in \mathcal{E}_\infty$, $\mathbb{P}(A) = 0$ or $\mathbb{P}(A) = 1$.
  \label{th:hewitt_savage}
\end{theorem}

The proof of Theorem~\ref{th:hewitt_savage} is given in Appendix~\ref{app:normality}. Now we use this result to derive Theorem~\ref{th:gaussian_theorem} from Theorem~\ref{th:mixture_theorem}.

\begin{proof}[Proof of Theorem~\ref{th:gaussian_theorem}]
  In this proof, we specify the matrices over which the $U$-statistics are taken, i.e. for a RCE matrix $Y$, we denote $U^h_{k,l}(Y)$ instead of $U^h_{k,l}$ the $U$-statistic of size $k \times l$ with kernel $h$ taken on $Y$, given by formula~\eqref{eq:ustats}, and analogously $U^h_N(Y) := U^h_{m_N, n_N}(Y)$ and $U^h_\infty(Y) := \mathbb{E}[h(Y_{(1,2;1,2)})|\mathcal{F}_\infty(Y)]$. We denote also $\mathcal{F}_N(Y) := \sigma\big((U^h_{k,l}(Y), k \ge m_N, l \ge n_N)\big)$ which are sets of events depending on $Y$, and $\mathcal{F}_\infty(Y) := \bigcap_{n = 1}^{\infty} \mathcal{F}_N(Y)$.

  Since $Y$ is RCE and dissociated, Proposition 3.3 of~\cite{aldous1981representations} states the existence of a real function $f$ such that for $1 \le i,j < \infty$, $Y^*_{ij} = f(\xi_i, \eta_j, \zeta_{ij})$ and $Y^* \overset{\mathcal{D}}{=} Y$, where  $\xi_i$, $\eta_j$ and $\zeta_{ij}$, for $1 \le i,j < \infty$ are i.i.d. random variables with uniform distribution on $[0,1]$. Therefore we can consider such function $f$ and  these random variables, the product spaces $(\Omega_N, \mathcal{A}_N, \mathbb{P}_N)$ and the sets $\mathcal{E}_N$ of invariant events defined earlier.
  
  But $\mathcal{F}_N(Y^*) = \sigma\big((U^h_{k,l}(Y^*), k \ge m_N, l \ge n_N)\big) \subset \sigma(U^h_N(Y^*), \xi_i, \eta_j, \zeta_{ij}, i > m_N, j > n_N)$, so for all $N$, $\mathcal{F}_N(Y^*) \subset \mathcal{E}_N$. It follows that $\mathcal{F}_\infty(Y^*) \subset \mathcal{E}_\infty$, so $U_\infty(Y^*)$ is $\mathcal{E}_\infty$-measurable. Theorem~\ref{th:hewitt_savage} states that all the events in $\mathcal{E}_\infty$ happen with probability 0 or 1, so it ensures that $U^h_\infty(Y^*) = \mathbb{E}[h(Y^*_{\{1,2;1,2\}})|\mathcal{F}_\infty(Y^*)] = \mathbb{E}[h(Y^*_{\{1,2;1,2\}})]$ is constant. Moreover, since the distribution of $U^h_{N}(Y)$ is the same as this of $U^h_{N}(Y^*)$, we can conclude that $U^h_\infty(Y) = \mathbb{E}[h(Y_{(1,2;1,2)})|\mathcal{F}_\infty(Y)] = \mathbb{E}[h(Y_{(1,2;1,2)})]$. 
  
  Likewise, we deduce that $\mathbb{E}[h(Y_{(1,2;1,2)})h(Y_{(1,3;3,4)})|\mathcal{F}_\infty(Y)] = \mathbb{E}[h(Y_{(1,2;1,2)})h(Y_{(1,3;3,4)})]$ and $\mathbb{E}[h(Y_{(1,2;1,2)})h(Y_{(3,4;1,3)})|\mathcal{F}_\infty(Y)] = \mathbb{E}[h(Y_{(1,2;1,2)})h(Y_{(3,4;1,3)})]$ which gives the desired result for $V$. Thus we conclude that $W$ of Theorem~\ref{th:mixture_theorem} follows a Gaussian distribution of variance $V$.
  
\end{proof}

\subsection{Proof of Theorem~\ref{th:multivariate}}

The proof of Theorem~\ref{th:multivariate} relies on the Cram\'er-Wold theorem (see Theorem 29.4 of~\citealp{billingsley1995probability}). If is enough to show that any linear combination of $U$-statistics converges to the corresponding linear combination of their limits.

\begin{proof}[Proof of Theorem~\ref{th:multivariate}]
  Let $(Z^{h_k})_{1 \le k \le n}$ be a vector of random variables following a centered multivariate Gaussian distribution with covariance matrix $\Sigma$ defined in the theorem. Then $Z^{h_k} \sim \mathcal{N}(0, V^{h_k})$ for all $1 \le k \le n$ and $\Cov(Z^{h_i}, Z^{h_j}) = C^{h_i, h_j}$ for all $1 \le i \le n$ and $1 \le j \le n$.
  
  For some $t = (t_1, t_2, ..., t_n) \in \mathbb{R}^n$, we set $h_t := t_1 h_1 + t_2 h_2 + ... + t_n h_n$. First, assume that $t \neq (0,...,0)$. Then by hypothesis, $h_t \not\equiv 0$, therefore $\sum_{k=1}^n t_k U^{h_k}_N = U^{h_t}_N$ is a $U$-statistic with quadruplet kernel $h_t$. Using Cauchy-Schwarz inequality and the fact that $\mathbb{E}[h_k(Y_{(1,2;1,2)})^2] < \infty$ for all $1 \le k \le n$, we have furthermore 
  \begin{equation*}
  \begin{split}
      \mathbb{E}[h_t(Y_{(1,2;1,2)})^2] = &~ \sum_{k=1}^n t_k^2 \mathbb{E}[h_k(Y_{(1,2;1,2)})^2] + 2 \sum_{1\le k\neq\ell \le n} t_k t_\ell \mathbb{E}[h_k(Y_{(1,2;1,2)})h_\ell(Y_{(1,2;1,2)})], \\
      \le&~ \sum_{k=1}^n t_k^2 \mathbb{E}[h_k(Y_{(1,2;1,2)})^2] + 2 \sum_{1\le k\neq\ell \le n} t_k t_\ell \sqrt{\mathbb{E}[h_k(Y_{(1,2;1,2)})^2]\mathbb{E}[h_\ell(Y_{(1,2;1,2)})^2]}, \\
      <&~ \infty.
  \end{split}
  \end{equation*}
  Therefore, Theorem~\ref{th:gaussian_theorem} also applies for $U^{h_t}_N$ and $\sqrt{N} (U^{h_t}_N - U^{h_t}_\infty) \xrightarrow[N \rightarrow \infty]{\mathcal{D}} \mathcal{N}(0, V^{h_t})$, where $U^{h_t}_\infty = \sum_{k=1}^n t_k U^{h_k}_\infty$ and $V^{h_t} = \sum_{k=1}^n \sum_{\ell=1}^n t_k t_\ell C^{h_k,h_\ell} = t^T \Sigma t$ with $C^{h_k,h_k} = V^{h_k} > 0$ since Theorem~\ref{th:gaussian_theorem} applies. This means that $\sqrt{N} (U^{h_t}_N - U^{h_t}_\infty) = \sqrt{N} \sum_{k=1}^n t_k (U^{h_k}_N - U^{h_k}_\infty) \xrightarrow[N \rightarrow \infty]{\mathcal{D}} \sum_{k=1}^n t_k Z^{h_k}$.
  
  Now assume that $t = (0,...,0)$. Then $h_t \equiv 0$ so $U^{h_t}_N = 0 = \sum_{k=1}^n t_k Z^{h_k}$. Therefore, $\sqrt{N} (U^{h_t}_N - U^{h_t}_\infty) = \sqrt{N} \sum_{k=1}^n t_k (U^{h_k}_N - U^{h_k}_\infty) \xrightarrow[N \rightarrow \infty]{\mathcal{D}} \sum_{k=1}^n t_k Z^{h_k}$ is still true. 
  
  We have proven that $\sqrt{N} (U^{h_t}_N - U^{h_t}_\infty) \xrightarrow[N \rightarrow \infty]{\mathcal{D}} \sum_{k=1}^n t_k Z^{h_k}$ for all $t \in \mathbb{R}^n$, so we can finally apply the Cram\'er-Wold theorem (Theorem 29.4 of~\citealp{billingsley1995probability}) which states that $\sqrt{N}\left(U^{h_k}_N - U^{h_k}_\infty\right)_{1\le k \le n}$ converges jointy in distribution to $(Z^{h_k})_{1 \le k \le n}$, which is a centered multivariate Gaussian with covariance matrix $\Sigma$, so this concludes the proof.
\end{proof}

\subsection{Proof of Theorem~\ref{th:kallenberg_slln}}

The proof for the first part of the theorem can be derived from the proof of Proposition~\ref{prop:martingale}, without needing the hypothesis $\mathbb{E}[h(Y_{(1,2;1,2)})^2] < \infty$. Indeed, it is enough to show that $(U^h_N, \mathcal{F}_N)_{N \ge 1}$ is a (not necessarily square-integrable) backward martingale and to apply Theorem~\ref{th:martingale_convergence}.

As for the dissociated case, $\mathbb{E}[h(Y_{(1,2;1,2)})|\mathcal{F}_\infty] = \mathbb{E}[h(Y_{(1,2;1,2)})]$ is ensured by the proof of Theorem~\ref{th:gaussian_theorem}.

\section{Applications \label{sec:app}}

In this section, we illustrate how the results from the previous section can be used for statistical inference for network data through different examples. First, we introduce the Bipartite Expected Degree Distribution (BEDD) models, a family of RCE models and we show how Theorems~\ref{th:mixture_theorem} and~\ref{th:gaussian_theorem} apply. Then, we detail three examples to show how one might exploit the $U$-statistics properties to analyze networks. In the first example, we use different quadruplet kernels to estimate the row heterogeneity of a network, with the help of the delta-method. Next, we extend this example to build a statistical test to compare the row heterogeneity of two networks. In the last example, we use $U$-statistics to estimate the frequency of network motifs.

\subsection{The BEDD model \label{subsec:bedd}}

\paragraph{General model} As examples of models for RCE matrices, we consider the family of the BEDD models, which are weighted, bipartite and exchangeable extensions of the Expected Degree Sequence model~\citep{chung2002average, ouadah2022motif}. For binary graphs, the degree of a node is the number of edges that stem from it. For weighted graphs, the equivalent notion is the sum of the weights of these edges. It is sometimes called node strength~\citep{barrat2004architecture}, but we will simply refer to it as node weight. A BEDD model draws the node weights from two distributions, characterised by c\`adl\`ag, non-decreasing and bounded real functions $f$ and $g$ of $[0,1] \rightarrow \mathbb{R}_+$. The expected edge weights $Y_{ij}$ are then proportional to the expected weights of the involved nodes. A BEDD model can be written in a hierarchical form
\begin{equation}
\begin{split}
    \xi_i, \eta_j &\overset{iid}{\sim} \mathcal{U}[0,1] \\ 
    Y_{ij}~|~\xi_i, \eta_j &\sim \mathcal{L}(\lambda f(\xi_i) g(\eta_j)).
\end{split}
\label{eq:wbedd}
\end{equation} 
where given any real number $\mu \ge 0$, we denote by $\mathcal{L}(\mu)$ a family of probability distributions with expectation $\mu$ and finite variance, $\lambda$ is a positive real number and $f$ and $g$ are normalized by the condition $\int f = \int g = 1$. For a graph of size $m \times n$, conditionally to $\lambda$, $\xi_i$ and $\eta_j$, the expected weight of the $i$-th row is $n \lambda f(\xi_i)$ and the expected weight of the $j$-th column is $m \lambda g(\eta_j)$. Consequently, $\lambda$ is the mean intensity of the network. One can define different BEDD models by specifying different families of distributions $\mathcal{L}$. For a family $\mathcal{L}$, we refer to the $\mathcal{L}$-BEDD model (e.g. Poisson-BEDD or Bernoulli-BEDD). 

Furthermore, we define two versions of BEDD models: 
\begin{description}
    \item[Version 1] $\lambda$ is constant,
    \item[Version 2] $\lambda$ is a random variable.
\end{description}

By construction, the BEDD models are RCE, so Theorem~\ref{th:mixture_theorem} can be applied to matrices $Y$ generated by these two versions of BEDD models. Theorem~\ref{th:gaussian_theorem} only applies to Version 1, where the matrix is dissociated. Indeed, we see that in both models conditionally on $\lambda$, the expected mean of the interactions of any submatrix is $\lambda$. Therefore any 2 submatrices are independent if $\lambda$ is constant. We could also have noticed that $\lambda$ is determined by the $\alpha$ from the representation theorem of Aldous-Hoover, see equation~\eqref{eq:aldous}. As a remark, it is straightforward that unlike Version 2, Version 1 of BEDD models can be written as a $W$-graph model as in Formula~\eqref{eq:graphon}, setting $\mathcal{W}(\xi_i, \eta_j) := \mathcal{L}(\lambda f(\xi_i) g(\eta_j))$. 

Since Theorem~\ref{th:gaussian_theorem} only applies to Version 1, we will be only considering this version in the rest of the article.

\begin{definition}
    Given a family of distributions $\mathcal{L}(\mu)$ a family of probability distributions with expectation $\mu$ and finite variance, a $\mathcal{L}$-BEDD model is a semi-parametric model described by the triplet $\Theta = (\lambda, f, g)$ where
    \begin{enumerate}
        \item $\lambda \in \mathbb{R}$,
        \item $f$ and $g$ are real functions $f$ and $g$ of $[0,1] \rightarrow \mathbb{R}_+$ which are bounded, c\`adl\`ag, non-decreasing and normalized with $\int f = \int g = 1$.
    \end{enumerate}
    We call $\Theta$ the BEDD parameters and the matrix $Y$ generated by a $\mathcal{L}$-BEDD model with these parameters is written $Y \sim \mathcal{L}\text{-BEDD}(\Theta)$ and is described by~\eqref{eq:wbedd}, for all $(i,j) \in \mathbb{N}^2$.
    \label{def:wbedd}
\end{definition}

In this definition, the normalizing constraint on $\int f = \int g = 1$ ensures that $\mathbb{E}[Y_{ij}] = \lambda$ for all $(i,j) \in \mathbb{N}^2$. The boundedness of $f$ and $g$ ensures that the variables $f(\xi_i)$ and $g(\eta_j)$ are bounded and their moments exist. In the binary case (Bernoulli-BEDD), it also puts a condition on $\lambda$. Since $\mathbb{P}(Y_{ij} = 1 | \xi_i, \eta_j) = \lambda f(\xi_i) g(\eta_j)$, the condition $\lambda \le \lVert f \rVert^{-1}_\infty \lVert g \rVert^{-1}_\infty$ must hold. The non-decreasing and c\`adl\`ag conditions are similar to the condition of~\cite{bickel2009nonparametric} for their random graph model and ensures the identifiability of the model since otherwise, $f$ and $g$ could be replaced with any $f \circ \pi_1$ and $g \circ \pi_2$, where $\pi_1$ and $\pi_2$ are measure-preserving transformations.

\paragraph{Identifiability by a quadruplet} In addition to being a dissociated RCE model, the BEDD models are particularly well adapted to use of quadruplet $U$-statistics. In this paragraph, we show that for some choices of $\mathcal{L}$ such as the Poisson distribution, the model can be recovered by a single quadruplet. We use the following two theorems of which proofs are given in Appendix~\ref{app:ident}. The first theorem implies that the functions $f$ and $g$ of the BEDD models are characterised by their moments $F_k := \int f^k$ and $G_k := \int g^k$. 

\begin{theorem}
    Let $\Theta = (\lambda, f, g)$ be BEDD parameters and $Y \sim \mathcal{L}\text{-BEDD}(\Theta)$ for some family of distributions $\mathcal{L}$. The distribution of $Y$ is uniquely determined by $\lambda$, $(F_k)_{k \ge 1}$ and $(G_k)_{k \ge 1}$, where $F_k := \int f^k$ and $G_k := \int g^k$ for all $k \ge 1$.
    \label{th:id_bedd_param}
\end{theorem}

Now we specify an assumption on the family of distributions $\mathcal{L}(\mu)$, under which a quadruplet identifies the parameters $\Theta = (\lambda, f, g)$ of a BEDD model. 
\begin{assumption}
    For the family of distributions $\mathcal{L}(\mu)$, there exists a sequence of functions $(\Psi_k)_{k \ge 1}$ such that if a random variable $X \sim \mathcal{L}(\mu)$, then for every $k \ge 1$, 
    \begin{equation*}
        \mathbb{E}[\Psi_k(X)] = \mu^k.
    \end{equation*}
    \label{ass:distfunc}
\end{assumption}

This assumption holds for many usual distributions families such as the Poisson or the Binomial distributions. As an example, if $X$ follows a Poisson distribution, we have $\Psi_k(X) = X(X-1)...(X-k+1)$. This assumption does not hold for the Bernoulli distribution, as if $X \sim \text{Bernoulli}(\mu)$, for any function $\varphi$, $\mathbb{E}[\varphi(X)] = \mu \varphi(1)$. This assumption is a sufficient condition to be able to recover the BEDD parameters from the joint distribution of a quadruplet.

\begin{theorem}
    If Assumption~\ref{ass:distfunc} holds for the family of distributions $\mathcal{L}(\mu)$, then for all $k \in \mathbb{N}$, $F_k$ and $G_k$ are uniquely determined by the joint distribution of a quadruplet.
    \label{th:id_quadruplet}
\end{theorem}

This theorem suggests that all the BEDD information is contained in the distribution of a quadruplet, therefore it is possible to extract any information only with quadruplet kernels. Quadruplet $U$-statistics are then especially of interest.

\subsection{Heterogeneity in the row weights of a network \label{subsec:est}}

In this first example, we are interested in evaluating the heterogeneity of the row weights of a network. In the BEDD models, conditional to the latent variables $(\xi_i)_{1 \le i \le m}$ following a uniform distribution on $[0,1]$, the expected weights of the row nodes are given by the $(f(\xi_i))_{1 \le i \le m}$, as seen in Section~\ref{subsec:bedd}. Therefore, the heterogeneity of the rows, i.e. the variance of the expected weight of a row node can be quantified by $F_2 := \int_0^1 f^2(u) du$. In the example of an interaction network, if $f$ is constant, i.e. $f \equiv 1$ and $F_2 = 1$, then the expected weight is constant for all rows and the row weight distribution is homogeneous, with all the row individuals having around the same number of interactions. Besides, the higher $F_2$ is, the more this distribution is unbalanced. In ecology, a large value of $F_2$ indicates a strong distinction between generalist (with high degree) and specialists (with low degree) species. 

$F_2$ can be estimated using $\widehat{\theta}_N := U^{h_1}_N/U^{h_2}_N$ where $U^{h_1}_N$ and $U^{h_2}_N$ are the $U$-statistics based on the quadruplet kernels $h_1$ and $h_2$ defined as
\begin{equation*}
    h_1(Y_{(i_1,i_2;j_1,j_2)}) = \frac{1}{2}(Y_{i_1j_1}Y_{i_1j_2} + Y_{i_2j_1}Y_{i_2j_2}),
\end{equation*}
and
\begin{equation*}
    h_2(Y_{(i_1,i_2;j_1,j_2)}) = \frac{1}{2}(Y_{i_1j_1}Y_{i_2j_2} + Y_{i_2j_1}Y_{i_1j_2}).
\end{equation*}

\begin{proposition}
    Let $\widehat{\theta}_N := U^{h_1}_N/U^{h_2}_N$ be defined as above. Then 
\begin{equation}
    \sqrt{\frac{N}{V^\delta}} \bigg(\widehat{\theta}_N - F_2\bigg) \xrightarrow[N \rightarrow \infty]{\mathcal{D}} \mathcal{N}(0, 1),
    \label{eq:convergence_vdt}
\end{equation}
where 
\begin{equation}
    V^\delta = \frac{1}{c} \left(F_4 + F_2(4 F_2^2 - F_2 - 4 F_3) \right)
    \label{eq:variance_vdt}
\end{equation}
and for all $k > 0$, $F_k := \int f^k$ and $G_k := \int g^k$.
\label{prop:convergence_est}
\end{proposition}

This result comes from the composition of the asymptotic normality of two $U$-statistics. In the following, we show how to obtain equations~\eqref{eq:convergence_vdt} and~\eqref{eq:variance_vdt}. First, we see that $\mathbb{E}[h_1(Y_{(i_1,i_2;j_1,j_2)})] = \lambda^2 F_2$ and $\mathbb{E}[h_2(Y_{(i_1,i_2;j_1,j_2)})] = \lambda^2$. So applying Theorem~\ref{th:gaussian_theorem} successively to $U^{h_1}_N$ and $U^{h_2}_N$ gives the following results ($V^{h_1}$ and $V^{h_2}$ are derived in Lemmas~\ref{lem:expression_vh1} and~\ref{lem:expression_vh2}):
\begin{equation}
    \sqrt{\frac{N}{V^{h_1}}}(U^{h_1}_N - \lambda^2 F_2) \xrightarrow[N \rightarrow \infty]{\mathcal{D}} \mathcal{N}(0, 1),
    \label{eq:convergence_lambda2f2}
\end{equation}
and
\begin{equation}
    \sqrt{\frac{N}{V^{h_2}}}(U^{h_2}_N - \lambda^2) \xrightarrow[N \rightarrow \infty]{\mathcal{D}} \mathcal{N}(0, 1),
    \label{eq:convergence_lambda2}
\end{equation}
where 
\begin{equation}
    V^{h_1} = \frac{\lambda^4}{c} (F_4 - F_2^2) + \frac{4 \lambda^4 }{1-c}  F_2^2 (G_2 - 1),
    \label{eq:variance_lambda2f2}
\end{equation}
and
\begin{equation}
    V^{h_2} = \frac{4 \lambda^4}{c} (F_2 - 1) + \frac{4 \lambda^4}{1-c}  (G_2 - 1).
    \label{eq:variance_lambda2}
\end{equation}
To combine the results~\eqref{eq:convergence_lambda2f2} and~\eqref{eq:convergence_lambda2}, we apply Theorem~\ref{th:multivariate} to $(h_1, h_2)$. We find that
\begin{equation}
    \sqrt{N}\left(\begin{pmatrix} U^{h_1}_N  \\ U^{h_2}_N 
  \end{pmatrix}-\begin{pmatrix} \lambda^2 F_2  \\ \lambda^2
  \end{pmatrix}\right) \xrightarrow[N \rightarrow \infty]{\mathcal{D}} \mathcal{N}(0, \Sigma),
  \label{eq:joint_h1h2}
\end{equation}
with
\begin{equation*}
    \Sigma = \begin{pmatrix} V^{h_1} & C^{h_1,h_2} \\ C^{h_1,h_2} & V^{h_2}
  \end{pmatrix},
\end{equation*}
where $C^{h_1,h_2} = 2 \lambda^4 c^{-1} (F_3 - F_2) + 4 \lambda^4 (1-c)^{-1} F_2 (G_2 - 1)$. The derivation of $C^{h_1,h_2}$ is given by Lemma~\ref{lem:expression_ch1h2}. We suggest two methods to derive the weak convergence result for $\widehat{\theta}_N = U^{h_1}_N/U^{h_2}_N$.

\paragraph{First method: the delta-method}

The first-order Taylor expansion of $\phi(U^{h_1}_N, U^{h_2}_N) = U^{h_1}_N/U^{h_2}_N = \widehat{\theta}_N$ at the point $(U^{h_1}_N, U^{h_2}_N) = (\lambda^2 F_2, \lambda^2)$ is
\begin{equation*}
  \widehat{\theta}_N - F_2 = \nabla \phi(\lambda^2 F_2, \lambda^2)^T \left(\begin{pmatrix} U^{h_1}_N  \\ U^{h_2}_N 
  \end{pmatrix}-\begin{pmatrix} \lambda^2 F_2  \\ \lambda^2
  \end{pmatrix}\right) + o_P \left(\left\lVert \begin{pmatrix} U^{h_1}_N  \\ U^{h_2}_N 
  \end{pmatrix}-\begin{pmatrix} \lambda^2 F_2  \\ \lambda^2
  \end{pmatrix}\right\rVert \right)
\end{equation*}
where $\nabla \phi$ is the gradient of $\phi$ and $\nabla \phi(U^{h_1}_N,U^{h_2}_N)^T = (1/U^{h_2}_N,- U^{h_1}_N/(U^{h_2}_N)^{2})$.

As the result of~\eqref{eq:joint_h1h2}, the delta-method (see Chapter 3 of~\citealp{van2000asymptotic}) gives equation~\eqref{eq:convergence_vdt} with 
\begin{equation*}
    V^\delta = \nabla \phi(\lambda^2 F_2, \lambda^2) \Sigma \nabla \phi(\lambda^2 F_2, \lambda^2)^T = \frac{1}{\lambda^4} V^{h_1} - \frac{2 F_2}{\lambda^4} C^{h_1,h_2} + \frac{F_2^2}{\lambda^4} V^{h_2},
\end{equation*}
which is equation~\eqref{eq:variance_vdt}.

\paragraph{Second method}

The delta-method is a generic method that applies to all differentiable functions $\phi$. However, for our particular case, there is another way to find the same confidence intervals without using the delta-method. Let $t:= (1, -F_2)^T$. One could have noticed that~\eqref{eq:joint_h1h2} also implies
\begin{equation*}
    \sqrt{\frac{N}{V^t}} t^T\left(\begin{pmatrix} U^{h_1}_N  \\ U^{h_2}_N 
  \end{pmatrix}-\begin{pmatrix} \lambda^2 F_2  \\ \lambda^2
  \end{pmatrix}\right) \xrightarrow[N \rightarrow \infty]{\mathcal{D}} \mathcal{N}(0, 1),
\end{equation*}
where $V^t := t^T\Sigma t = V^{h_1} - 2 F_2 C^{h_1,h_2} + F_2^2 V^{h_2} = \lambda^{4} V^{\delta}$. This can be rewritten 
\begin{equation}
    \sqrt{\frac{N}{V^\delta}} \frac{U^{h_2}_N}{\lambda^2} \left(\widehat{\theta}_N - F_2 \right) \xrightarrow[N \rightarrow \infty]{\mathcal{D}} \mathcal{N}(0, 1).
    \label{eq:convergence_vt}
\end{equation}
Since $U^{h_2}_N/\lambda^{2} \xrightarrow[N \rightarrow \infty]{\mathbb{P}} 1$, then Slutsky's theorem yields equations~\eqref{eq:convergence_vdt} and~\eqref{eq:variance_vdt}.

\paragraph{Confidence intervals} In order to exploit equations~\eqref{eq:convergence_vdt} and~\eqref{eq:variance_vdt}, we have to estimate the remaining unknown quantities $\lambda^2$, $G_2$, $F_3$ and $F_4$. We use the kernels $h_3$, $h_4$, $h_5$ and $h_6$ listed in Table~\ref{tab:kernels_1}. Using equation~\eqref{eq:variance_vdt} and two additional $U$-statistics $U^{h_5}_N$ and $U^{h_6}_N$ based on the kernels $h_5$ and $h_6$ defined in Table~\ref{tab:kernels_1}, we build $\widehat{V}^{\delta}_N$ a consistent estimator for $V^{\delta}$, defined as
\begin{equation}
    \widehat{V}^{\delta}_N = \frac{1}{c} \left( \frac{U^{h_4}_N}{(U^{h_3}_{N})^2} + \frac{U^{h_1}_N}{U^{h_2}_N} \left( 4 \frac{(U^{h_1}_N)^2}{(U^{h_2}_N)^2} - \frac{U^{h_1}_N}{U^{h_2}_N} - 4 \frac{U^{h_6}_N}{U^{h_5}_N U^{h_3}_N} \right) \right).
    \label{eq:estimator_vd}
\end{equation}

Finally, it follows from Slutsky's theorem that
\begin{equation}
    \sqrt{\frac{N}{\widehat{V}^{\delta}_N}} \bigg(\widehat{\theta}_N - F_2\bigg) \xrightarrow[N \rightarrow \infty]{\mathcal{D}} \mathcal{N}(0, 1),
    \label{eq:convergence_vd}
\end{equation}
From this result, one can derive the following asymptotic confidence interval at level $\alpha \in ]0,1[$ for $F_2$ using the $(1-\alpha/2)$-th percentile $q_{1 - \alpha/2}$ of the standard normal distribution: for $N \ge 1$,
\begin{equation}
    CI_{F_2}^\delta(\alpha, N) = \left[\widehat{\theta}_N - q_{1 - \alpha/2} \sqrt{\frac{\widehat{V}^{\delta}_N}{N}}, \widehat{\theta}_N + q_{1 - \alpha/2} \sqrt{\frac{\widehat{V}^{\delta}_N}{N}} \right].
    \label{eq:ci_delta}
\end{equation}

\renewcommand{\arraystretch}{2}
\begin{table}[hbt!]
\centering
\begin{tabular}{ |p{0.5cm}|p{9cm}|p{3cm}|  }
 \hline
 $h$ & $h(Y_{(i_1,i_2;j_1,j_2)})$ & $\mathbb{E}[h(Y_{(i_1,i_2;j_1,j_2)})]$ \\[3pt]
 \hline
 $h_1$ & $\displaystyle{\frac12(Y_{i_1j_1}Y_{i_1j_2} + Y_{i_2j_1}Y_{i_2j_2})}$ & $\lambda^2 F_2$ \\[3pt]
 \hline
 $h_2$ & $\displaystyle{\frac12(Y_{i_1j_1} Y_{i_2j_2} + Y_{i_1j_2} Y_{i_2j_1})}$ & $\lambda^2$ \\[3pt]
 \hline
 $h_3$ & $\displaystyle{\frac12(Y_{i_1j_1}Y_{i_2j_1} + Y_{i_1j_2}Y_{i_2j_2})}$ & $\lambda^2 G_2$ \\[3pt]
 \hline
 $h_4$ & $\displaystyle{\frac12 \left((Y_{i_1j_1}^2 -Y_{i_1j_1})(Y_{i_1j_2}^2 -Y_{i_1j_2}) + (Y_{i_2j_1}^2 -Y_{i_2j_1})(Y_{i_2j_2}^2 -Y_{i_2j_2})\right)}$ & $\lambda^4 F_4 G_2^2$ \\[3pt]
 \hline
 $h_5$ & $\displaystyle{\frac14 (Y_{i_1j_1}+Y_{i_1j_2}+ Y_{i_2j_1}+Y_{i_2j_2})}$ & $\lambda$ \\[3pt]
 \hline
 $h_6$ & $\displaystyle{\frac14 (Y_{i_1j_1}Y_{i_1j_2}(Y_{i_1j_1}+Y_{i_1j_2}-2) + Y_{i_2j_1}Y_{i_2j_2}(Y_{i_2j_1}+Y_{i_2j_2}-2))}$ & $\lambda^3 F_3 G_2$ \\[3pt]
 \hline
\end{tabular}
\caption{Kernels $h_1$ to $h_6$ and their expectations}
\label{tab:kernels_1}
\end{table}
\renewcommand{\arraystretch}{1}

\paragraph{Simulations} To illustrate this example, we have simulated networks with the Poisson-BEDD model. We have chosen power functions for $f$ and $g$, i.e. we have set $\alpha_f$ and $\alpha_g$ in $[0,+\infty[$ and $f(u) = (\alpha_f+1) u^{\alpha_f}$ and $g(v) = (\alpha_g+1) v^{\alpha_g}$. Therefore, the values of $F_2$ and $G_2$ can be set by $\alpha_f$ and $\alpha_g$. The constant $c$ is set at $0.5$, so we have considered square matrices ($m = n$). Figure~\ref{fig:estimation_rate} represents the frequency with which 2 confidence intervals, built with respectively equations~\eqref{eq:convergence_vdt} and~\eqref{eq:convergence_vd} for $\alpha = 0.95$, contain the true value of $F_2$. The curve associated with $V^\delta$ suggests that $\sqrt{N}(\widehat{\theta}_N - F_2)$ becomes close to its limiting distribution for $N \gtrsim 250$. For smaller values of $N$, the frequencies are significantly higher than $0.95$, so the confidence intervals are slightly larger than they should. The curve associated with $\widehat{V}^\delta_N$ suggests that $\widehat{V}^{\delta}_N$ underestimates $V^{\delta}$, but using Slutsky to plug in $\widehat{V}^{\delta}_N$ for $V^{\delta}$ in~\eqref{eq:convergence_vd} still leads to acceptable frequencies that converge when $N$ grows, especially for $N \gtrsim 250$. Figure~\ref{fig:estimation_dist} represents the empirical distribution of $\widehat{\theta}_N$ for different sizes $N$. It confirms that $\sqrt{N}(\widehat{\theta}_N - F_2)$ converges quickly to a normal distribution with variance $V^{\delta}$.

\begin{figure}[!hbt]
\includegraphics[width=0.9\linewidth]{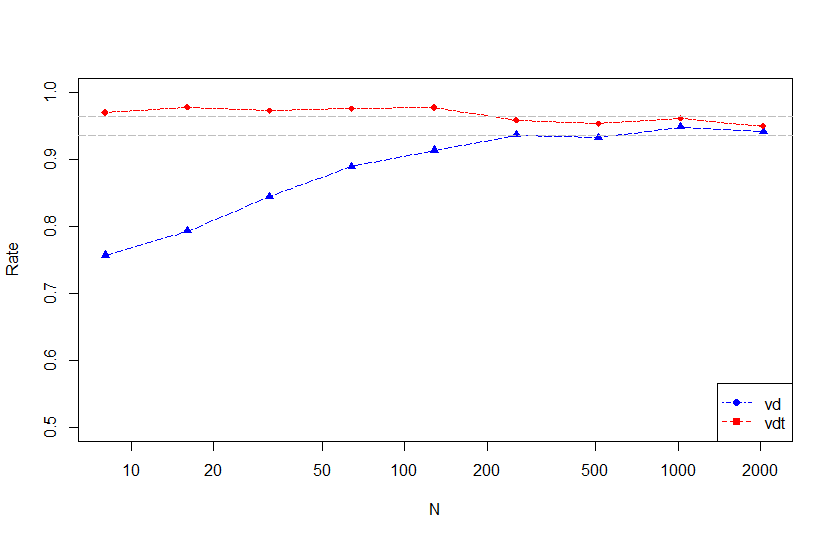}
\caption{Estimation of $F_2$: Frequency of the confidence intervals that contain the true value of $F_2$ for different values of $N$ (on a logarithmic scale). For each $N \in \{8,16,32,64,128,256,512,1024,2048\}$, we simulate $K = 1000$ networks with $\lambda=1$, $F_2=3$, $G_2=2$. For each simulated network, we estimate $F_2$ with the estimator $\widehat{\theta}_N$ and at level $\alpha = 0.95$, we build the asymptotic confidence intervals from the weak convergence results: [vdt] built from~\eqref{eq:convergence_vdt} (true value of $V^\delta$) and [vd] built from~\eqref{eq:convergence_vd} (estimated value of $V^\delta$ by $\widehat{V}^\delta_N$). The horizontal dashed lines represent the confidence interval at level $0.95$ of the frequency $Z = X/K$, if $X$ follows the binomial distribution with parameters $K$ and $\alpha$.}
\label{fig:estimation_rate}
\end{figure}

\begin{figure}[!hbt]
\includegraphics[width=\linewidth]{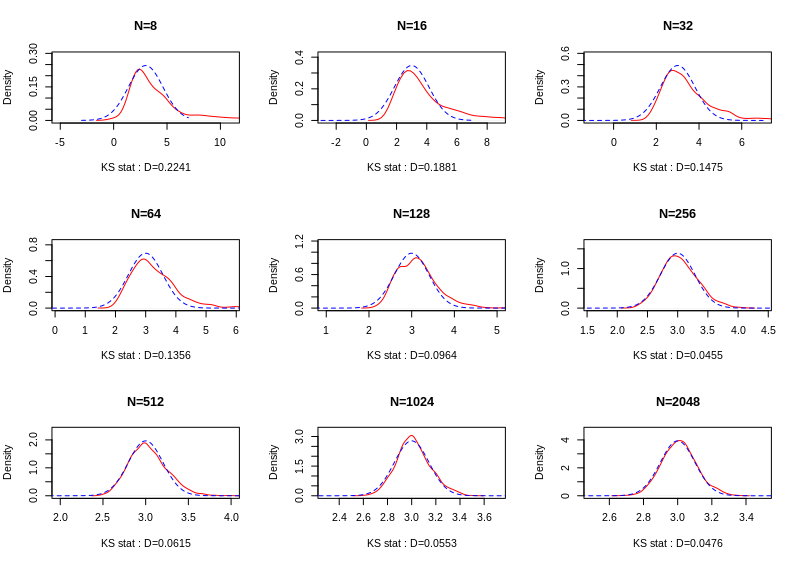}
\caption{Estimation of $F_2$: Distribution of $\widehat{\theta}_N$ for different values of $N$. For each $N \in \{8,16,32,64,128,256,512,1024,2048\}$, we simulate $K = 1000$ networks with $\lambda=1$, $F_2=3$, $G_2=2$. For each simulated network, we estimate $F_2$ with the estimator $\widehat{\theta}_N$. The empirical distributions (solid red lines) are interpolated using the density() function from base R stats package. The dashed curves in blue correspond to the normal distribution densities with mean $F_2 = 3$ and variance $V^\delta/N$. Under each plot, the value of the Kolmogorov-Smirnov test statistic $D$ between the empirical distribution of $\widehat{\theta}_N$ and the normal distribution with mean $F_2 = 3$ and variance $V^\delta/N$ is given. $D=\sup_x | F_{emp}(x) - F(x) |$ where $F_{emp}$ is the empirical c.d.f. of $\widehat{\theta}_N$ and $F(x)$ the c.d.f. of the normal distribution with mean $F_2 = 3$ and variance $V^\delta/N$.}
\label{fig:estimation_dist}
\end{figure}

\subsection{Network comparison \label{subsec:comp}}

Some methods have been developed to compare networks. Network statistics, graph spectra, network motifs or graph alignment methods can be used to build a distance (or similarity scores) between two networks \citep{emmert2016fifty, tantardini2019comparing}. In the context of random networks, fewer comparison methods rely on generative random graph models and they are relatively recent \citep{asta2015geometric, maugis2020testing}. A model-based approach offers two advantages. First, by suggesting a distribution on the networks, one might be able to design a distance with known distribution and therefore use statistical tests to compare networks. Second, the use of a generative model makes it considerably easier to interpret, one can use to the model parameters to design a suitable distance to compare the networks, giving insights into the underlying process generating them. Such ability to interpret is particularly interesting in applications such as ecology, where it is crucial to understand how and why the networks differ \citep{pellissier2018comparing}. In this section, we show how one can extend the usage of $U$-statistics to network comparison, providing a framework for model-based network comparison.

In the previous example, our analysis has been carried out on a single network. Now, consider two independent networks $Y^A$ and $Y^B$ and we wish to compare their row heterogeneity. Assume that they are respectively generated by the BEDD parameters $\Theta^A = (\lambda^A, f^A, g^A)$ and $\Theta^B = (\lambda^B, f^B, g^B)$. Then, each network is associated with their respective values $F_2^A$ and $F_2^B$. The data consists in two observed networks $Y^A_{N_A}$ and $Y^B_{N_B}$, which are assumed to be extracted from the first $m_A$ (respectively $m_B$) rows and $n_A$ (respectively $n_B$) columns of the infinite matrices $Y_A$ and $Y_B$. We would like to perform the following test: $\mathcal{H}_0 : F^A_2 = F^B_2$ vs. $\mathcal{H}_1 : F^A_2 \neq F^B_2$ using the two observed networks.

\paragraph{General method} Let $N := N_A + N_B$. Suppose that $N_A/N \xrightarrow[N \rightarrow + \infty]{} \rho \in ]0,1[$. Then one can simply notice that $\widehat{\delta}_N(Y^A,Y^B) := \widehat{\theta}_{N_A}(Y^A) - \widehat{\theta}_{N_B}(Y^B)$, where $\widehat{\theta}_{N}(Y)$ is the estimator of Proposition~\ref{prop:convergence_est} taken on the matrix $Y$, is still asymptotically normal from equation~\eqref{eq:convergence_vdt}
\begin{equation*}
    \sqrt{\frac{N}{V(\Theta^A, \Theta^B)}} \bigg(\widehat{\delta}_N(Y^A,Y^B) - (F^A_2 - F^B_2)\bigg) \xrightarrow[N \rightarrow \infty]{\mathcal{D}} \mathcal{N}(0, 1),
\end{equation*}
with $V(\Theta^A, \Theta^B) = V^\delta(\Theta^A)/\rho + V^\delta(\Theta^B)/(1-\rho)$ and for BEDD parameters $\Theta$, $V^{\delta}(\Theta)$ is given by~\eqref{eq:variance_vdt}.

Using the estimators $\widehat{V}^{\delta}_N$ stemming from the delta-method~\eqref{eq:estimator_vd}, we build $\widehat{V}_N(Y^A, Y^B)$ a consistent estimator for $V(\Theta^A, \Theta^B)$
$$\widehat{V}_N(Y^A, Y^B) = \frac{1}{\rho}\widehat{V}^{\delta}_N(Y^A)+ \frac{1}{1-\rho}\widehat{V}^{\delta}_N(Y^B).$$
Hence, Slutsky's theorem ensures that
\begin{equation*}
    \sqrt{\frac{N}{\widehat{V}_N(Y^A, Y^B)}} \bigg(\widehat{\delta}_N(Y^A,Y^B) - (F^A_2 - F^B_2)\bigg) \xrightarrow[N \rightarrow \infty]{\mathcal{D}} \mathcal{N}(0, 1).
\end{equation*}

In this example, we consider the statistical test $\mathcal{H}_0 : F^A_2 = F^B_2$ vs. $\mathcal{H}_1 : F^A_2 \neq F^B_2$. So we use the test statistic
\begin{equation}
    Z_N(Y^A, Y^B) = \sqrt{\frac{N}{\widehat{V}_N(Y^A, Y^B)}}\widehat{\delta}_N(Y^A,Y^B),
    \label{eq:test_statistic_comparison}
\end{equation}
for which Slutsky's theorem applies
$$Z_N(Y^A, Y^B) - \sqrt{\frac{N}{\widehat{V}_N(Y^A, Y^B)}}(F^A_2 - F^B_2) \xrightarrow[N \rightarrow + \infty]{\mathcal{D}} \mathcal{N}(0,1).$$

Under $\mathcal{H}_0$, $Z_N(Y^A, Y^B) \xrightarrow[N \rightarrow + \infty]{\mathcal{D}} \mathcal{N}(0,1)$ which allows us to build asymptotic acceptance intervals for this test at level $\alpha$ with the $(1-\alpha/2)^{th}$ percentile $q_{1 - \alpha/2}$ of the standard normal distribution:
$$I(\alpha) = [-q_{1-\frac{\alpha}{2}}, q_{1-\frac{\alpha}{2}}].$$

\paragraph{Simulations} Figure~\ref{fig:comparison_power} shows simulation results for this test. Once again, we consider networks generated by the Poisson-BEDD model with power law functions $f$ and $g$. To perform the test, we generate couples of observed networks $(Y^A_{N_A}, Y^B_{N_B})$ with fixed and identical $\lambda^A = \lambda^B$ and $g^A=g^B$. $f^A$ is also fixed, but we let $f^B$ vary by setting the parameter $\alpha_{f^B}$ of the power law, which is used to set $F_2^B$. The empirical power for this test with varying $F_2^B$ is evaluated for several values of $N$. It is compared with the asymptotic theoretical power $\psi_N(\Theta^A, \Theta^B)$ for this test. Let $\mu_N(\Theta^A, \Theta^B) := \sqrt{\frac{N}{V(\Theta^A, \Theta^B)}}(F^A_2 - F^B_2)$. If a random variable $\tilde{Z}_N$ is such that $\tilde{Z}_N - \mu_N(\Theta^A, \Theta^B) \overset{\mathcal{D}}{\sim} \mathcal{N}(0,1)$, then $\psi_N(\Theta^A, \Theta^B) = \mathbb{P}\left(\tilde{Z}_N \in I(\alpha) \right)$, so it can be computed with
\begin{equation}
    \psi_N(\Theta^A, \Theta^B) = F_{(\Theta^A, \Theta^B)}\left(q_{1-\frac{\alpha}{2}}\right) - F_{(\Theta^A, \Theta^B)}\left(-q_{1-\frac{\alpha}{2}}\right)
    \label{eq:theoretical_power}
\end{equation}
where $F_{(\Theta^A, \Theta^B)}(t)$ is the cumulative distribution function of a Gaussian variable with mean $\mu_N(\Theta^A, \Theta^B)$ and variance $1$. We notice that the empirical power becomes very close to the asymptotic theoretical power as $N$ grows, which suggests that this test works well for networks with $N \gtrsim 100$.

\begin{figure}[!hbt]
\includegraphics[width=\linewidth]{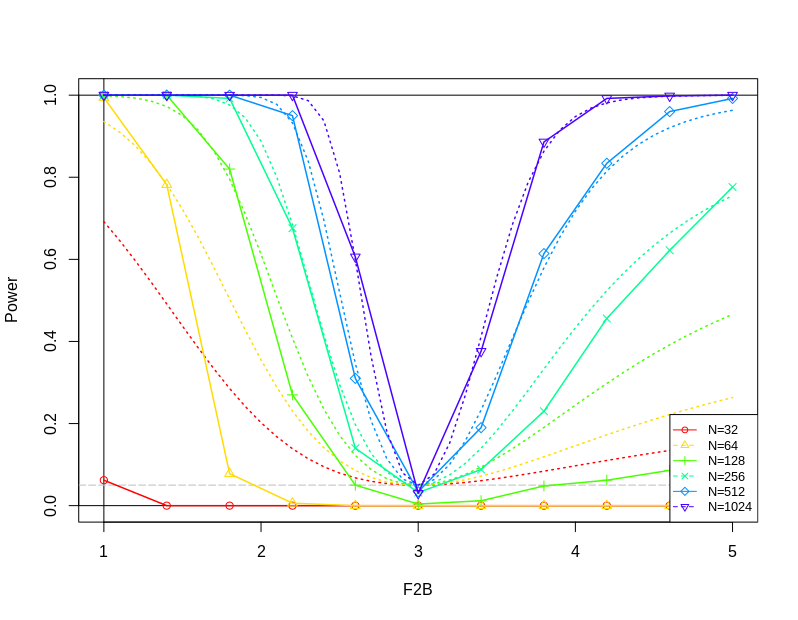}
\caption{Comparison of $F_2$ for two networks: Power of the test $\mathcal{H}_0 : F^A_2 = F^B_2$ vs. $\mathcal{H}_1 : F^A_2 \neq F^B_2$ using the statistic $Z_N(Y^A, Y^B)$ defined by~\eqref{eq:test_statistic_comparison}. We set $\lambda^A = \lambda^B = 1$, $G_2^A = G_2^B = 2$, $c^A = c^B = 0.5$. The value of $F^A_2$ is fixed at $3$. Only $N$ and $F_2^B$ will vary. Several values of $F^B_2$ are considered between $1$ and $5$. For each $N \in \{32,64,128,256,512,1024\}$, for each $F^B_2$, we generate $K = 200$ couple of networks of same size $N^A = N^B = N/2$ with respective $F_2$ values $F^A_2$ and $F^B_2$. On each couple of networks $(Y^A, Y^B)$, we compute $Z_N(Y^A, Y^B)$ and we reject the hypothesis $\mathcal{H}_0 $ if $Z_N(Y^A, Y^B) \not\in I(\alpha)$. The empirical power (solid lines) is the frequency with which the hypothesis is admitted among the $K$ simulations. The theoretical power (dashed lines) is the function $\psi_N(\Theta^A, \Theta^B)$, which only depends on $F^B_2$ since the other parameters are constant, computed with equation~\eqref{eq:theoretical_power}. }
\label{fig:comparison_power}
\end{figure}

\subsection{Motif frequencies}
\label{subsec:motifs}

A motif is a small-size subgraph. The frequencies of occurences of motifs (sometimes called network moments) are widely studied in network theory. Motifs frequencies have known asymptotic distribution under many generative models, so they can be used to analyze binary networks \citep{stark2001compound, picard2008assessing, reinert2010random, bickel2011method, bhattacharyya2015subsampling, levin2019bootstrapping, maugis2020testing, naulet2021bootstrap, ouadah2022motif}. Many probabilistic graph models also rely on motif frequencies, such as the Exponential Random Graph Model \citep{frank1986markov} or the \textit{dk}-random graphs \citep{orsini2015quantifying}. For many real networks, one can interpret the frequencies of certain motifs, see examples for transcriptional networks \citep{shen2002network}, protein networks \citep{prvzulj2004modeling}, social networks \citep{bearman2004chains}, evolutionary trait networks \citep{przytycka2006important}, ecological food webs \citep{bascompte2005simple, stouffer2007evidence}, ecological mutualistic networks \citep{baker2015species, simmons2019motifs}.

It naturally arises that frequencies of bipartite motifs of size $2 \times 2$ can be expressed as quadruplet $U$-statistics and can be integrated in our framework. If $Y$ is a binary adjacency matrix, then one can count the motifs using a kernel and obtain statistical guarantees. For example, the motif represented in Figure~\ref{fig:motif_example} can be counted with the kernel
\begin{equation*}
\begin{split}
    h_7(Y_{(i_1,i_2;j_1,j_2)}) =& \frac{1}{4} \bigg(Y_{i_1j_1} Y_{i_1j_2} Y_{i_2j_1} (1 - Y_{i_2j_2}) + Y_{i_1j_1} Y_{i_1j_2} Y_{i_2j_2} (1 - Y_{i_2j_1}) \\
    &+ Y_{i_1j_1} Y_{i_2j_1} Y_{i_2j_2} (1 - Y_{i_1j_2}) + Y_{i_1j_2}Y_{i_2j_1} Y_{i_2j_2} (1 - Y_{i_1j_1})\bigg).
\end{split}
\end{equation*}

\begin{figure}[!hbt]
\includegraphics[width=0.2\linewidth]{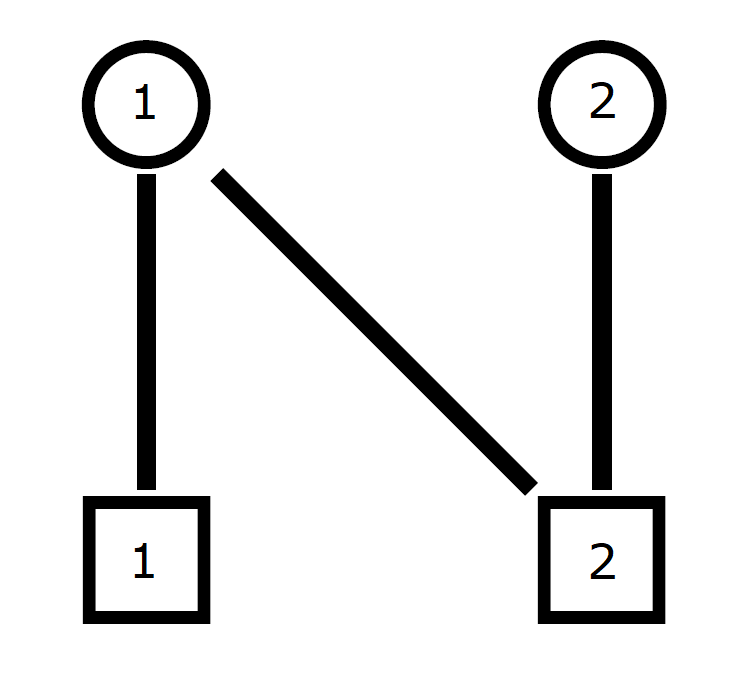}
\caption{Motif counted by $U^{h_7}_N$. The circles and the squares represent the two types of nodes of a bipartite network. Assuming that the circles correspond to the rows and the squares to the columns of the adjacency matrix, then the submatrix associated to this subgraph is $Y_{(1,2;1,2)} = \begin{pmatrix} 1 & 1 \\ 0 & 1 \end{pmatrix}$  (figure taken from~\citealp{ouadah2022motif}).}
\label{fig:motif_example}
\end{figure}

Theorem~\ref{th:gaussian_theorem} shows that the associated $U$-statistic $U^{h_7}_N$ converges to the theoretical frequency $T$ of this motif given the network model and it is asymptotically normal. Suppose $Y \sim \text{Bernoulli-BEDD}(\Theta)$, where $\Theta = (\lambda, f, g)$ are BEDD parameters, then
\begin{equation}
    \sqrt{\frac{N}{V^{h_7}}} \bigg(U^{h_7}_N - T(\Theta)\bigg) \xrightarrow[N \rightarrow \infty]{\mathcal{D}} \mathcal{N}(0, 1).
    \label{eq:convergence_motif_th}
\end{equation}
where following derivations given in Lemma~\ref{lem:motif_variance}, $T(\Theta) = \lambda^3 F_2 G_2 (1 - \lambda F_2 G_2)$ and 
\begin{equation*}
\begin{split}
    V^{h_7} &= \frac{4 \lambda^6}{c} G_2^2 \left[\lambda^2 F_4 F_2^2 G_2^2 - \lambda F_4 F_2 G_2 - \lambda F_3 F_2^2 G_2 + \frac{1}{2} F_3 F_2 + \frac{1}{4} F_4 + \frac{1}{4} F_2^3 \right] \\
    & + \frac{4 \lambda^6}{1-c} F_2^2 \left[\lambda^2 G_4 G_2^2 F_2^2 - \lambda G_4 G_2 F_2 - \lambda G_3 G_2^2 F_2 + \frac{1}{2} G_3 G_2 + \frac{1}{4} G_4 + \frac{1}{4} G_2^3 \right] \\
    & - \frac{4}{c(1-c)} \left(\lambda^3 F_2 G_2 (1 - \lambda F_2 G_2)\right)^2.
\end{split}
\end{equation*}

The quantities $\lambda$, $(F_k)_{k \ge 1}$ and $(G_k)_{k \ge 1}$ appearing in the expression of the asymptotic variance $V^{h_7}$ can be consistently estimated using $U$-statistics of larger subgraphs. For any $(p,q)$, define the kernel $h_{p,q}$ of submatrices $Y_{(i_1,...,i_{p};j_1,...,j_{q})}$ of size $p \times q$ as follows:
\begin{equation*}
    h_{p,q}(Y_{(i_1,...,i_{p};j_1,...,j_{q})}) = \prod_{u = 1}^{p} \prod_{v = 1}^{q} Y_{i_u j_v}.
\end{equation*}
Then the $U$-statistic associated to $h_{p,q}$ is
\begin{equation*}
    U^{h_{p,q}}_N = \binom{m_N}{p}^{-1} \binom{n_N}{q}^{-1} \sum_{1 \le i_1 < ... < i_p \le m_N} \sum_{1 \le j_1 < ... < j_q \le n_N} h_{p,q}(Y_{(i_1,...,i_{p};j_1,...,j_{q})}).
\end{equation*}

Lemma~\ref{lem:product_pq_kernel} states that $\mathbb{E}[h_{1,q}(Y_{(1;1,...,q)})] = \lambda^q F_q$ and $\mathbb{E}[h_{p,1}(Y_{(1,...,p;1)})] = \lambda^p G_p$. Since $Y$ is a RCE matrix, Kallenberg's law of large number (Lemma 12 of~\citealp{kallenberg1999multivariate}) applies to these $U$-statistics and $U^{h_{1,q}}_N \xrightarrow[N \rightarrow \infty]{a.s.} \lambda^q F_q$ and $U^{h_{p,1}}_N \xrightarrow[N \rightarrow \infty]{a.s.} \lambda^p G_p$.

Using these consistent estimators for $\lambda$, and the $(F_k)_{k \ge 1}$ and $(G_k)_{k \ge 1}$, we build an example of consistent estimator for $V^{h_7}$ 
\begin{equation*}
\begin{split}
    \widehat{V}_N &= \frac{4(U^{h_{2,1}}_N)^2}{c} \left[\frac{U^{h_{1,4}}_N (U^{h_{1,2}}_N)^2 (U^{h_{2,1}}_N)^2}{(U^{h_{1,1}}_N)^{8}} - \frac{U^{h_{1,4}}_N U^{h_{1,2}}_N U^{h_{2,1}}_N}{(U^{h_{1,1}}_N)^5} - \frac{U^{h_{1,3}}_N (U^{h_{1,2}}_N)^2 U^{h_{2,1}}_N}{(U^{h_{1,1}}_N)^6} \right. \\
    &\quad\quad\quad\quad\quad\quad \left. + \frac{1}{2} \frac{U^{h_{1,3}}_N U^{h_{1,2}}_N}{(U^{h_{1,1}}_N)^3} + \frac{1}{4} \frac{(U^{h_{1,4}}_N)^4}{(U^{h_{1,1}}_N)^2} + \frac{1}{4} \frac{(U^{h_{1,2}}_N)^3}{(U^{h_{1,1}}_N)^4} \right] \\
    &+ \frac{4(U^{h_{1,2}}_N)^2}{1-c} \left[\frac{U^{h_{4,1}}_N (U^{h_{2,1}}_N)^2 (U^{h_{1,2}}_N)^2}{(U^{h_{1,1}}_N)^{8}} - \frac{U^{h_{4,1}}_N U^{h_{2,1}}_N U^{h_{1,2}}_N}{(U^{h_{1,1}}_N)^5} - \frac{U^{h_{3,1}}_N (U^{h_{2,1}}_N)^2 U^{h_{1,2}}_N}{(U^{h_{1,1}}_N)^6} \right. \\
    &\quad\quad\quad\quad\quad\quad \left. + \frac{1}{2} \frac{U^{h_{3,1}}_N U^{h_{2,1}}_N}{(U^{h_{1,1}}_N)^3} + \frac{1}{4} \frac{(U^{h_{4,1}}_N)^4}{(U^{h_{1,1}}_N)^2} + \frac{1}{4} \frac{(U^{h_{2,1}}_N)^3}{(U^{h_{1,1}}_N)^4} \right] \\
    &- \frac{4}{c(1-c)} (U^{h_7}_N)^2.
\end{split}
\end{equation*}
This expression may seem complex at first, however it is computationally simple as one only needs to compute $U^{h_{p,1}}_N$ for $1 \le p \le 4$ and $U^{h_{1,q}}_N$ for $1 \le q \le 4$ which can be easily done (see Appendix~\ref{app:matrix_operation}).

From Slutsky's theorem, it follows that 
\begin{equation}
    \sqrt{\frac{N}{\widehat{V}_N}} \bigg(U^{h_7}_N - T(\Theta)\bigg) \xrightarrow[N \rightarrow \infty]{\mathcal{D}} \mathcal{N}(0, 1).
    \label{eq:convergence_motif}
\end{equation}
This result can be used to build asymptotic confidence intervals for the motif frequency $T(\Theta)$, like in the previous examples of Sections~\ref{subsec:est} and~\ref{subsec:comp}.

\paragraph{Simulations} We simulate networks with the Bernoulli-BEDD model, with power law functions for $f$ and $g$, similar to which of previous examples. For $\alpha_f$ and $\alpha_g$ in $[0, +\infty[$,  $f(u) = (\alpha_f + 1) u^{\alpha_f}$ and $g(v) = (\alpha_g + 1) v^{\alpha_g}$. $\alpha_f$ and $\alpha_g$ can be used to set $F_2$ and $G_2$. $\alpha_f$ and $\alpha_g$ also determine the maximum value for $\lambda$, as we should have $\lambda \le \lambda_M = (f(1)g(1))^{-1} = (\alpha_f +1)^{-1}(\alpha_g+1)^{-1}$. The $c$ constant remains at $0.5$. 

Figure~\ref{fig:motifs_rate} represents the frequency with which the 2 confidence intervals built from equations~\eqref{eq:convergence_motif_th} and~\eqref{eq:convergence_motif} contain the true value of the target motif frequency $T(\Theta)$. We see that as $N$ grows larger than $250$, the frequency becomes very close to $0.95$, although the variance is still underestimated until $N \approx 2000$. In contrast to the example of Section~\ref{subsec:est}, the frequencies for $N$ smaller than $16$ are very low ($0.5$ and lower). This is an expected result as the estimator $U^{h_7}_N$ counts the motifs contained in the networks. The maximum number of motifs in the network is $M_N = \binom{m_N}{2} \binom{n_N}{2}$. For a fixed $N$, $U^{h_7}_N$ can only take discrete values in $\left(k M_N^{-1} \right)_{0 \le k \le M_N}$. The support of $U^{h_7}_N$ is more and more restricted as $N$ becomes smaller, which makes the empirical distribution of $U^{h_7}_N$ more dissimilar from a Gaussian distribution. 

This is also reflected in Figure~\ref{fig:motifs_dist} as the discrete support still appears very clearly in the yet smoothed distribution density of $U^{h_7}_N$ for $N=8$ and $N=16$. Nevertheless, we see that the empirical distribution converges quickly to a Gaussian distribution, even faster than in the $F_2$ estimation example of Section~\ref{subsec:est}, as the Kolmogorov-Smirnov statistics are smaller if $N$ is larger than $32$.

\begin{figure}[!hbt]
\includegraphics[width=0.9\linewidth]{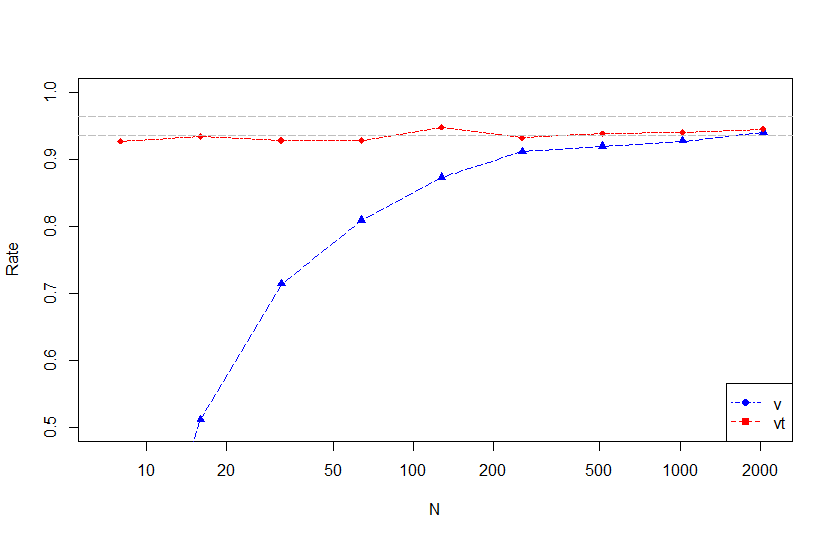}
\caption{Motifs count: Frequency of the confidence intervals that contain the theoretical value $T(\Theta)$ for different values of $N$ (on a logarithmic scale). For each $N \in \{8,16,32,64,128,256,512,1024,2048\}$, we simulate $K = 1000$ networks with $F_2=2$, $G_2=2$, $\lambda=0.9 \lambda_M$. For each simulated network, we determine the motif frequency with the estimator $U^{h_7}_N$ and at level $\alpha = 0.95$, we build the asymptotic confidence intervals from the weak convergence results: [vt] built from~\eqref{eq:convergence_motif_th} (true value of $V^{h_7}$) and [v] built from~\eqref{eq:convergence_motif} (estimated value of $V^{h_7}$ by $\widehat{V}_N$). The horizontal dashed lines represent the confidence interval at level $0.95$ of the frequency $Z = X/K$, if $X$ follows the binomial distribution with parameters $K$ and $\alpha$.}
\label{fig:motifs_rate}
\end{figure}

\begin{figure}[!hbt]
\includegraphics[width=\linewidth]{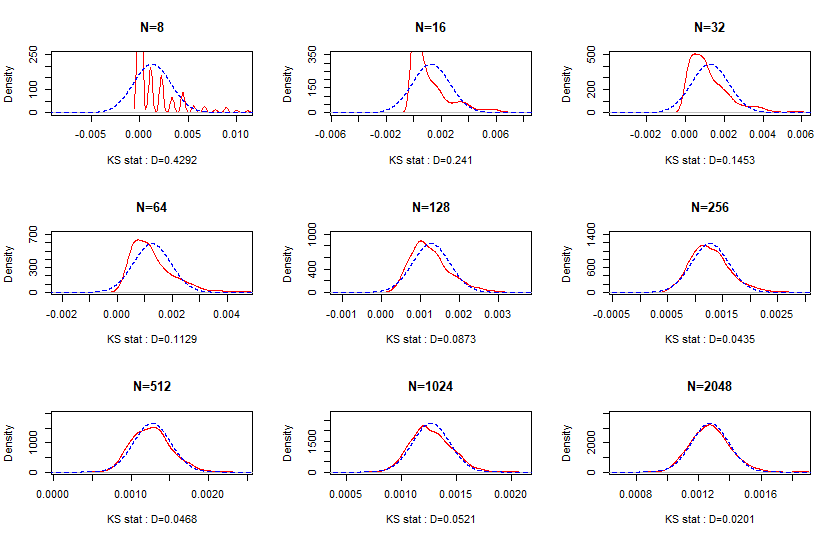}
\caption{Motifs count: Distribution of $U^{h_7}_N$ for different values of $N$. For each $N \in \{8,16,32,64,128,256,512,1024,2048\}$, we simulate $K = 1000$ networks with $F_2=2$, $G_2=2$, $\lambda=0.9 \lambda_M$. For each simulated network, we determine the motif frequency with the estimator $U^{h_7}_N$. The empirical distributions (solid red lines) are interpolated using the density() function from base R stats package. The dashed curves in blue correspond to the normal distribution densities with mean $T(\Theta)$ and variance $V^{h_7}/N$. Under each plot, the value of the Kolmogorov-Smirnov test statistic $D$ between the empirical distribution of $U^{h_7}_N$ and the normal distribution with mean $T(\Theta)$ and variance $V^{h_7}/N$ is given. $D=\sup_x | F_{emp}(x) - F(x) |$ where $F_{emp}$ is the empirical c.d.f. of $U^{h_7}_N$ and $F(x)$ the c.d.f. of the normal distribution with mean $T(\Theta)$ and variance $V^{h_7}/N$.}
\label{fig:motifs_dist}
\end{figure}

\section{Discussion \label{sec:disc}}

In this paper, we have proven a weak convergence result for quadruplet $U$-statistics over RCE matrices, using a backward martingale approach. We use the Aldous-Hoover representation of RCE matrices and a Hewitt-Savage type argument to extend this result and obtain a CLT in the dissociated case. Using this CLT, we provide a general framework to perform statistical inference on bipartite exchangeable networks through several examples.

Indeed, $U$-statistics can be used to build estimators. The advantage of taking quadruplets is to define functions over several interactions of the same row or column. This allows us to extract information on the row and column distribution. The CLT then guarantees an asymptotic normality result of the estimators, where the only unknown is their asymptotic variances, which have to be estimated then plugged in with Slutsky's theorem.

\paragraph{Computational cost} One interesting feature of the kernels chosen in Section~\ref{sec:app} is their computational simplicity. This simplicity comes naturally when considering quadruplet kernels consisting of small products. Indeed, if we denote $Y_N := (Y_{ij})_{1 \le i \le m_N, 1 \le j \le n_N}$, one can write $U^{h_1}_N$ and $U^{h_2}_N$ used in the $F_2$ estimation example (Sections~\ref{subsec:est} and~\ref{subsec:comp}) as
\begin{equation}
\begin{split}
    U^{h_1}_{N} &= \frac{1}{n_N m_N (m_N-1)} \left[ |Y_N^TY_N|_1 - \Trace(Y_N^TY_N)\right] \\
    U^{h_2}_{N} &= \frac{1}{m_N(m_N-1)n_N(n_N-1)} \left[ (|Y_N|_1)^2 - |Y_N^TY_N|_1 + \Trace(Y_N^TY_N)\right. \\
    & \quad\left.- |Y_NY_N^T|_1 + \Trace(Y_NY_N^T) - |Y^{\odot2}_N|_1 \right].
\end{split}
    \label{eq:ustat_matrix_operation}
\end{equation}
where $\Trace$ is the trace operator. We see that $U^{h_1}_N$ and $U^{h_2}_N$ can be computed using only basic operations on matrices, which are optimized in most computing software. This can also be said for all the other $U$-statistics used in this example, and by extension for the estimators $\widehat{\theta}_N$ and $\widehat{V}_N$. Expressions for the remaining $U$-statistics are given in Appendix~\ref{app:matrix_operation}.

In the motif count example (Section~\ref{subsec:motifs}), the $U$-statistic $U^{h_7}_N$ can also be easily computed, despite the $h_{p,q}$ being kernels over submatrices larger than a quadruplet (at least one dimension greater than 2). The $U^{h_{p,q}}_N$ $U$-statistics normally involve more complex summations but fortunately, we show in Appendix~\ref{app:matrix_operation} that simpler expressions can be found for $p=1$ or $q=1$.

\paragraph{Other models: graphons} We have seen that for a class of BEDD models (those falling under Assumption~\ref{ass:distfunc}), the quadruplet $U$-statistics are particularly interesting because a single quadruplet contains all the information of the model. The Bernoulli-BEDD used in Section~\ref{subsec:motifs} is an example of model where this assumption does not hold. Still, one can build estimators, apply Theorem~\ref{th:gaussian_theorem} and perform statistical inference on this model, like in Section~\ref{subsec:motifs}. In fact, the only conditions on the model are that it should be RCE and dissociated, i.e. it can be written as a bipartite W-graph model (see Section~\ref{subsec:aldous_hoover}). For example, given the W-graph model $Y_{ij}~|~\xi_i, \eta_j \sim \mathcal{P}(\lambda w(\xi_i, \eta_j))$ with $\int \int w = 1$, one could have tested if it is of product form, i.e. if $w$ can be written as $w(u,v) = f(u) g(v)$ (as in the BEDD models). An appropriate kernel for this test would be
\begin{equation*}
\begin{split}
    h(Y_{(i_1, i_2; j_1, j_2)}) =& \frac{1}{4}Y_{i_1j_1} Y_{i_2j_2}(Y_{i_1j_1} + Y_{i_2j_2} - Y_{i_1j_2} - Y_{i_1j_2} - 2) \\
    &+ \frac{1}{4}Y_{i_1j_2} Y_{i_2j_1} (Y_{i_1j_2} + Y_{i_2j_1} - Y_{i_1j_1} - Y_{i_2j_2} - 2)
\end{split}
\end{equation*}
as $\mathbb{E}[h(Y_{(i_1, i_2; j_1, j_2)})] = \int \int w(u,v)(w(u,v) - f(u)g(v)) du dv$ with $f(u) = \int w(u,v) dv$ and $g(v) = \int w(u,v) du$ and should be equal to 0 if the hypothesis is true.

\paragraph{Extension to larger subgraphs} It is legitimate to wonder if one can extend our framework to $U$-statistics over submatrices of size different from $2 \times 2$, for example $Y_{(i_1,...,i_{p};j_1,...,j_{q})}$ of size $p \times q$, i.e. 
\begin{equation*}
  U^h_N = \left[\binom{m_N}{p}\binom{n_N}{q}\right]^{-1} \sum_{1\le i_1<...<i_{p}\le m_N} \sum_{1\le j_1<...<j_{q}\le n_N} h(Y_{(i_1,...,i_{p};j_1,...,j_{q})}).
\end{equation*} 
Such generalization opens up many possibilities by building new estimators. 

First, as seen in Section~\ref{subsec:motifs}, in the Bernoulli-BEDD model, the quantities $F_k$ and $G_k$ cannot be retrieved by a quadruplet for $k \ge 3$, but $F_k$ can be retrieved with subgraphs of size $1 \times k$ and $G_k$ with subgraphs of size $k \times 1$. Second, our framework can also be used to count motifs of size larger that $2 \times 2$, since the maximum size of the motifs is determined by the size of the kernel. Finally, in the row heterogeneity example where we used formula~\eqref{eq:convergence_lambda2f2} to derive an asymptotic confidence interval for $F_2$, we notice that one could have estimated the term $\lambda^4 F_4$ appearing in $V$ with a kernel over submatrices of size $1 \times 4$ such as $h(Y_{(i_1; j_1, j_2,j_3,j_4)}) = Y_{i_1j_1}Y_{i_1j_2}Y_{i_1j_3}Y_{i_1j_4}$ and $\mathbb{E}[h(Y_{(i_1; j_1, j_2,j_3,j_4)})] = \lambda^4 F_4$. 

Actually, our theorem can be extended to $U$-statistics over larger subgraphs under similar conditions. All the steps of our proof can be adapted to $U$-statistics of larger subgraphs. These $U$-statistics are indeed backward martingales and the equivalent of Propositions~\ref{prop:asymptotic_variance} and~\ref{prop:lindeberg_condition} require more calculus. As a consequence, the asymptotic variance also has a different expression. On the one hand, such an extension would allow more flexibility in the choice of the kernel, hence the ability to build more complex estimators that are asymptotically normal. On the other hand, in practice, the computation of such $U$-statistics may also be more complex and computationally demanding, whereas simple functions on quadruplets can easily be expressed with matrix operations. 

\paragraph{Degeneracy}
Degenerate cases are of interest because they are rather common. The degeneracy depends on the kernels and the distribution of $Y$. As an example, assume that one is interested to test $\mathcal{H}_0 : f \equiv 1$ vs. $\mathcal{H}_1 : f \not\equiv 1$ for a Poisson-BEDD model. Under $\mathcal{H}_0$, $F_2 = 1$ whereas under $\mathcal{H}_1$, $F_2 > 1$. We plan to use the same estimator of $F_2$ than in Section~\ref{subsec:est}. Equation~\eqref{eq:convergence_vt} of the second method could be also obtained applying Theorem~\ref{th:gaussian_theorem} to the kernel $h = h_1 - F_2 h_2$. From equation~\eqref{eq:variance_vdt}, we see that from  that the asymptotic variance $V^t = V^\delta = 0$ under $\mathcal{H}_0$, since $F_2 = F_3 = F_4 = 1$. Thus, under $\mathcal{H}_0$, this is a degenerate case and Theorem~\ref{th:gaussian_theorem} does not apply and the limiting distribution of a test statistic using this estimator is unidentified. 

Theorems~\ref{th:mixture_theorem} and~\ref{th:gaussian_theorem} avoid degeneracy by deliberately excluding the case where $V = 0$ almost surely. However, these theorems would remain valid in degenerate cases. Indeed, if $V = 0$ almost surely, then Theorems~\ref{th:mixture_theorem} and~\ref{th:gaussian_theorem} would yield $\sqrt{N}(U^h_N - U^h_\infty) \xrightarrow[n \rightarrow \infty]{\mathbb{P}} 0$. 

This can be proven to be true. First, notice that from Corollary~\ref{cor:asymptotic_variance}, $\mathbb{V}[U^h_N|\mathcal{F}_\infty] = V/N + o\left(1/N\right)$, therefore
\begin{equation*}
  \begin{split}
    \mathbb{V}[\sqrt{N}(U^h_N - U^h_\infty)] &= N \mathbb{E}[\mathbb{V}[U^h_N - U^h_\infty|\mathcal{F}_\infty]] + N \mathbb{V}[\mathbb{E}[U^h_N - U^h_\infty | \mathcal{F}_\infty]] \\
    & = N \mathbb{E}[\mathbb{V}[U^h_N|\mathcal{F}_\infty]] \\
    & = \mathbb{E}[V] + o\left(1\right)
  \end{split}
\end{equation*}
where we denote $\mathbb{V}[X]$ the variance of a random variable $X$ and we used the fact that $\Cov(U^h_N, U^h_\infty) = \Cov(U^h_N, \mathbb{E}[U^h_N|\mathcal{F}_\infty]) = \mathbb{V}[\mathbb{E}[U^h_N|\mathcal{F}_\infty]] = \mathbb{V}[U^h_\infty]$. If $V = 0$ almost surely, then $\mathbb{V}[\sqrt{N}(U^h_N - U^h_\infty)] = o(1)$. By Chebyshev's inequality, we get $\sqrt{N}(U^h_N - U^h_\infty) \xrightarrow[n \rightarrow \infty]{\mathbb{P}} 0$. \cite{austern2022limit} also comes to this conclusion if $\eta^2 = 0$ in their Theorem 17. However, they do not explicitly discuss the implications of this case.

In fact, if $V = 0$ a.s., the $U$-statistic is degenerate and the rate of convergence of $U^h_N - U^h_\infty$ is faster than $\sqrt{N}$. This behaviour is similar to regular $U$-statistic of i.i.d. variables as described by \cite{lee1990ustatistics} or \cite{arcones1992bootstrap}. In the proof of Lemma~\ref{lem:covariance_formula}, one could go further in the derivation of the covariance and developed more the content of the $o(1/N)$ term. This would have yielded a decomposition of the form:
\begin{equation*}
    \mathbb{V}[U^h_N | \mathcal{F}_\infty] = \frac{V^{(1)}}{N} + \frac{V^{(2)}}{N^2} + \frac{V^{(3)}}{N^3} + \frac{V^{(4)}}{N^4} + o\left(\frac{1}{N^4}\right),
\end{equation*}
where $V^{(1)} = V$ of Theorem~\ref{th:mixture_theorem} and $V^{(2)}$, $V^{(3)}$ and $V^{(4)}$ are non-negative $\mathcal{F}_\infty$-measurable random variables. The derivation of closed-form expressions for $V^{(2)}$, $V^{(3)}$ and $V^{(4)}$ is possible but out of scope.

If $V = V^{(1)} = 0$ a.s. but $\mathbb{P}(V^{(2)} > 0) > 0$, then we say that the $U$-statistic is degenerate of order 1 and the above formula indicates that the right normalization is $N(U^h_N - U^h_\infty)$ instead of $\sqrt{N}(U^h_N - U^h_\infty)$. We can generalize this intuition as follows: for $2 \le d \le 4$, if $V^{(d')} = 0$ a.s. for all $1 \le d' \le d-1$ and $\mathbb{P}(V^{(d)} > 0) > 0$, then we say that the $U$-statistic is degenerate of order $d-1$ and $N^{\frac{d}{2}}(U^h_N - U^h_\infty)$ converges in distribution to some random variable. However, the asymptotic distribution for degenerate $U$-statistics is not trivial in general. Even for $U$-statistics of i.i.d. variables, the limit is very dependent of the kernel $h$ and it involves combinations of products of independent Gaussian variables in a form that is not always tractable \citep{rubin1980asymptotic, lee1990ustatistics}.

Further work might be carried out to investigate the degenerate cases. One lead is to derive some Hoeffding-type decomposition (see for example \citealp{chiang2021inference} for jointly and separately exchangeable arrays, \citealp{wang2015u} for kernels of size 2) but for quadruplet kernels taken on RCE matrices. Hoeffding-type decompositions can help identify the limiting distribution of degenerate $U$-statistics, as shown by \cite{lee1990ustatistics} and \cite{arcones1992bootstrap} in the i.i.d. case.

\paragraph{Berry-Esseen} Further studies might be carried out to investigate the rate of convergence of $\sqrt{N}(U^h_{N}-U^h_\infty)$ to its limiting distribution. For specific applications, one can for now rely on simulations to assess how quickly it converges. A possible direction to find theoretical guarantees is the derivation of a Berry-Esseen-type bound, similar to \cite{austern2022limit} in their limit theorems.

\paragraph{Choice of the optimal kernels} It is simple to design new $U$-statistics using various kernels. So when it comes to estimate some particular parameter, one may have the choice between several kernels. Even though they have the same expectation, not only the asymptotic variances might differ, but their rates of convergence to their asymptotic distributions can also vary. In addition, one should remain careful that the derived $U$-statistics are easily computable. In conclusion, each kernel does not necessarily lead to the same statistical and computational guarantees. The art of designing the best estimation procedures or statistical tests using our approach relies on finding the optimal kernels, depending on the situation.


\begin{appendix}

\section{Properties of $m_N$ and $n_N$ \label{app:mn}}

In this appendix, we provide the proofs for Proposition~\ref{prop:sequences} and further properties of the sequences $m_N$ and $n_N$ defined as $m_N = 2 + \lfloor c(N+1) \rfloor$ and $n_N = 2 + \lfloor (1-c)(N+1) \rfloor$ for all $N \geq 1$, where $c$ is an irrational number  (Definition \ref{def:mn}).

\begin{proof}[Proof of Proposition~\ref{prop:sequences}]
    The second result stems from the fact that 
    \begin{equation*}
        m_N + n_N = 4 + \lfloor c(N+1) \rfloor + \lfloor (1-c)(N+1) \rfloor = 4+\lfloor c(N+1) \rfloor + \lfloor -c(N+1) \rfloor + N+1
    \end{equation*}
    and $\lfloor c(N+1) \rfloor + \lfloor -c(N+1) \rfloor = -1$ because $c(N+1)$ is not an integer since $c$ is irrational. Then, the first result simply follows as
    \begin{equation*}
        \frac{m_N}{m_N + n_N} = \frac{\lfloor c(N+1) \rfloor + 2}{N + 4} \underset{N}{\sim} \frac{c(N+1) + 2}{N + 4} \underset{N}{\sim} \frac{cN}{N},
    \end{equation*}
    where $\underset{N}{\sim}$ denotes the asymptotic equivalence when $N$ grows to infinity, i.e. $a_N \underset{N}{\sim} b_N$ if and only if $a_N / b_N \underset{N \rightarrow \infty}{\longrightarrow} 1$.
\end{proof}

\begin{proof}[Proof of Corollary~\ref{cor:partition}]
    As $m_N$ and $n_N$ are non-decreasing, the corollary is a direct consequence of $m_N + n_N = 4 + N$, because then $m_{N+1} + n_{N+1} = 4 + N + 1 = m_N + n_N + 1$.
\end{proof}

The following definition and proposition pertain to the partition of $\mathbb{N}^*$ which will be helpful in later proofs.

\begin{definition}
    We define $\mathcal{B}_c$ and $\mathcal{B}_{1-c}$ two complementary subsets of $\mathbb{N}^*$ as
    \begin{equation*}
        \mathcal{B}_c = \left\{ N \in \mathbb{N}^* : m_N = m_{N-1} + 1 \right\} \text{~and~~} \mathcal{B}_{1-c} = \left\{ N \in \mathbb{N}^* : n_N = n_{N-1} + 1 \right\}.
    \end{equation*}
    \label{def:partition_sets}
\end{definition}

\begin{proposition}
    Set $\kappa_c(m) := \big\lfloor \frac{m-2}{c} \big\rfloor$ and $\kappa_{1-c}(n) := \big\lfloor \frac{n-2}{1-c} \big\rfloor$. If $N \in \mathcal{B}_c$, then $N = \kappa_c(m_N)$. Similarly, if $N \in \mathcal{B}_{1-c}$, then $N = \kappa_{1-c}(n_N)$.
    \label{prop:form_k_m}
\end{proposition}

\begin{proof}
  Remember that $c$ is an irrational number, so if $N \in \mathcal{B}_c$, then
  \begin{equation*}
        cN + 2 < \lfloor cN \rfloor + 3 = m_{N-1} + 1 = m_N = \lfloor c(N+1) \rfloor + 2 < c(N+1) + 2,
    \end{equation*}
    which means that $\frac{m_N-2}{c} -1 < N < \frac{m_N-2}{c}$, thus $N = \big\lfloor \frac{m_N-2}{c} \big\rfloor$.
\end{proof}

\section{Backward martingales \label{app:backward}}

In this appendix, we recall the definition of decreasing filtrations, backward martingales and their convergence theorem. The proof of Theorem~\ref{th:martingale_convergence} can be found in~\cite{doob1953stochastic}, Section 7, Theorem 4.2.

\begin{definition}
    A decreasing filtration is a decreasing sequence of $\sigma$-fields $\mathcal{F} = (\mathcal{F}_n)_{n \ge 1}$, i.e. such that for all $n \ge 1$, $\mathcal{F}_{n+1} \subset \mathcal{F}_n$.
\end{definition}

\begin{definition}
    Let $\mathcal{F} = (\mathcal{F}_n)_{n \ge 1}$ be a decreasing filtration and $M = (M_n)_{n \ge 1}$ a sequence of integrable random variables adapted to $\mathcal{F}$. $(M_n, \mathcal{F}_n)_{n \ge 1}$ is a backward martingale if and only if for all $n \ge 1$, $\mathbb{E}[M_n |\mathcal{F}_{n+1}] = M_{n+1}$.
\end{definition}

\begin{theorem}
    Let $(M_n, \mathcal{F}_n)_{n \ge 1}$ be a backward martingale. Then, $(M_n)_{n \ge 1}$ is uniformly integrable, and,  denoting $M_\infty = \mathbb{E}[M_1 | \mathcal{F}_\infty]$ where $\mathcal{F}_\infty = \bigcap_{n = 1}^{\infty} \mathcal{F}_n$, we have
    \begin{equation*}
        M_n \xrightarrow[n \rightarrow \infty]{a.s., L_1} M_\infty.
    \end{equation*}
    Furthermore, if $(M_n)_{n \ge 1}$ is square-integrable, then $M_n \xrightarrow[n \rightarrow \infty]{L_2} M_\infty$.
    \label{th:martingale_convergence}
\end{theorem}

\section{Square-integrable backward martingale \label{app:square}}

In this appendix, we prove Proposition~\ref{prop:martingale}, which states that $U^h_N$ is a square-integrable backward martingale.

\begin{proposition}
    Let $Y$ be a RCE matrix. Let $h$ be a quadruplet kernel such that $\mathbb{E}[h(Y_{(1,2;1,2)})^2] < \infty$. Let $\mathcal{F}_N = \sigma\big((U^h_{k,l}, k \ge m_N, l \ge n_N)\big)$ and $\mathcal{F}_\infty = \bigcap_{N = 1}^{\infty} \mathcal{F}_N$. Set $U^h_\infty := \mathbb{E}[h(Y_{(1,2;1,2)})|\mathcal{F}_\infty]$. Then $(U^h_N, \mathcal{F}_N)_{N \ge 1}$ is a square-integrable backward martingale and $U^h_N \xrightarrow[N \rightarrow \infty]{a.s., L_2} U^h_\infty = \mathbb{E}[h(Y_{(1,2;1,2)}) | \mathcal{F}_\infty]$.
    \label{prop:martingale}
\end{proposition}

The proof relies on the following lemma.
\begin{lemma}\label{lemma:expect}
  For all $1 \le i_1 < i_2 \le m_N$ and $1 \le j_1 < j_2 \le n_N$, $\mathbb{E}[h(Y_{(i_1, i_2; j_1, j_2)}) | \mathcal{F}_{N}] = \mathbb{E}[h(Y_{(1, 2; 1, 2)}) | \mathcal{F}_{N}]$.
  \label{lem:cond_exp_quadruplet}
\end{lemma}

\begin{proof}
  In the proof of this lemma, we specify the matrices over which the $U$-statistics are taken, i.e. we denote $U^h_{k,l}(Y)$ instead of $U^h_{k,l}$ the $U$-statistic of kernel $h$ and of size $k \times l$ taken on $Y$. 
  
  By construction, for all $k \ge m_N, l \ge n_N$, for all matrix permutations $\Phi \in \mathbb{S}_{m_N} \times \mathbb{S}_{n_N}$ (only acting on the first $m_N$ rows and $n_N$ columns), we have $U^h_{k,l}(\Phi Y) = U^h_{k,l}(Y)$. Moreover, since $Y$ is RCE, we also have $\Phi Y \overset{\mathcal{D}}{=} Y$. Therefore,
  \begin{equation*}
    \Phi Y | (U^h_{k,l}(Y), k \ge m_N, l \ge n_N) \overset{\mathcal{D}}{=} Y | (U^h_{k,l}(Y), k \ge m_N, l \ge n_N).
  \end{equation*}
  That means that conditionally on $\mathcal{F}_N$, the first $m_N$ rows and $n_N$ columns of $Y$ are exchangeable and the result to prove follows from this.
\end{proof}

\begin{proof}[Proof of Proposition~\ref{prop:martingale}]
  First, we remark that as $\mathbb{E}[h(Y_{(1,2;1,2)})^2] < \infty$, then for all $N$, $\mathbb{E}[(U^h_N)^2] < \infty$. Thus, the $(U^h_N)_{N \ge 1}$ are square-integrable.
  Second, $\mathcal{F} = (\mathcal{F}_N)_{N \ge 1}$ is a decreasing filtration and for all $N$, $U^h_N$ is $\mathcal{F}_N$-measurable. 
  
  Now using lemma~\ref{lem:cond_exp_quadruplet}, we have for all $K \le N$,
  \begin{equation*}
    \begin{split}
        \mathbb{E}[U^h_{K} | \mathcal{F}_{N}] &= \binom{m_{K}}{2}^{-2} \binom{n_{K}}{2}^{-2} \sum_{\substack{1 \le i_1 < i_2 \le m_{K} \\ 1 \le j_1 < j_2 \le n_{K}}} \mathbb{E}[h(Y_{(i_1, i_2; j_1, j_2)}) | \mathcal{F}_{N}] \\
        &= \binom{m_{K}}{2}^{-2} \binom{n_{K}}{2}^{-2} \sum_{\substack{1 \le i_1 < i_2 \le m_{K} \\ 1 \le j_1 < j_2 \le n_{K}}} \mathbb{E}[h(Y_{(1, 2; 1, 2)}) | \mathcal{F}_{N}] \\
        &= \mathbb{E}[h(Y_{(1, 2; 1, 2)}) | \mathcal{F}_{N}],\\
    \end{split}
  \end{equation*}
  In particular, $\mathbb{E}[U^h_{N-1} | \mathcal{F}_{N}] = \mathbb{E}[U^h_{N} | \mathcal{F}_{N}] = U^h_{N}$, which concludes the proof that $(U^h_N, \mathcal{F}_N)_{N \ge 1}$ is a square-integrable backward martingale. Finally, Theorem~\ref{th:martingale_convergence} ensures that $U^h_N \xrightarrow[N \rightarrow \infty]{a.s., L_2} U^h_\infty$.
\end{proof}

\section{Asymptotic variance \label{app:variance}}

We prove Proposition~\ref{prop:asymptotic_variance} which gives the convergence and an expression for the asymptotic variance. The proof involves some tedious calculations. Before that, we introduce some notations to make the proof of Proposition~\ref{prop:asymptotic_variance} more readable.

\begin{notation*} In this appendix and in Appendix~\ref{app:lindeberg}, we denote 
    \begin{align}
        X_{[i_1, i_2; j_1, j_2]} & := h(Y_{\{i_1, i_2; j_1, j_2\}}), & 
        Z_{NK} & := \sqrt{N} (U_K-U_{K+1}), \nonumber \\
        S_{NK} & := \mathbb{E}[Z_{NK}^2 | \mathcal{F}_{K+1}], &
        V_N & := \sum_{K=N}^\infty S_{NK}.
        \label{eq:xzsv}
    \end{align}
    The exchangeability of $Y$ implies that $\mathbb{E}[X_{\{i_1, i_2; j_1, j_2\}} X_{\{i'_1, i'_2; j'_1, j'_2\}}|\mathcal{F}_K]$ only depends on the numbers of rows and columns shared by both $\{i_1, i_2; j_1, j_2\}$ and $\{i'_1, i'_2; j'_1, j'_2\}$. For $0 \le p \le 2$ and $0 \le q \le 2$, we set 
    \begin{equation*}
        c^{(p,q)}_K := \mathbb{E}[X_{\{i_1, i_2; j_1, j_2\}} X_{\{i'_1, i'_2; j'_1, j'_2\}}|\mathcal{F}_K],
    \end{equation*} 
    and 
    \begin{equation*}
        c^{(p,q)}_\infty := \mathbb{E}[X_{\{i_1, i_2; j_1, j_2\}} X_{\{i'_1, i'_2; j'_1, j'_2\}}|\mathcal{F}_\infty],
    \end{equation*} 
    where they share $p$ rows and $q$ columns. 
    \label{not:xzsv}
\end{notation*}

\begin{proposition}
    Let $(V_N)_{N \ge 1}$ be as defined in~\eqref{eq:xzsv}. Then, under the hypotheses of Theorem~\ref{th:mixture_theorem}, we have
    \begin{equation*}
        V_N \xrightarrow[N \rightarrow \infty]{\mathbb{P}} V = 4c^{-1}(c^{(1,0)}_\infty - (U^h_\infty)^2) + 4(1-c)^{-1}(c^{(0,1)}_\infty - (U^h_\infty)^2).
    \end{equation*}
    \label{prop:asymptotic_variance}
\end{proposition}
The proof of Proposition~\ref{prop:asymptotic_variance} will be based on the following five lemmas.

\begin{lemma}
  If $K \in \mathcal{B}_c$ (see Definition~\ref{def:partition_sets}), then 
  \begin{equation*}
      Z_{N,K-1} = \sqrt{N} \frac{2}{m_{K} - 2}( U^h_{K} - \delta_{K} ),
  \end{equation*}
  where
  \begin{equation*}
      \delta_{K} = (m_{K}-1)^{-1} \binom{n_{K}}{2}^{-1} \sum\limits_{\substack{1 \le i_1 \le m_{K}-1 \\ 1 \le j_1 < j_2 \le n_{K}}} X_{\{i_1, m_K; j_1, j_2\}}.
  \end{equation*}
  \label{lem:form_z}
\end{lemma}

\begin{proof}
  Observe that
  \begin{equation}
      \sum_{\substack{1 \le i_1 < i_2 \le m_{K} \\ 1 \le j_1 < j_2 \le n_{K}}} X_{\{i_1, i_2; j_1, j_2\}} = \sum_{\substack{1 \le i_1 < i_2 \le m_{K}-1 \\ 1 \le j_1 < j_2 \le n_{K}}} X_{\{i_1, i_2; j_1, j_2\}} + \sum_{\substack{1 \le i_1 \le m_{K}-1 \\ 1 \le j_1 < j_2 \le n_{K}}} X_{\{i_1, m_K; j_1, j_2\}}.
      \label{eq:d2_1}
  \end{equation}
  But if $K \in \mathcal{B}_c$, then $m_{K-1} = m_K - 1$ and $n_{K-1} = n_K$. Therefore, equation~(\ref{eq:d2_1}) is equivalent to
  \begin{equation*}
    \binom{m_{K}}{2} \binom{n_{K}}{2} U^h_K = \binom{m_{K}-1}{2} \binom{n_{K}}{2} U^h_{K-1} + (m_K-1)\binom{n_{K}}{2} \delta_K,
  \end{equation*}
  so
  \begin{equation*}
    U^h_{K-1} = \frac{1}{m_K-2} \left( m_K U^h_{K} - 2 \delta_K \right).
  \end{equation*}
  This concludes the proof since $Z_{N,K-1} = \sqrt{N} ( U^h_{K-1} - U^h_{K} )$.
\end{proof}

We now calculate $S_{NK}$ in the following lemmas. 

\begin{lemma}
    For all $0 \le p \le 2$ and $0 \le q \le 2$, $c^{(p,q)}_N \xrightarrow[N \rightarrow \infty]{a.s., L_1} c^{(p,q)}_\infty$.
  \label{lem:conv_cov}
\end{lemma}

\begin{proof}
  This follows from the fact that $(c^{(p,q)}_N, \mathcal{F}_N)_{N \ge 1}$ is a backward martingale.
\end{proof}

\begin{lemma}
  If $K \in \mathcal{B}_c$ (see Definition~\ref{def:partition_sets}), then 
  \begin{equation*}
      S_{N,K-1} = 4 N \bigg(\frac{(n_K-2)(n_K -3)}{(m_K -1)(m_K-2)n_K(n_K -1)} c^{(1,0)}_{K} - \frac{1}{(m_{K} - 2)^2} (U^h_{K})^2 + \psi(K) \bigg),
  \end{equation*}
  where $\psi$ does not depend on $N$ and $\psi(K) = o(m_K^{-2})$.
  \label{lem:form_s}
\end{lemma}

\begin{proof}
  Because of Lemma~\ref{lem:form_z} and the $\mathcal{F}_K$-measurability of $U^h_K$,
  \begin{equation*}
    S_{N,K-1} = \frac{4N}{(m_{K} - 2)^2}(\mathbb{E}[\delta_{K}^2|\mathcal{F}_K] + (U^h_{K})^2 -  2 U^h_{K}\mathbb{E}[\delta_{K}|\mathcal{F}_K] ).
  \end{equation*}
  First, Lemma~\ref{lem:cond_exp_quadruplet} implies that 
  \begin{equation*}
      \mathbb{E}[\delta_{K}|\mathcal{F}_K] = U^h_K.
  \end{equation*}
  Then, we can calculate
  \begin{equation*}
    \mathbb{E}[\delta_{K}^2|\mathcal{F}_K] = (m_{K}-1)^{-2} \binom{n_{K}}{2}^{-2} \sum_{\substack{1 \le i_1 \le m_{K}-1 \\ 1 \le j_1 < j_2 \le n_{K}}} \sum_{\substack{1 \le i'_1 \le m_{K}-1 \\ 1 \le j'_1 < j'_2 \le n_{K}}} \mathbb{E}[X_{\{i_1, m_K; j_1, j_2\}} X_{\{i'_1, m_K; j'_1, j'_2\}}|\mathcal{F}_K].
  \end{equation*} 
 Each term of the sum only depends on the number of rows and columns the quadruplets in $X_{\{i_1, m_K; j_1, j_2\}}$ and $X_{\{i'_1, m_K; j'_1, j'_2\}}$ have in common. For example, if they share $p$ rows and $q$ columns, it is equal to $c^{(p,q)}_K$. So by breaking down the different cases for $p$ and $q$, we may count the number of possibilities. For example, if $(p,q) = (1,2)$, then the number of possibilities is $(m_K - 1) (m_K-2) \binom{n_{K}}{2}$. This gives
  \begin{equation*}
    \begin{split}
      \mathbb{E}[\delta_{K}^2|\mathcal{F}_K] &= (m_{K}-1)^{-1} \binom{n_{K}}{2}^{-1} \bigg\{ \frac{1}{2}(m_K -2)(n_K -2)(n_K -3) c^{(1,0)}_{K} + 2 (m_K -2)(n_K-2) c^{(1,1)}_{K} \\
      &\quad+ (m_K -2) c^{(1,2)}_{K} + \frac{1}{2}(n_K-2)(n_K-3)  c^{(2,0)}_{K} + 2(n_K -2) c^{(2,1)}_{K} + c^{(2,2)}_{K} \bigg\}.
    \end{split}
  \end{equation*}  
    
    Finally, setting
  \begin{equation*}
    \begin{split}
      \psi(K) &:= (m_{K}-1)^{-3} \binom{n_{K}}{2}^{-1} \bigg\{ 2 (m_K -2)(n_K-2) c^{(1,1)}_{K}+ (m_K -2) c^{(1,2)}_{K}  \\
      &\quad+ \frac{1}{2}(n_K-2)(n_K-3)  c^{(2,0)}_{K} + 2(n_K -2) c^{(2,1)}_{K} + c^{(2,2)}_{K} \bigg\},
    \end{split}
  \end{equation*}
  we obtain the desired result, with $\psi(K) = o(m_K^{-2})$ since $\frac{m_K}{c} \underset{K}{\sim} \frac{n_K}{1-c} \underset{K}{\sim} K$.
\end{proof}

\begin{remark*}
  In the case where $K \in \mathcal{B}_{1-c}$, the equivalent formulas to those of Lemmas~\ref{lem:form_z} and~\ref{lem:form_s} are derived from similar proofs. 
  If $K \in \mathcal{B}_{1-c}$, then 
  \begin{equation*}
      Z_{N,K-1} = \sqrt{N} \frac{2}{n_{K} - 2}( U^h_{K} - \gamma_{K} ),
  \end{equation*}
  where
  \begin{equation*}
      \gamma_{K} = (n_{K}-1)^{-1} \binom{m_{K}}{2}^{-1} \sum\limits_{\substack{1 \le i_1  < i_2 \le m_{K} \\ 1 \le j_1 \le n_{K}-1}} X_{\{i_1, i_2; j_1, n_K\}},
  \end{equation*}
  and 
  \begin{equation*}
      S_{N,K-1} = 4 N \bigg(\frac{(m_K-2)(m_K -3)}{(n_K -1)(n_K-2)m_K(m_K -1)} c^{(0,1)}_{K} - \frac{1}{(n_{K} - 2)^2} (U^h_{K})^2 + \varphi(K) \bigg),
  \end{equation*}
  where $\varphi$ does not depend on $N$ and $\varphi(K) = o(n_K^{-2})$.
\end{remark*}

\begin{lemma} \label{lem:cesaro}
  Let $(R_{n})_{n \ge 1}$ be a sequence of random variables and $(\lambda_{n})_{n \ge 1}$ a sequence of real positive numbers. Set $C_n := n \sum_{k=n}^\infty \lambda_{k} R_{k}$. If
  \begin{itemize}
      \item $n \sum_{k=n}^{\infty} \lambda_{k} \xrightarrow[n \rightarrow \infty]{} 1$, and
      \item there exists a random variable $R_\infty$ such that $R_{n} \xrightarrow[n \rightarrow \infty]{a.s.} R_\infty$,
  \end{itemize} 
  then $C_n \xrightarrow[n \rightarrow \infty]{a.s.} R_\infty$. 
  Furthermore, if $R_{n} \xrightarrow[n \rightarrow \infty]{L_1} R_\infty$, then $C_n \xrightarrow[n \rightarrow \infty]{L_1} R_\infty$.
\end{lemma}

\begin{proof}
  Notice that
  \begin{equation*}
      \begin{split}
        |C_n - R_\infty| &= \big|n \sum_{k=n}^\infty \lambda_{k} R_{k} - R_\infty\big|\\
           &\le \big|n\sum_{k=n}^\infty \lambda_{k} R_{k} - n\sum_{k=n}^\infty \lambda_{k} R_\infty\big| + \big|n\sum_{k=n}^\infty \lambda_{k} R_\infty - R_\infty\big| \\
           &\le \bigg(n\sum_{k=n}^\infty \lambda_{k} \bigg) \times \sup_{k \ge n} |R_{k} - R_\infty| + \big| n \sum_{k=n}^\infty \lambda_{k} - 1\big| \times |R_\infty|. \\
      \end{split}
  \end{equation*}
  If $n \sum_{k=n}^{\infty} \lambda_{k} \xrightarrow[n \rightarrow \infty]{} 1$ and $R_{n} \xrightarrow[n \rightarrow \infty]{a.s.} R_\infty$, then for all $\omega$ fixed except a set of neglectable size, $C_n(\omega) \xrightarrow[n \rightarrow \infty]{} R_\infty(\omega)$, which gives the a.s. convergence.
  Now, consider also that
  \begin{equation*}
      \begin{split}
        \mathbb{E}\bigg[|C_n - R_\infty|\bigg] &\le n\sum_{k=n}^\infty \lambda_{k} \mathbb{E}\bigg[|R_{k} - R_\infty|\bigg] + \big| n\sum_{k=n}^\infty \lambda_{k} - 1\big| \mathbb{E}\bigg[|R_\infty|\bigg] \\
        &\le \bigg(n\sum_{k=n}^\infty \lambda_{k} \bigg) \times \sup_{k \ge n}\mathbb{E}\bigg[|R_{k} - R_\infty|\bigg] + \big| n\sum_{k=n}^\infty \lambda_{k} - 1\big| \mathbb{E}\bigg[|R_\infty|\bigg]. \\
      \end{split}
  \end{equation*}
  So if $R_{n} \xrightarrow[n \rightarrow \infty]{L_1} R_\infty$, then $\mathbb{E}\big[|R_{n} - R_\infty|\big] \xrightarrow[n \rightarrow \infty]{L_1} 0$ and $\sup_{k \ge n}\mathbb{E}\big[|R_{k} - R_\infty|\big] \xrightarrow[n \rightarrow \infty]{L_1} 0$. Since $n\sum_{k=n}^{\infty} \lambda_{k} \xrightarrow[n \rightarrow \infty]{} 1$, the first term converges to 0, and the second term too because $\mathbb{E}\big[|R_\infty|\big] < \infty$. Finally, $\mathbb{E}\big[|C_n - R_\infty|\big] \xrightarrow[n \rightarrow \infty]{} 0$.
\end{proof}

\begin{lemma}
  Let $(Q_{n})_{n \ge 1}$ be a sequence of random variables. Set $C_n := n \sum_{k=n}^\infty Q_{k}$. If there exists a random variable $C_\infty$ such that $n^2 Q_{n} \xrightarrow[n \rightarrow \infty]{a.s.} C_\infty$, then $C_n \xrightarrow[n \rightarrow \infty]{a.s.} C_\infty$. Furthermore, if $n^2 Q_{n} \xrightarrow[n \rightarrow \infty]{L_1} C_\infty$, then $C_n \xrightarrow[n \rightarrow \infty]{L_1} C_\infty$. 
  \label{lem:cesaro_weighted}
\end{lemma}

\begin{proof}
  This is a direct application of Lemma~\ref{lem:cesaro}, where $R_{n} := n^2 Q_{n}$ and $\lambda_{n} := n^{-2}$, as $n \sum_{k=n}^{\infty} k^{-2} \xrightarrow[n \rightarrow \infty]{} 1$.
\end{proof}

\begin{proof}[Proof of Proposition~\ref{prop:asymptotic_variance}]
    Recall that from Corollary~\ref{cor:partition}, $\mathcal{B}_c$ and $\mathcal{B}_{1-c}$ form a partition of the set of the positive integers $\mathbb{N}^*$, so that we can write 
    \begin{equation*}
        V_N = V^{(c)}_N + V^{(1-c)}_N,
    \end{equation*}
    where $V^{(c)}_N = \sum_{\substack{K=N+1 \\ K \in \mathcal{B}_c}}^\infty S_{N,K-1}$ and $V^{(1-c)}_N = \sum_{\substack{K=N+1 \\ K \in \mathcal{B}_{1-c}}}^\infty S_{N,K-1}$. Here, we only detail the computation of $V^{(c)}_N$, as one would proceed analogously with $V^{(1-c)}_N$.

    In $V^{(c)}_N$, the sum is over the $K \in \mathcal{B}_c$. So, from Lemma~\ref{lem:form_s},
    \begin{equation*}
        S_{N,K-1} = 4 N \bigg(\frac{(n_K-2)(n_K -3)}{(m_K -1)(m_K-2)n_K(n_K -1)} c^{(1,0)}_{K} - \frac{1}{(m_{K} - 2)^2} (U^h_{K})^2 + \psi(K)\bigg).
    \end{equation*}
    Now we use Proposition~\ref{prop:form_k_m} to replace $K$ with $\kappa_c(m_K) = \big\lfloor \frac{m_K-2}{c} \big\rfloor$ and
    \begin{equation*}
        \begin{split}
            S_{N,\kappa_c(m_K)-1} &= 4 N \bigg(\frac{(\kappa_c(m_K) - m_K +2)(\kappa_c(m_K) - m_K +1)}{(m_K -1)(m_K-2)(\kappa_c(m_K) - m_K +4)(\kappa_c(m_K) - m_K +3)} c^{(1,0)}_{\kappa_c(m_K)} \\ 
            &\quad- \frac{1}{(m_{K} - 2)^2} (U^h_{\kappa_c(m_K)})^2  + \psi(\kappa_c(m_K))\bigg).
        \end{split}
    \end{equation*}
    Therefore, because for all $K \in \mathcal{B}_c$ we have $m_K = m_{K-1} + 1$, we can then transform the sum over $K$ into a sum over $m$ and 
    \begin{equation*}
        V^{(c)}_N = \sum_{\substack{K=N+1 \\ K \in \mathcal{B}_c}}^\infty S_{N,K-1} = \sum_{m=m_{N+1}}^\infty S_{N,\kappa_c(m)-1} = N \sum_{m=m_{N+1}}^\infty R_m,
    \end{equation*}
    where $R_{m} := S_{N,\kappa_c(m)-1}/N$, i.e.
    \begin{equation*}
    \begin{split}
        R_{m} = \frac{4(\kappa_c(m) - m +2)(\kappa_c(m) - m +1)}{(m -1)(m-2)(\kappa_c(m) - m +4)(\kappa_c(m) - m +3)} c^{(1,0)}_{\kappa_c(m)} - \frac{4}{(m - 2)^2} (U^h_{\kappa_c(m)})^2  + 4\psi(\kappa_c(m)).
    \end{split}
    \end{equation*}
    
    But we notice that since $\psi(\kappa_c(m)) = o(m^{-2})$, then Lemma~\ref{lem:conv_cov} and Proposition~\ref{prop:martingale} give for all $N$,
    \begin{equation*}
        m^2 R_m \xrightarrow[m \rightarrow \infty]{a.s., L_1} 4(c^{(1,0)}_\infty - (U^h_\infty)^2).
    \end{equation*}
    And since $\frac{m_{N+1}}{N} \xrightarrow[N \rightarrow \infty]{} c$ from Proposition~\ref{prop:sequences}, we find with Lemma~\ref{lem:cesaro_weighted} that
    \begin{equation*}
        V^{(c)}_N = \frac{N}{m_{N+1}} \times m_{N+1} \sum_{m=m_{N+1}}^\infty R_m \xrightarrow[N \rightarrow \infty]{a.s., L_1} \frac{4}{c}(c^{(1,0)}_\infty - (U^h_\infty)^2).
    \end{equation*}

    We can proceed likewise with $V^{(1-c)}_N$, where all the terms have $K \in \mathcal{B}_{1-c}$, to get
    \begin{equation*}
        V^{(1-c)}_N \xrightarrow[N \rightarrow \infty]{a.s., L_1} \frac{4}{1-c}(c^{(0,1)}_\infty - (U^h_\infty)^2),
    \end{equation*}
    which finally gives
    \begin{equation*}
        V_N = V^{(c)}_N + V^{(1-c)}_N \xrightarrow[N \rightarrow \infty]{a.s., L_1} V := \frac{4}{c}(c^{(1,0)}_\infty - (U^h_\infty)^2) + \frac{4}{1-c}(c^{(0,1)}_\infty - (U^h_\infty)^2).
    \end{equation*}
\end{proof}

\section{Conditional Lindeberg condition \label{app:lindeberg}}

We verify the conditional Lindeberg condition as stated by Proposition~\ref{prop:lindeberg_condition}. We use the notations defined in Appendix~\ref{app:variance}.

\begin{proposition}
    Let $\epsilon > 0$. Then the conditional Lindeberg condition is satisfied:
    \begin{equation*}
        \sum_{K=N}^\infty \mathbb{E}\big[Z_{NK}^2 \mathds{1}_{\{ |Z_{NK}| > \epsilon\}} \big| \mathcal{F}_{K+1}\big] \xrightarrow[N \rightarrow \infty]{\mathbb{P}} 0
    \end{equation*}
    \label{prop:lindeberg_condition}
\end{proposition}

The proof relies on the four following lemmas.
\begin{lemma}
  Let $(Q_{n})_{n \ge 1}$ be a sequence of random variables. Set $C_n := n\sum_{k=n}^\infty Q_{k}$. If $n^2 \mathbb{E}\big[|Q_{n}|\big] \xrightarrow[n \rightarrow \infty]{} 0$, then $C_n \xrightarrow[n \rightarrow \infty]{\mathbb{P}} 0$.
  \label{lem:cesaro_expectation}
\end{lemma}

\begin{proof}
  Lemma~\ref{lem:cesaro_weighted} and the triangular inequality give $\mathbb{E}\big[|C_n|\big] \le n \sum_{k=n}^\infty \mathbb{E}\big[|Q_{k}|\big] \xrightarrow[n \rightarrow \infty]{} 0$. Let some $\epsilon > 0$, then Markov's inequality ensures that
  \begin{equation*}
    \mathbb{P}(|C_n| > \epsilon) \le \frac{\mathbb{E}\big[|C_n|\big]}{\epsilon} \xrightarrow[n \rightarrow \infty]{} 0.
  \end{equation*}
\end{proof}

\begin{lemma}
  For sequences of random variables $T_n$ and sets $B_n$, if $T_n \xrightarrow[n \rightarrow \infty]{L_2} T$ and $\mathds{1}(B_n) \xrightarrow[n \rightarrow \infty]{\mathbb{P}} 0$, then $\mathbb{E}[T_n^2 \mathds{1}(B_n)] \xrightarrow[n \rightarrow \infty]{} 0$.
  \label{lem:uniform_integrability_sq}
\end{lemma}
\begin{proof}
  Note that for all $n$, $a > 0$,
  \begin{equation*}
      \begin{split}
        \mathbb{E}\big[T_n^2 \mathds{1}(B_n)\big] &= \mathbb{E}\big[T_n^2 \mathds{1}(B_n)\mathds{1}(T_n^2 > a)\big] + \mathbb{E}\big[T_n^2 \mathds{1}(B_n)\mathds{1}(T_n^2 \le a)\big] \\
        &\le \mathbb{E}\big[T_n^2 \mathds{1}(T_n^2 > a)\big] + \mathbb{E}\big[a \mathds{1}(B_n)\big] \\
        &\le \mathbb{E}\big[T_n^2 \mathds{1}(T_n^2 > a)\big] + a \mathbb{P}(B_n)
      \end{split}
  \end{equation*}
  Let $\epsilon > 0$. $T_n \xrightarrow[n \rightarrow \infty]{L_2} T$, so $(T_n^2)_{n \ge 1}$ is uniformly integrable and there exists $a > 0$ such that $\mathbb{E}[T_n^2 \mathds{1}(T_n^2 > a)] \le \sup_k \mathbb{E}[T_k^2 \mathds{1}(T_k^2 > a)] \le \frac{\epsilon}{2}$. Moreover, $\mathds{1}(B_n) \xrightarrow[n \rightarrow \infty]{\mathbb{P}} 0$, which translates to $\mathbb{P}(B_n) \xrightarrow[n \rightarrow \infty]{} 0$ and there exists an integer $n_0$ such that for all $n > n_0$, $\mathbb{P}(B_n) \le \frac{\epsilon}{2a}$. Choosing such a real number $a$, we can always find an integer $n_0$ such that for $n > n_0$, we have $\mathbb{E}[T_n^2 \mathds{1}(B_n)] \le \epsilon$.
\end{proof}

\begin{lemma}
  For sequences of random variables $M_n$ and sets $B_n$, if $(M_n)_{n \ge 1}$ is a backward martingale with respect to some filtration and $\mathds{1}(B_n) \xrightarrow[n \rightarrow \infty]{\mathbb{P}} 0$, then $\mathbb{E}[M_n \mathds{1}(B_n)] \xrightarrow[n \rightarrow \infty]{} 0$.
  \label{lem:uniform_integrability_mart}
\end{lemma}
\begin{proof}
    We notice that from Theorem~\ref{th:martingale_convergence}, $(M_n)_{n \ge 1}$ is uniformly integrable, then the proof is similar to that of Lemma~\ref{lem:uniform_integrability_sq}.
\end{proof}

\begin{lemma}
  Set $A_K := m_K^{-1} \binom{n_K}{2}^{-1} \sum_{\substack{2 \le i_2 \le m_K+1 \\1 \le j_1 < j_2 \le n_K}} X_{\{1, i_2; j_1, j_2\}}$. If $K \in \mathcal{B}_c$ (see Definition~\ref{def:partition_sets}), then $A_K \overset{\mathcal{D}}{=} \delta_{K}$,where $\delta_{K}$ is defined in Lemma \ref{lem:form_z}.
  \label{lem:delta_exch}
\end{lemma}

\begin{proof}
  Remember that if $K \in \mathcal{B}_c$, then by symmetry of $h$, $\delta_{K} = (m_{K}-1)^{-1} \binom{n_{K}}{2}^{-1} \sum\limits_{\substack{1 \le i_2 \le m_{K}-1 \\ 1 \le j_1 < j_2 \le n_{K}}} X_{\{m_K, i_2; j_1, j_2\}}$.
  The exchangeability of $Y$ says that all permutations on the rows and the columns of $Y$ leave its distribution unchanged, hence for all $(\sigma_1, \sigma_2) \in \mathbb{S}_{m_K} \times \mathbb{S}_{n_K}$, we have 
  \begin{equation*}
      \delta_{K} \overset{\mathcal{L}}{=} (m_{K}-1)^{-1} \binom{n_{K}}{2}^{-1} \sum\limits_{\substack{1 \le i_2 \le m_{K}-1 \\ 1 \le j_1 < j_2 \le n_{K}}} X_{\{\sigma_1(m_K), \sigma_1(i_2); \sigma_2(j_1), \sigma_2(j_2)\}}.
  \end{equation*}
  Consider $\sigma_2$ to be the identity and $\sigma_1 \in \mathbb{S}_{m_K}$ the permutation defined by:
  \begin{itemize}
      \item $\sigma_1(i) = i+1$ if $i < m_K$,
      \item $\sigma_1(m_K) = 1$,
      \item $\sigma_1(i) = i$ if $i > m_K$.
  \end{itemize}
  Then $A_K = (m_{K}-1)^{-1} \binom{n_{K}}{2}^{-1} \sum\limits_{\substack{1 \le i_2 \le m_{K}-1 \\ 1 \le j_1 < j_2 \le n_{K}}} X_{\{\sigma_1(m_K), \sigma_1(i_2); \sigma_2(j_1), \sigma_2(j_2)\}}$, hence $A_K \overset{\mathcal{L}}{=} \delta_K$.
\end{proof}

\begin{proof}[Proof of Proposition~\ref{prop:lindeberg_condition}]
    Similarly to the proof of the Proposition~\ref{prop:asymptotic_variance}, we can verify the conditional Lindeberg condition by decomposing the sum along with $K+1 \in \mathcal{B}_c$ and $K+1 \in \mathcal{B}_{1-c}$ (Corollary~\ref{cor:partition}), so here we only consider $\sum_{\substack{K=N+1 \\ K \in \mathcal{B}_c}}^\infty \mathbb{E}\big[Z_{N,K-1}^2 \mathds{1}_{\{ |Z_{N,K-1}| > \epsilon\}} | \mathcal{F}_{K}\big]$.

    Like previously, using Proposition~\ref{prop:form_k_m}, we can transform the sum over $K$ into a sum over $m$:
    \begin{equation*}
        \sum_{\substack{K=N+1 \\ K \in \mathcal{B}_c}}^\infty \mathbb{E}\big[Z_{N,K-1}^2 \mathds{1}_{\{ |Z_{N,K-1}| > \epsilon\}} | \mathcal{F}_{K}\big] = \sum_{m=m_{N+1}}^\infty \mathbb{E}\big[Z_{N,\kappa_c(m)-1}^2 \mathds{1}_{\{ |Z_{N,\kappa_c(m)-1}| > \epsilon\}} | \mathcal{F}_{\kappa_c(m)}\big],
    \end{equation*}
    where $\kappa_c(m) = \lfloor \frac{m-2}{c} \rfloor$.
    
    We remark that using Lemma~\ref{lem:form_z}, for $m \ge m_{N+1} = m_N + 1> c(N+1)+ 2$,
    \begin{equation*}
        \begin{split}
           \mathds{1}_{\{ |Z_{N,\kappa_c(m)-1}| > \epsilon\}} &\le \mathds{1}_{\big\{ \frac{2\sqrt{N}}{m-2}|U^h_{\kappa_c(m)} - \delta_{\kappa_c(m)}| > \epsilon\big\}}\\
            &\le \mathds{1}_{\big\{ |U^h_{\kappa_c(m)} - \delta_{\kappa_c(m)}| > \frac{m-2}{2\sqrt{\frac{m-2}{c}}}\epsilon\big\}} \\
            &\le \mathds{1}_{\big\{ |U^h_{\kappa_c(m)}| > \frac{\sqrt{c(m-2)}}{4}\epsilon\big\}} + \mathds{1}_{\big\{ |\delta_{\kappa_c(m)}| > \frac{\sqrt{c(m-2)}}{4}\epsilon\big\}}.
        \end{split}
    \end{equation*}
    So, using Lemma~\ref{lem:form_z} again and the identity $(U^h_{\kappa_c(m)} - \delta_{\kappa_c(m)})^2 \le 2 \left((U^h_{\kappa_c(m)})^2 + \delta_{\kappa_c(m)}^2\right)$, we get for $m \ge m_{N+1}$,
    \begin{equation*}
    \begin{split}
        \mathbb{E}&\big[Z_{N,\kappa_c(m)-1}^2 \mathds{1}_{\{ |Z_{N,\kappa_c(m)-1}| > \epsilon\}} | \mathcal{F}_{\kappa_c(m)}\big] \\ 
        &\le \frac{8N}{(m - 2)^2} \mathbb{E}\bigg[\left((U^h_{\kappa_c(m)})^2 + \delta_{\kappa_c(m)}^2\right)  \bigg(\mathds{1}_{\big\{ |U^h_{\kappa_c(m)}| > \frac{\sqrt{c(m-2)}}{4}\epsilon\big\}} + \mathds{1}_{\big\{ |\delta_{\kappa_c(m)}| > \frac{\sqrt{c(m-2)}}{4}\epsilon\big\}}\bigg) \bigg| \mathcal{F}_{\kappa_c(m)}\bigg].
    \end{split}
    \end{equation*}
    This inequality and Lemma~\ref{lem:cesaro_expectation} imply that a sufficient condition to have the conditional Lindeberg condition is
    \begin{equation}
       \mathbb{E}\bigg[\left((U^h_{\kappa_c(m)})^2 + \delta_{\kappa_c(m)}^2\right)  \bigg(\mathds{1}_{\big\{ |U^h_{\kappa_c(m)}| > \frac{\sqrt{c(m-2)}}{4}\epsilon\big\}} + \mathds{1}_{\big\{ |\delta_{\kappa_c(m)}| > \frac{\sqrt{c(m-2)}}{4}\epsilon\big\}}\bigg)\bigg] \xrightarrow[m \rightarrow \infty]{} 0.
        \label{eq:e1_5}
    \end{equation}
    Next, we prove that this condition is satisfied.
    
    First, note that from Markov's inequality,
    \begin{equation*}
        \mathbb{P} \left( |U^h_{\kappa_c(m)}| > \frac{\sqrt{c(m-2)}}{4}\epsilon \right) \le \frac{4 \mathbb{E}[|U^h_{\kappa_c(m)}|]}{\epsilon \sqrt{c(m-2)}} \xrightarrow[m \rightarrow \infty]{} 0
    \end{equation*}
    and 
    \begin{equation*}
        \mathbb{P} \left( |\delta_{\kappa_c(m)}| > \frac{\sqrt{c(m-2)}}{4}\epsilon \right) \le \frac{4 \mathbb{E}[|\delta_{\kappa_c(m)}|]}{\epsilon \sqrt{c(m-2)}} \xrightarrow[m \rightarrow \infty]{} 0.
    \end{equation*}

    Now, remember that from Proposition~\ref{prop:martingale}, $U^h_K \xrightarrow[K \rightarrow \infty]{L_2} U^h_\infty$, therefore $U^h_{\kappa_c(m)} \xrightarrow[m \rightarrow \infty]{L_2} U^h_\infty$ and Lemma~\ref{lem:uniform_integrability_sq} can be applied, which gives
    \begin{equation}
        \mathbb{E}\bigg[(U^h_{\kappa_c(m)})^2  \bigg(\mathds{1}_{\big\{ |U^h_{\kappa_c(m)}| > \frac{\sqrt{c(m-2)}}{4}\epsilon\big\}} + \mathds{1}_{\big\{ |\delta_{\kappa_c(m)}| > \frac{\sqrt{c(m-2)}}{4}\epsilon\big\}}\bigg)\bigg] \xrightarrow[m \rightarrow \infty]{} 0.
        \label{eq:e1_1}
    \end{equation}

    Likewise, we calculated $\mathbb{E}[\delta_K^2 | \mathcal{F}_K ]$ in the proof of Lemma~\ref{lem:form_s}. The application of Lemma~\ref{lem:conv_cov} shows that $\mathbb{E}[\delta_{\kappa_c(m)}^2 | \mathcal{F}_{\kappa_c(m)} ]$ is a backward martingale. It follows from Lemma~\ref{lem:uniform_integrability_mart} that
    \begin{equation}
        \mathbb{E}\bigg[\delta_{\kappa_c(m)}^2  \mathds{1}_{\big\{ |U^h_{\kappa_c(m)}| > \frac{\sqrt{c(m-2)}}{4}\epsilon\big\}}\bigg] = \mathbb{E}\bigg[\mathbb{E}[\delta_{\kappa_c(m)}^2 | \mathcal{F}_{\kappa_c(m)} ] \mathds{1}_{\big\{ |U^h_{\kappa_c(m)}| > \frac{\sqrt{c(m-2)}}{4}\epsilon\big\}}\bigg] \xrightarrow[m \rightarrow \infty]{} 0.
        \label{eq:e1_2}
    \end{equation}
    
    Finally, applying Lemma~\ref{lem:delta_exch}, we obtain
    \begin{equation}
        \mathbb{E}\bigg[\delta_{\kappa_c(m)}^2\mathds{1}_{\big\{ |\delta_{\kappa_c(m)}| > \frac{\sqrt{c(m-2)}}{4}\epsilon\big\}}\bigg] = \mathbb{E}\bigg[A_{\kappa_c(m)}^2\mathds{1}_{\big\{ |A_{\kappa_c(m)}| > \frac{\sqrt{c(m-2)}}{4}\epsilon\big\}}\bigg],
        \label{eq:e1_3}
    \end{equation}
    where $A_K = m_K^{-1} \binom{n_K}{2}^{-1} \sum_{\substack{2 \le i_2 \le m_K+1 \\1 \le j_1 < j_2 \le n_K}} X_{\{1, i_2; j_1, j_2\}}$. Using similar arguments as in the proof of Proposition~\ref{prop:martingale}, it can be shown that $A_K$ is a square integrable backward martingale with respect to the decreasing filtration $\mathcal{F}^A_K = \sigma(A_K, A_{K+1},...)$. Therefore, Theorem~\ref{th:martingale_convergence} ensures that there exists $A_\infty$ such that $A_K \xrightarrow[K \rightarrow \infty]{L_2} A_\infty$. This proves that $A_{\kappa_c(m)} \xrightarrow[m \rightarrow \infty]{L_2} A_\infty$, so applying Lemma~\ref{lem:uniform_integrability_sq} again, we obtain
    \begin{equation}
        \mathbb{E}\bigg[A_{\kappa_c(m)}^2\mathds{1}_{\big\{ |A_{\kappa_c(m)}| > \frac{\sqrt{c(m-2)}}{4}\epsilon\big\}}\bigg] \xrightarrow[m \rightarrow \infty]{} 0.
        \label{eq:e1_4}
    \end{equation}

    Combining~(\ref{eq:e1_1}),~(\ref{eq:e1_2}),~(\ref{eq:e1_3}) and~(\ref{eq:e1_4}), we deduce that the sufficient condition~(\ref{eq:e1_5}) is satisfied, thus concluding the proof.

\end{proof}

\section{Hewitt-Savage theorem \label{app:normality}}

\begin{proof}[Proof of Theorem~\ref{th:hewitt_savage}]
  This proof adapts the steps taken by \cite{feller1971introduction} and detailed by \cite{durrett2019probability} to our case. Let $A \in \mathcal{E}_\infty$. 
  
  First, let $\mathcal{A}_N = \sigma\big((\xi_i)_{1 \le i \le m_N}, (\eta_j)_{1 \le j \le n_N}, (\zeta_{ij})_{1 \le i \le m_N, 1 \le j \le n_N}\big)$, the $\sigma$-field generated by the random variables associated with the first $m_N$ rows and $n_N$ columns. Notice that $A \in \mathcal{A} := \bigcap_{n = 1}^{\infty} \mathcal{A}_N$. Since $\mathcal{A}$ is the limit of $\mathcal{A}_N$, then for all $\epsilon > 0$, there exists a $N$ and an associated set $A_N \in \mathcal{A}_N$ such that $\mathbb{P}(A - A \cap A_N) < \epsilon$ and $\mathbb{P}(A_N - A \cap A_N) < \epsilon$, so that $\mathbb{P}(A \Delta A_N) < 2 \epsilon$, where $\Delta$ is the symmetric difference operator, i.e. $B \Delta C = (B-C)\cup(C-B)$. Therefore, we can pick a sequence of sets $A_N$ such that $\mathbb{P}(A \Delta A_N) \longrightarrow 0$.
  
  Next, we consider the row-column permutation $\Phi^{(N)} = (\sigma^{(N)}_1, \sigma^{(N)}_2) \in  \mathbb{S}_{m_N} \times \mathbb{S}_{n_N}$ defined by
  \begin{align*}
      \sigma^{(N)}_1(i) &= \begin{cases} i + m_N &\text{if } 1 \le i \le m_N,  \\
               i - m_N &\text{if } m_N +1 \le i \le 2 m_N, \\
               i &\text{if } 2 m_N + 1 \le i.
 \end{cases}
 \\
      \sigma^{(N)}_2(j) &= \begin{cases} j + n_N &\text{if } 1 \le j \le n_N,  \\
               j - n_N &\text{if } n_N +1 \le j \le 2 n_N, \\
               j &\text{if } 2 n_N + 1 \le j.
 \end{cases}
  \end{align*}
  Since $A \in \mathcal{E}_\infty$, by the definition of $\mathcal{E}_\infty$, it follows that
  \begin{equation*}
      \left\{ \omega : \Phi^{(N)} \omega \in A \right\} = \left\{ \omega : \omega \in A \right\} = A.
  \end{equation*}
  Using this, if we denote $A'_N := \left\{ \omega : \Phi^{(N)} \omega \in A_N \right\}$, then we can write that 
  \begin{equation*}
      \left\{ \omega : \Phi^{(N)} \omega \in A_N \Delta A \right\} = \left\{ \omega : \omega \in A'_N \Delta A \right\} = A'_N \Delta A.
  \end{equation*}
  Furthermore, the $(\xi_i)_{1 \le i < \infty}$, $(\eta_j)_{1 \le j < \infty}$ and $(\zeta_{ij})_{1 \le i < \infty, 1 \le j < \infty}$ are i.i.d., so
  \begin{equation*}
      \mathbb{P}(A_N \Delta A) = \mathbb{P}(\omega : \omega \in A_N \Delta A) = \mathbb{P}(\omega : \Phi^{(N)} \omega \in A_N \Delta A).
  \end{equation*}
  and we conclude that $\mathbb{P}(A'_N \Delta A) = \mathbb{P}(A_N \Delta A) \longrightarrow 0$.
  
  From this, we derive that $\mathbb{P}(A_N) \longrightarrow \mathbb{P}(A)$ and $\mathbb{P}(A'_N) \longrightarrow \mathbb{P}(A)$.
  We also remark that $\mathbb{P}(A_N \Delta A'_N) \le \mathbb{P}(A_N \Delta A) + \mathbb{P}(A'_N \Delta A) \longrightarrow 0$, so $\mathbb{P}(A_N \cap A'_N) \longrightarrow \mathbb{P}(A)$.
  
  But $A_N$ and $A'_N$ are independent, so we have $\mathbb{P}(A_N \cap A'_N) = \mathbb{P}(A_N) \mathbb{P}(A'_N) \longrightarrow \mathbb{P}(A)^2$, therefore $\mathbb{P}(A) = \mathbb{P}(A)^2$, which means that $\mathbb{P}(A) = 0$ or $1$.
\end{proof}

\section{Identifiability of the BEDD model}
\label{app:ident}

\subsection{Proof of Theorem~\ref{th:id_bedd_param}}

First, we define the generalized inverse of a cumulative distribution function and we prove some useful properties. We need Lemmas~\ref{lem:link_cdf_inverse} and~\ref{lem:moment_problem} to prove Theorem~\ref{th:id_bedd_param}.

\begin{definition}
    For any increasing, bounded, c\`adl\`ag function $t : \mathbb{R} \rightarrow \mathbb{R}$, we define its generalized inverse by $t^{-1} : \mathbb{R} \rightarrow \mathbb{R}$ as follows:
    \begin{equation*}
        t^{-1}(x) = \inf \{ u \in \mathbb{R} : t(u) > x\}.
    \end{equation*}
    \label{def:generalized_inverse}
\end{definition}

\begin{lemma}[\citealp{fortelle2020generalized}, Proposition 2.2.]
  Let $t$ be an increasing, bounded, c\`adl\`ag function. Then $(t^{-1})^{-1} = t$.
  \label{lem:inverse_inverse}
\end{lemma}

\begin{lemma}
  Let $U$ be a random variable such that $U \sim \mathcal{U}[0,1]$. Let $D = f(U)$, where $f$ is an increasing, bounded, c\`adl\`ag function. Let $\tilde{f}$ be the cumulative distribution function of $D$. Then $f^{-1} = \tilde{f}$ and $\tilde{f}^{-1} = f$.
  \label{lem:link_cdf_inverse}
\end{lemma}

\begin{proof}
    Since $f$ is right-continuous, for all $x \in \mathbb{R}$, $\mathbb{P}(f(U) \le x) =  \inf \{ u \in [0,1] : f(u) > x\}$ so that means $\tilde{f} = f^{-1}$ according to Definition~\ref{def:generalized_inverse}, so the first equation is proven. Also Lemma~\ref{lem:inverse_inverse} ensures that $f = (f^{-1})^{-1} = \tilde{f}^{-1}$, which concludes the proof of this lemma.
\end{proof}

\begin{lemma}
  Let $D$ be a random variable such for all $k \in \mathbb{N}$, $\mathbb{E}[D^k] \le \alpha^k$, for some $\alpha > 0$. Then, the distribution of $D$ is uniquely characterized by its moments.
  \label{lem:moment_problem}
\end{lemma}

\begin{proof}
  We show that the exponential generating series of the moments of $D$ $M(r) = \sum_{k=1}^\infty \mathbb{E}[D^k] \frac{r^k}{k!}$ has a positive radius of convergence.
  
  Using the fact that $k!/k^{k} \xrightarrow[k \rightarrow \infty]{} 1$ and that $\mathbb{E}[D^k] \le \alpha^k$ for all $k \in \mathbb{N}$, we see that
  \begin{equation*}
  \begin{split}
      \limsup_{k \rightarrow \infty} \left(\frac{\mathbb{E}[D^k]}{k!}\right)^\frac{1}{k} &= \limsup_{k \rightarrow \infty} \frac{\mathbb{E}[D^k]^\frac{1}{k}}{k} \\
      & \le \alpha \limsup_{k \rightarrow \infty} \frac{1}{k} \\
      & < \infty.
  \end{split}
  \end{equation*}
  So by the Cauchy-Hadamard's theorem, the series $\sum_{k} \frac{\mathbb{E}[D^k] r^k}{k!}$ converges for any $r > 0$, which is a sufficient condition so that $D$ is determined by its moments $(\mathbb{E}[D^k])_{k \ge 1}$ (see Section 9.2, Theorem 2 of \cite{billingsley1995probability}).
\end{proof}

Using Lemmas~\ref{lem:link_cdf_inverse} and~\ref{lem:moment_problem}, we can finally prove Theorem~\ref{th:id_bedd_param}.

\begin{proof}[Proof of Theorem~\ref{th:id_bedd_param}]
    Let $\Theta = (\lambda, f, g)$ be BEDD parameters. Here, we prove that $f$ is uniquely characterised by $(F_k)_{k \ge 1}$. In order to do that, we introduce a random variable $D$ which both has moments $(F_k)_{k \ge 1}$ and $f$ as the generalized inverse of its cumulative distribution function (Definition~\ref{def:generalized_inverse}). We show that $D$ is uniquely characterised by $f$ and then, by its moments. 
    \begin{enumerate}
        \item Let $U$ be a random variable such that $U \sim \mathcal{U}[0,1]$. Let $D = f(U)$. For all $k \ge 1$, $\mathbb{E}[D^k] = F_k$. 
        
        \item Since $f$ is bounded, we notice that for all $k \in \mathbb{N}$, $F_k = \int f^k \le \sup_{[0,1]} f^k = \lVert f \rVert_\infty^k$, therefore $\mathbb{E}[D^k] \le \lVert f \rVert_\infty^k$. 
        
        So Lemma~\ref{lem:moment_problem} ensures that the distribution of $D$ is uniquely characterised by its moments $(F_k)_{k \ge 1}$. 
        
        \item Now, for some other increasing, bounded, c\`adl\`ag function $f^*$. Let $D^* = f^*(U)$ and $\tilde{f}^*$ its cumulative distribution function. If $D \sim D^*$, then $\tilde{f} = \tilde{f}^*$. Therefore, using the generalized inverse (Definition~\ref{def:generalized_inverse}), we have $\tilde{f}^{-1} = (\tilde{f}^*)^{-1}$. Finally, Lemma~\ref{lem:link_cdf_inverse} implies that $f = f^*$. Thus, the distribution of $D$ is uniquely characterised by $f$.
    \end{enumerate}
    
    We can conclude by stating that the moments $(F_k)_{k \ge 1}$ of $D$ are then uniquely characterised by $f$.
    
    By symmetry, the same follows for $g$ and $(G_k)_{k \ge 1}$.
\end{proof}

\subsection{Proof of Theorem~\ref{th:id_quadruplet}}

Theorem~\ref{th:id_quadruplet} is proven by induction using two lemmas.

\begin{lemma}
  Let $\Theta = (\lambda, f, g)$ be BEDD parameters and $Y \sim \mathcal{L}\text{-BEDD}(\Theta)$. For all $i \in \mathbb{N}$ and for all $(j_1,j_2) \in \mathbb{N}^2$ such that $j_1 \neq j_2$,
  \begin{equation*}
      \mathbb{E}[\Psi_k(Y_{ij_1}) \times Y_{ij_2}] = \lambda^{k+1} F_{k+1} G_k.
  \end{equation*}
  \label{lem:ident1}
\end{lemma}

\begin{proof}
  Let $\Theta = (\lambda, f, g)$ be BEDD parameters and $Y \sim \mathcal{L}\text{-BEDD}(\Theta)$, for any $i \in \mathbb{N}$ and $(j_1,j_2) \in \mathbb{N}^2$ such that $j_1 \neq j_2$,
  \begin{equation*}
      \begin{split}
          \mathbb{E}[\Psi_k(Y_{ij_1}) \times Y_{ij_2}] &= \mathbb{E}\big[\mathbb{E}[\Psi_k(Y_{ij_1}) \times Y_{ij_2} | \xi_{i}, \eta_{j_1}, \eta_{j_2}]\big] \\
          & = \mathbb{E}\big[\mathbb{E}[\Psi_k(Y_{ij_1}) | \xi_{i}, \eta_{j_1}] \times \mathbb{E}[Y_{ij_2} | \xi_{i}, \eta_{j_2}]\big] \\
          & = \mathbb{E}\big[\lambda^k f(\xi_{i})^k g(\eta_{j_1})^k \times \lambda f(\xi_{i}) g(\eta_{j_2})\big] \\
          & = \lambda^{k+1} \mathbb{E}\big[ f(\xi_{i})^{k+1}\big] \mathbb{E}\big[ g(\eta_{j_1})^k\big]   \mathbb{E}\big[ g(\eta_{j_2})\big] \\
          & = \lambda^{k+1} F_{k+1} G_k.
      \end{split}
  \end{equation*}
\end{proof}

\begin{lemma}
  Let $\Theta = (\lambda, f, g)$ be BEDD parameters and $Y \sim \text{BEDD}(\Theta)$. For all $(i_1,i_2) \in \mathbb{N}^2$ such that $i_1 \neq i_2$ and for all $j \in \mathbb{N}$,
  \begin{equation*}
      \mathbb{E}[\Psi_k(Y_{i_1j_1})\times Y_{i_2j_1}] = \lambda^{k+1} F_{k} G_{k+1}.
  \end{equation*}
  \label{lem:ident2}
\end{lemma}

\begin{proof}
  Let $\Theta = (\lambda, f, g)$ be BEDD parameters and $Y \sim \mathcal{L}\text{-BEDD}(\Theta)$, for any $(i_1,i_2) \in \mathbb{N}^2$ such that $i_1 \neq i_2$ and for any $j \in \mathbb{N}$,
  \begin{equation*}
      \begin{split}
          \mathbb{E}[\Psi_k(Y_{i_1j}) \times Y_{i_2 j}] &= \mathbb{E}\big[\mathbb{E}[\Psi_k(Y_{i_1j}) \times Y_{i_2j} | \xi_{i_1}, \xi_{i_2}, \eta_{j}]\big] \\
          & = \mathbb{E}\big[\mathbb{E}[\Psi_k(Y_{i_1j}) | \xi_{i_1}, \eta_{j}] \times \mathbb{E}[Y_{i_2j} | \xi_{i_2}, \eta_{j}]\big] \\
          & = \mathbb{E}\big[\lambda^k f(\xi_{i_1})^k g(\eta_{j})^k \times \lambda f(\xi_{i_2}) g(\eta_{j})\big] \\
          & = \lambda^{k+1} \mathbb{E}\big[ f(\xi_{i_1})^{k}\big] \mathbb{E}\big[ f(\xi_{i_2})\big]   \mathbb{E}\big[ g(\eta_{j})^{k+1}\big] \\
          & = \lambda^{k+1} F_{k} G_{k+1}.
      \end{split}
  \end{equation*}
\end{proof}

\begin{proof}[Proof of Theorem~\ref{th:id_quadruplet}]
  Let $\Theta = (\lambda, f, g)$ be BEDD parameters and $Y \sim \mathcal{L}\text{-BEDD}(\Theta)$. Let $Y_{(i_1, i_2; j_1, j_2)}$ be a quadruplet of $Y$. Since Assumption~\ref{ass:distfunc} holds, we set the $(\Psi_k)_{k \ge 1}$ such that $\mathbb{E}[\Psi_k(X)] = \mu^k$ for all $k \in \mathbb{N}$. We know that $\mathbb{E}[Y_{i_1j_1}] = \lambda$.
  \begin{itemize}
      \item First, recall that $F_1 = G_1 = 1$.
      \item Then, having recovered $\lambda$ and all the $F_\ell$, $G_\ell$ for $1 \le \ell \le k$, we can recover $F_{k+1}$ and $G_{k+1}$ as from Lemmas~\ref{lem:ident1} and~\ref{lem:ident2},
      \begin{equation*}
          \begin{split}
              F_{k+1} &= \lambda^{-(k+1)}  G_k^{-1} \mathbb{E}[\Psi_k(Y_{i_1j_1}) \times Y_{i_1j_2}]\\
              G_{k+1} &= \lambda^{-(k+1)} F_{k}^{-1} \mathbb{E}[\Psi_k(Y_{i_1j_1})\times Y_{i_2j_1}]
          \end{split}
      \end{equation*}
  \end{itemize}
  Then, the theorem is proven by induction. 
\end{proof}

\section{Derivation of variances}

In this section, we derive a general formula for the covariance of two $U$-statistics and then we derive asymptotic variances for specific kernels used in Section~\ref{sec:app}. We denote for any $k > 0$, $F_k := \int f^k$ and $G_k := \int g^k$.

\begin{lemma}
Let $Y$ be a RCE matrix. Let $h_1$ and $h_2$ be two quadruplet kernels. Let $U^{h_1}_N := U^{h_1}_{m_N, n_N}$ and $U^{h_2}_N := U^{h_2}_{m_N, n_N}$, defined by Formula~\eqref{eq:ustats} and Definition~\ref{def:mn}. For any $\sigma$-field $\mathcal{F}$,
\begin{equation*}
\begin{split}
    \Cov(U^{h_1}_N, U^{h_2}_N | \mathcal{F}) &= \frac{4}{cN} \Cov(h_1(Y_{(1,2;1,2)}), h_2(Y_{(1,3;3,4)}) | \mathcal{F}) \\
    & + \frac{4}{(1-c)N}  \Cov(h_1(Y_{(1,2;1,2)}), h_2(Y_{(3,4;1,3)}) | \mathcal{F}) + o\left(\frac{1}{N}\right).
\end{split}
\end{equation*}
\label{lem:covariance_formula}
\end{lemma}

\begin{proof}
  Similar to the proof of Lemma~\ref{lem:form_s}, using the exchangeability of the quadruplets, we deduce that
  \begin{equation*}
      \begin{split}
          \Cov(U^{h_1}_N, U^{h_2}_N | \mathcal{F}) &= \binom{m_N}{2}^{-2} \binom{n_N}{2}^{-2} \Cov\left(\sum_{\substack{1 \le i_1 < i_2 \le m_N \\ 1 \le j_1 < j_2 \le n_N}} h_1(Y_{(i_1,i_2;j_1,j_2)}), \sum_{\substack{1 \le i_1 < i_2 \le m_N \\ 1 \le j_1 < j_2 \le n_N}} h_2(Y_{(i_1,i_2;j_1,j_2)}) \Bigg| \mathcal{F}\right), \\
          &= \binom{m_N}{2}^{-2} \binom{n_N}{2}^{-2} \sum_{\substack{1 \le i_1 < i_2 \le m_N \\ 1 \le j_1 < j_2 \le n_N}} \sum_{\substack{1 \le i'_1 < i'_2 \le m_N \\ 1 \le j'_1 < j'_2 \le n_N}} \Cov(h_1(Y_{(i_1,i_2;j_1,j_2)}), h_2(Y_{(i'_1,i'_2;j'_1,j'_2)}) | \mathcal{F}), \\
          &= \binom{m_N}{2}^{-1} \binom{n_N}{2}^{-1} \sum_{\substack{1 \le i'_1 < i'_2 \le m_N \\ 1 \le j'_1 < j'_2 \le n_N}} \Cov(h_1(Y_{(1,2;1,2)}), h_2(Y_{(i'_1,i'_2;j'_1,j'_2)}) | \mathcal{F}), \\
          &= \binom{m_N}{2}^{-1} \binom{n_N}{2}^{-1} \bigg((m_N-2) (n_N -2) (n_N-3) \Cov(h_1(Y_{(1,2;1,2)}), h_2(Y_{(1,3;3,4)}) | \mathcal{F}) \\
          & + (m_N-2) (m_N-3) (n_N -2)  \Cov(h_1(Y_{(1,2;1,2)}), h_2(Y_{(3,4;1,3)}) | \mathcal{F}) \\
          & + o(m_N n_N^2) + o(m_N^2 n_N)\bigg), \\
          &= \frac{4}{c^2 (1-c)^2 N^4} \bigg(c(1-c)^2 N^3 \Cov(h_1(Y_{(1,2;1,2)}), h_2(Y_{(1,3;3,4)}) | \mathcal{F}) \\
          & + c^2(1-c) N^3  \Cov(h_1(Y_{(1,2;1,2)}), h_2(Y_{(3,4;1,3)}) | \mathcal{F}) + o(N^3)\bigg), \\
          &= \frac{4}{cN} \Cov(h_1(Y_{(1,2;1,2)}), h_2(Y_{(1,3;3,4)}) | \mathcal{F}) \\
          & + \frac{4}{(1-c)N}  \Cov(h_1(Y_{(1,2;1,2)}), h_2(Y_{(3,4;1,3)}) | \mathcal{F}) + o\left(\frac{1}{N}\right). \\
      \end{split}
  \end{equation*}
\end{proof}

\begin{corollary}
    Let $Y$ be a RCE matrix. Let $h_1$ and $h_2$ be two quadruplet kernels. Let $U^{h_1}_N := U^{h_1}_{m_N, n_N}$ and $U^{h_2}_N := U^{h_2}_{m_N, n_N}$, defined by Formula~\eqref{eq:ustats} and Definition~\ref{def:mn}. For any $\sigma$-field $\mathcal{F}$,
\begin{equation*}
\begin{split}
    \mathbb{V}[U^{h}_N | \mathcal{F}] &= \frac{4}{cN} \Cov(h(Y_{(1,2;1,2)}), h(Y_{(1,3;3,4)}) | \mathcal{F}) \\
    & + \frac{4}{(1-c)N}  \Cov(h(Y_{(1,2;1,2)}), h(Y_{(3,4;1,3)}) | \mathcal{F}) + o\left(\frac{1}{N}\right).
\end{split}
\end{equation*}
\label{cor:asymptotic_variance}
\end{corollary}

\begin{lemma}
Let $Y$ be a RCE matrix generated by the Poisson-BEDD model, with density $\lambda$ and degree functions $f$ and $g$. Let $h$ be a quadruplet kernel defined by
\begin{equation*}
    h(Y_{(i_1,i_2;j_1,j_2)}) = \frac{1}{2}(Y_{i_1j_1}Y_{i_1j_2} + Y_{i_2j_1}Y_{i_2j_2}),
\end{equation*}
then a closed-form expression for $V$ of Theorem~\ref{th:gaussian_theorem} is
\begin{equation*}
    V = \frac{\lambda^4}{c} (F_4 - F_2^2) + \frac{4 \lambda^4 }{1-c}  F_2^2 (G_2 - 1).
\end{equation*}
\label{lem:expression_vh1}
\end{lemma}

\begin{proof}
Using the fact that the $(Y_{ij})_{(i,j) \in \mathbb{N}^2}$ are independent conditionally on $\xi = (\xi_i)_{i\in\mathbb{N}}$ and $\eta = (\eta_j)_{j\in\mathbb{N}}$, we find that
  \begin{equation*}
      \begin{split}
          \Cov\big(h(Y_{(1,2;1,2)}),h(Y_{(1,3;3,4)})\big) &=  \frac{1}{4}\Cov\big(Y_{11}Y_{12} + Y_{21}Y_{22},Y_{13}Y_{14} + Y_{33}Y_{34}\big), \\
          &= \frac{1}{4}\Cov\big(Y_{11}Y_{12},Y_{13}Y_{14}\big), \\
          &= \frac{1}{4} \left( \mathbb{E}[Y_{11}Y_{12}Y_{13}Y_{14}] - \mathbb{E}[Y_{11}Y_{12}]\mathbb{E}[Y_{13}Y_{14}] \right),\\
          &= \frac{1}{4} \left( \mathbb{E}[Y_{11}Y_{12}Y_{13}Y_{14}] - \mathbb{E}[Y_{11}Y_{12}]^2 \right),\\
          &= \frac{1}{4} \left( \mathbb{E}[\mathbb{E}[Y_{11}Y_{12}Y_{13}Y_{14}|\xi, \eta]] - \mathbb{E}[\mathbb{E}[Y_{11}Y_{12}|\xi, \eta]]^2 \right),\\
          &= \frac{1}{4}\left( \mathbb{E}[\lambda^4 f(\xi_1)^4 g(\eta_1)g(\eta_2)g(\eta_3)g(\eta_4)] - \mathbb{E}[\lambda^2f(\xi_1)^2 g(\eta_1)g(\eta_2)]^2 \right),\\
          &= \frac{\lambda^4}{4}(F_4-F_2^2),
      \end{split}
  \end{equation*}
and
  \begin{equation*}
      \begin{split}
          \Cov\big(h(Y_{(1,2;1,2)}),h(Y_{(3,4;1,3)})\big) &=  \frac{1}{4}\Cov\big(Y_{11}Y_{12} + Y_{21}Y_{22},Y_{31}Y_{33} + Y_{41}Y_{43}\big), \\
          &= \frac{1}{4} \times 4 \Cov\big(Y_{11}Y_{12} ,Y_{31}Y_{33}\big), \\
          &= \mathbb{E}[Y_{11}Y_{12}Y_{31}Y_{33}] - \mathbb{E}[Y_{11}Y_{12}]\mathbb{E}[Y_{31}Y_{33}],\\
          &= \mathbb{E}[Y_{11}Y_{12}Y_{31}Y_{33}] - \mathbb{E}[Y_{11}Y_{12}]^2,\\
          &= \mathbb{E}[\mathbb{E}[Y_{11}Y_{12}Y_{31}Y_{33}|\xi, \eta]] - \mathbb{E}[\mathbb{E}[Y_{11}Y_{12}|\xi, \eta]]^2,\\
          &= \mathbb{E}[\lambda^4 f(\xi_1)^2f(\xi_3)^2 g(\eta_1)^2g(\eta_2)g(\eta_3)] - \mathbb{E}[\lambda^2f(\xi_1)^2 g(\eta_1)g(\eta_2)]^2,\\
          &= \lambda^4 (F_2^2 G_2-F_2^2),\\
          &= \lambda^4 F_2^2 (G_2-1).
      \end{split}
  \end{equation*}
  The proof is concluded injecting these results in the expression of $V = \frac{4}{c}\Cov\big(h(Y_{(1,2;1,2)}),h(Y_{(1,3;3,4)})\big) + \frac{4}{1-c}\Cov\big(h(Y_{(1,2;1,2)}),h(Y_{(3,4;1,3)})\big)$ given by Theorem~\ref{th:gaussian_theorem}.
\end{proof}

\begin{lemma}
Let $Y$ be a RCE matrix generated by the Poisson-BEDD model, with density $\lambda$ and degree functions $f$ and $g$. Let $h$ be a quadruplet kernel defined by
\begin{equation*}
    h(Y_{(i_1,i_2;j_1,j_2)}) = \frac{1}{2}(Y_{i_1j_1} Y_{i_2j_2} + Y_{i_1j_2} Y_{i_2j_1}),
\end{equation*}
then a closed-form expression for $V$ of Theorem~\ref{th:gaussian_theorem} is
\begin{equation*}
    V = \frac{4 \lambda^4}{c} (F_2 - 1) + \frac{4 \lambda^4}{1-c}  (G_2 - 1).
\end{equation*}
\label{lem:expression_vh2}
\end{lemma}

\begin{proof}
Using the fact that the $(Y_{ij})_{(i,j) \in \mathbb{N}^2}$ are independent conditionally on $\xi = (\xi_i)_{i\in\mathbb{N}}$ and $\eta = (\eta_j)_{j\in\mathbb{N}}$, we find that
  \begin{equation*}
      \begin{split}
          \Cov\big(h(Y_{(1,2;1,2)}),h(Y_{(1,3;3,4)})\big) &=  \frac{1}{4}\Cov\big(Y_{11}Y_{22} + Y_{12}Y_{21},Y_{13}Y_{34} + Y_{14}Y_{33}\big), \\
          &= \frac{1}{4} \times 4 \Cov\big(Y_{11}Y_{22},Y_{13}Y_{34}\big), \\
          &= \mathbb{E}[Y_{11}Y_{22}Y_{13}Y_{34}] - \mathbb{E}[Y_{11}Y_{22}]\mathbb{E}[Y_{13}Y_{34}],\\
          &= \mathbb{E}[Y_{11}Y_{22}Y_{13}Y_{34}]  - \mathbb{E}[Y_{11}Y_{22}]^2,\\
          &= \mathbb{E}[\mathbb{E}[Y_{11}Y_{22}Y_{13}Y_{34}|\xi, \eta]] - \mathbb{E}[\mathbb{E}[Y_{11}Y_{22}|\xi, \eta]]^2,\\
          &= \mathbb{E}[\lambda^4 f(\xi_1)^2f(\xi_2)f(\xi_3) g(\eta_1)g(\eta_2)g(\eta_3)g(\eta_4)] - \mathbb{E}[\lambda^2f(\xi_1)f(\xi_2) g(\eta_1)g(\eta_2)]^2,\\
          &= \lambda^4 (F_2-1),
      \end{split}
  \end{equation*}
and
  \begin{equation*}
      \begin{split}
          \Cov\big(h(Y_{(1,2;1,2)}),h(Y_{(3,4;1,3)})\big) &=  \frac{1}{4}\Cov\big(Y_{11}Y_{22} + Y_{12}Y_{21},Y_{31}Y_{43} + Y_{33}Y_{41}\big), \\
          &= \frac{1}{4} \times 4 \Cov\big(Y_{11}Y_{22} ,Y_{31}Y_{43}\big), \\
          &= \mathbb{E}[Y_{11}Y_{22}Y_{31}Y_{43}] - \mathbb{E}[Y_{11}Y_{22}]\mathbb{E}[Y_{31}Y_{43}],\\
          &= \mathbb{E}[Y_{11}Y_{22}Y_{31}Y_{43}] - \mathbb{E}[Y_{11}Y_{22}]^2,\\
          &= \mathbb{E}[\mathbb{E}[Y_{11}Y_{22}Y_{31}Y_{43}|\xi, \eta]] - \mathbb{E}[\mathbb{E}[Y_{11}Y_{22}|\xi, \eta]]^2,\\
          &= \mathbb{E}[\lambda^4 f(\xi_1)f(\xi_2)f(\xi_3)f(\xi_4) g(\eta_1)^2g(\eta_2)g(\eta_3)] - \mathbb{E}[\lambda^2f(\xi_1)f(\xi_2) g(\eta_1)g(\eta_2)]^2,\\
          &= \lambda^4 (G_2-1).
      \end{split}
  \end{equation*}
  The proof is concluded injecting these results in the expression of $V = \frac{4}{c}\Cov\big(h(Y_{(1,2;1,2)}),h(Y_{(1,3;3,4)})\big) + \frac{4}{1-c}\Cov\big(h(Y_{(1,2;1,2)}),h(Y_{(3,4;1,3)})\big)$ given by Theorem~\ref{th:gaussian_theorem}.
\end{proof}

\begin{lemma}
Let $Y$ be a RCE matrix generated by the Poisson-BEDD model, with density $\lambda$ and degree functions $f$ and $g$. Let $h_1$ and $h_2$ be two quadruplet kernels defined by
\begin{equation*}
    h_1(Y_{(i_1,i_2;j_1,j_2)}) = \frac{1}{2}(Y_{i_1j_1}Y_{i_1j_2} + Y_{i_2j_1}Y_{i_2j_2}),
\end{equation*}
and
\begin{equation*}
    h_2(Y_{(i_1,i_2;j_1,j_2)}) = \frac{1}{2}(Y_{i_1j_1} Y_{i_2j_2} + Y_{i_1j_2} Y_{i_2j_1}).
\end{equation*}
Let $U^{h_1}_N := U^{h_1}_{m_N, n_N}$ and $U^{h_2}_N := U^{h_2}_{m_N, n_N}$, defined by Formula~\eqref{eq:ustats} and Definition~\ref{def:mn}. Set $C^{h_1,h_2} := \lim_{N \rightarrow +\infty} N \Cov(U^{h_1}_N, U^{h_2}_N)$, then 
\begin{equation*}
    C^{h_1,h_2} = \frac{2 \lambda^4}{ c} (F_3 - F_2) + \frac{4 \lambda^4}{1-c} F_2 (G_2 - 1).
\end{equation*}
\label{lem:expression_ch1h2}
\end{lemma}

\begin{proof}
  First, using Lemma~\ref{lem:covariance_formula}, we deduce that
  \begin{equation*}
      \begin{split}
          \Cov(U^{h_1}_N, U^{h_2}_N) &= \frac{4}{cN} \Cov(h_1(Y_{(1,2;1,2)}), h_2(Y_{(1,3;3,4)})) \\
          & + \frac{4}{(1-c)N}  \Cov(h_1(Y_{(1,2;1,2)}), h_2(Y_{(3,4;1,3)})) + o\left(\frac{1}{N}\right), \\
      \end{split}
  \end{equation*}
  so that 
  \begin{equation*}
      C^{h_1,h_2} = \frac{4}{c} \Cov(h_1(Y_{(1,2;1,2)}), h_2(Y_{(1,3;3,4)})) + \frac{4}{1-c}  \Cov(h_1(Y_{(1,2;1,2)}), h_2(Y_{(3,4;1,3)})).
  \end{equation*}
  To conclude the proof, we proceed analogously to the proofs of Lemma~\ref{lem:expression_vh1} and~\ref{lem:expression_vh2} to derive the expressions of
  \begin{equation*}
  \begin{split}
      \Cov(h_1(Y_{(1,2;1,2)}), h_2(Y_{(1,3;3,4)})) &= \frac{1}{4} \Cov(Y_{11}Y_{12} + Y_{21}Y_{22}, Y_{13}Y_{34} + Y_{33}Y_{14}), \\
      &= \frac{1}{4} \Cov(Y_{11}Y_{12}, Y_{13}Y_{34} + Y_{33}Y_{14}), \\
      &= \frac{1}{4} \times 2\Cov(Y_{11}Y_{12}, Y_{13}Y_{34}), \\
      &= \frac{1}{2} \left(\mathbb{E}[Y_{11}Y_{12}Y_{13}Y_{34}] - \mathbb{E}[Y_{11}Y_{12}]\mathbb{E}[Y_{13}Y_{34}] \right),\\
      &= \frac{1}{2} \left(\mathbb{E}[Y_{11}Y_{12}Y_{13}Y_{34}] - \mathbb{E}[Y_{11}Y_{12}]\mathbb{E}[Y_{11}Y_{22}] \right),\\
      &= \frac{1}{2} \left(\mathbb{E}[\mathbb{E}[Y_{11}Y_{12}Y_{13}Y_{34}|\xi, \eta]] - \mathbb{E}[\mathbb{E}[Y_{11}Y_{12}|\xi, \eta]]\mathbb{E}[\mathbb{E}[Y_{11}Y_{22}|\xi, \eta]] \right),\\
      &= \frac{1}{2} \left(\mathbb{E}[\lambda^4 f(\xi_1)^3 f(\xi_3) g(\eta_1)g(\eta_2)g(\eta_3)g(\eta_4)]\right. \\
      &\quad \left.- \mathbb{E}[\lambda^2f(\xi_1)^2 g(\eta_1)g(\eta_2)]\mathbb{E}[\lambda^2f(\xi_1)f(\xi_2) g(\eta_1)g(\eta_2)]\right),\\
      &= \frac{1}{2}\lambda^4 (F_3-F_2),
  \end{split}
  \end{equation*}
  and
  \begin{equation*}
  \begin{split}
      \Cov(h_1(Y_{(1,2;1,2)}), h_2(Y_{(3,4;1,3)})) &= \frac{1}{4} \Cov(Y_{11}Y_{12} + Y_{21}Y_{22}, Y_{31}Y_{43} + Y_{41}Y_{33}), \\
      &= \frac{1}{4} \times 4\Cov(Y_{11}Y_{12}, Y_{31}Y_{43}), \\
      &= \mathbb{E}[Y_{11}Y_{12}Y_{31}Y_{43}] - \mathbb{E}[Y_{11}Y_{12}]\mathbb{E}[Y_{31}Y_{43}],\\
      &= \mathbb{E}[Y_{11}Y_{12}Y_{31}Y_{43}] - \mathbb{E}[Y_{11}Y_{12}]\mathbb{E}[Y_{11}Y_{22}],\\
      &= \mathbb{E}[\mathbb{E}[Y_{11}Y_{12}Y_{31}Y_{43}|\xi, \eta]] - \mathbb{E}[\mathbb{E}[Y_{11}Y_{12}|\xi, \eta]]\mathbb{E}[\mathbb{E}[Y_{11}Y_{22}|\xi, \eta]],\\
      &= \mathbb{E}[\lambda^4 f(\xi_1)^2 f(\xi_3) f(\xi_4) g(\eta_1)^2g(\eta_2)g(\eta_3)] \\
      &\quad - \mathbb{E}[\lambda^2f(\xi_1)^2 g(\eta_1)g(\eta_2)]\mathbb{E}[\lambda^2f(\xi_1)f(\xi_2) g(\eta_1)g(\eta_2)],\\
      &= \lambda^4 F_2(G_2-1).
  \end{split}
  \end{equation*}
\end{proof}

\begin{lemma}
Let $Y$ be a RCE matrix generated by the Bernoulli-BEDD model, with density $\lambda$ and degree functions $f$ and $g$. Let $h$ be the quadruplet kernel defined by
\begin{equation*}
\begin{split}
    h(Y_{(i_1,i_2;j_1,j_2)}) =& \frac{1}{4} \bigg(Y_{i_1j_1} Y_{i_1j_2} Y_{i_2j_1} (1 - Y_{i_2j_2}) + Y_{i_1j_1} Y_{i_1j_2} Y_{i_2j_2} (1 - Y_{i_2j_1}) \\
    &+ Y_{i_1j_1} Y_{i_2j_1} Y_{i_2j_2} (1 - Y_{i_1j_2}) + Y_{i_1j_2}Y_{i_2j_1} Y_{i_2j_2} (1 - Y_{i_1j_1}) \bigg),
\end{split}
\end{equation*}
then $\mathbb{E}[h(Y_{(1,2;1,2)})] = \lambda^3 F_2 G_2 (1 - \lambda F_2 G_2)$ and a closed-form expression for $V$ of Theorem~\ref{th:gaussian_theorem} is
\begin{equation*}
\begin{split}
    V &= \frac{4 \lambda^6}{c} G_2^2 \left[\lambda^2 F_4 F_2^2 G_2^2 - \lambda F_4 F_2 G_2 - \lambda F_3 F_2^2 G_2 + \frac{1}{2} F_3 F_2 + \frac{1}{4} F_4 + \frac{1}{4} F_2^3 \right] \\
    & + \frac{4 \lambda^6}{1-c} F_2^2 \left[\lambda^2 G_4 G_2^2 F_2^2 - \lambda G_4 G_2 F_2 - \lambda G_3 G_2^2 F_2 + \frac{1}{2} G_3 G_2 + \frac{1}{4} G_4 + \frac{1}{4} G_2^3 \right] \\
    & - \frac{4 \lambda^6}{c(1-c)} F_2^2 G_2^2 (1 - \lambda F_2 G_2)^2.
\end{split}
\end{equation*}
\label{lem:motif_variance}
\end{lemma}

\begin{proof}
  First, the expectation of $h$ is
  \begin{equation*}
  \begin{split}
      \mathbb{E}[h(Y_{(1,2;1,2)})] &= \mathbb{E}[Y_{11} Y_{12} Y_{21} (1 - Y_{22})] \\
      &= \mathbb{E}[Y_{11} Y_{12} Y_{21}] - \mathbb{E}[Y_{11} Y_{12} Y_{21}Y_{22}] \\
      &= \mathbb{E}[\mathbb{E}[Y_{11} Y_{12} Y_{21}|\xi, \eta]] - \mathbb{E}[\mathbb{E}[Y_{11} Y_{12} Y_{21}Y_{22}|\xi, \eta]] \\
      &= \mathbb{E}[\lambda^3 f(\xi_1)^2 f(\xi_2) g(\xi_1)^2 g(\xi_2)] - \mathbb{E}[\lambda^4 f(\xi_1)^2 f(\xi_2)^2 g(\xi_1)^2 g(\xi_2)^2] \\
      &= \lambda^3 F_2 G_2 - \lambda^4 F_2^2 G_2^2. 
  \end{split}
  \end{equation*}
  Now we derive the expression of $V$. From the expression given by Theorem~\ref{th:gaussian_theorem}, we deduce the following form
  \begin{equation*}
      V = \frac{4}{c}\mathbb{E}[h(Y_{(1,2;1,2)}) h(Y_{(1,3;3,4)})] + \frac{4}{1-c}\mathbb{E}[h(Y_{(1,2;1,2)}) h(Y_{(3,4;1,3)})] - \frac{4}{c(1-c)}\mathbb{E}[h(Y_{(1,2;1,2)})]^2.
  \end{equation*}
  Since the calculation of $\mathbb{E}[h(Y_{(1,2;1,2)}) h(Y_{(1,3;3,4)})]$ and $\mathbb{E}[h(Y_{(1,2;1,2)}) h(Y_{(3,4;1,3)})]$ is completely symmetric in this case, we only need to prove that
  \begin{equation}
      \mathbb{E}[h(Y_{(1,2;1,2)}) h(Y_{(1,3;3,4)})] = \lambda^6 G_2^2 \left[\lambda^2 F_4 F_2^2 G_2^2 - \lambda F_4 F_2 G_2 - \lambda F_3 F_2^2 G_2 + \frac{1}{2} F_3 F_2 + \frac{1}{4} F_4 + \frac{1}{4} F_2^3 \right].
      \label{eq:intermediate_lemma_motif}
  \end{equation}
  The direct derivation of this quantity is more tedious than technical. Using symmetries and exchangeability, one can decompose it.
 \begin{equation*}
     \begin{split}
        \mathbb{E}[h(Y_{(1,2;1,2)}) h(Y_{(1,3;3,4)})] &= \frac{1}{16}\mathbb{E}\bigg[\bigg(Y_{11} Y_{12} Y_{21} (1 - Y_{22}) + Y_{11} Y_{12} Y_{22} (1 - Y_{21}) \\
         &+ Y_{11} Y_{21} Y_{22} (1 - Y_{12}) + Y_{12}Y_{21} Y_{22} (1 - Y_{11})\bigg) \\
         & \times \bigg(Y_{13} Y_{14} Y_{33} (1 - Y_{34}) + Y_{13} Y_{14} Y_{34} (1 - Y_{33}) \\
         &+ Y_{13} Y_{33} Y_{34} (1 - Y_{14}) + Y_{14}Y_{33} Y_{34} (1 - Y_{13})\bigg)\bigg] \\
         &= \frac{1}{4} \mathbb{E}\bigg[\bigg(Y_{11} Y_{12} Y_{21} (1 - Y_{22}) +  Y_{11} Y_{21} Y_{22} (1 - Y_{12})\bigg) \\
         & \times \bigg(Y_{13} Y_{14} Y_{33} (1 - Y_{34}) + Y_{13} Y_{33} Y_{34} (1 - Y_{14}))\bigg)\bigg] \\
         &= \frac{1}{4} \bigg( 4 \mathbb{E}[Y_{11} Y_{12} Y_{21} Y_{22} Y_{13} Y_{14} Y_{33} Y_{34}] \\
         &\quad - 4 \mathbb{E}[Y_{11} Y_{12} Y_{21} Y_{22} Y_{13} Y_{14} Y_{33}] \\
         &\quad - 4 \mathbb{E}[Y_{11} Y_{12} Y_{21} Y_{22} Y_{13} Y_{33} Y_{34}] \\
         &\quad + 2 \mathbb{E}[Y_{11} Y_{12} Y_{21} Y_{13} Y_{33} Y_{34}] \\
         &\quad + \mathbb{E}[Y_{11} Y_{12} Y_{21} Y_{13} Y_{14} Y_{33}] \\
         &\quad + \mathbb{E}[Y_{11} Y_{21} Y_{22} Y_{13} Y_{33} Y_{34}]\bigg).
     \end{split}
 \end{equation*}
 Now we derive each simpler expectation.
 \begin{equation*}
     \begin{split}
         \mathbb{E}[Y_{11} Y_{12} Y_{21} Y_{22} Y_{13} Y_{14} Y_{33} Y_{34}] &= \mathbb{E}[\mathbb{E}[Y_{11} Y_{12} Y_{21} Y_{22} Y_{13} Y_{14} Y_{33} Y_{34}|\xi, \eta]] \\
         &= \mathbb{E}[\lambda^8 f(\xi_1)^4 f(\xi_2)^2 f(\xi_3)^2 g(\eta_1)^2 g(\eta_2)^2 g(\eta_3)^2 g(\eta_4)^2] \\
         &= \lambda^8 F_4 F_2^2 G_2^4.
     \end{split}
 \end{equation*}
 \begin{equation*}
     \begin{split}
         \mathbb{E}[Y_{11} Y_{12} Y_{21} Y_{22} Y_{13} Y_{14} Y_{33}] &= \mathbb{E}[\mathbb{E}[Y_{11} Y_{12} Y_{21} Y_{22} Y_{13} Y_{14} Y_{33}|\xi, \eta]] \\
         &= \mathbb{E}[\lambda^8 f(\xi_1)^4 f(\xi_2)^2 f(\xi_3) g(\eta_1)^2 g(\eta_2)^2 g(\eta_3)^2 g(\eta_4)] \\
         &= \lambda^7 F_4 F_2 G_2^3.
     \end{split}
 \end{equation*}
 \begin{equation*}
     \begin{split}
         \mathbb{E}[Y_{11} Y_{12} Y_{21} Y_{22} Y_{13} Y_{33} Y_{34}] &= \mathbb{E}[\mathbb{E}[Y_{11} Y_{12} Y_{21} Y_{22} Y_{13} Y_{33} Y_{34}|\xi, \eta]] \\
         &= \mathbb{E}[\lambda^7 f(\xi_1)^3 f(\xi_2)^2 f(\xi_3)^2 g(\eta_1)^2 g(\eta_2)^2 g(\eta_3)^2 g(\eta_4)] \\
         &= \lambda^7 F_3 F_2 G_2^3.
     \end{split}
 \end{equation*}
 \begin{equation*}
     \begin{split}
         \mathbb{E}[Y_{11} Y_{12} Y_{21} Y_{13} Y_{33} Y_{34}] &= \mathbb{E}[\mathbb{E}[Y_{11} Y_{12} Y_{21} Y_{13} Y_{33} Y_{34}|\xi, \eta]] \\
         &= \mathbb{E}[\lambda^6 f(\xi_1)^3 f(\xi_2) f(\xi_3)^2 g(\eta_1)^2 g(\eta_2) g(\eta_3)^2 g(\eta_4)] \\
         &= \lambda^6 F_3 F_2 G_2^2.
     \end{split}
 \end{equation*}
 \begin{equation*}
     \begin{split}
         \mathbb{E}[Y_{11} Y_{12} Y_{21} Y_{13} Y_{14} Y_{33}] &= \mathbb{E}[\mathbb{E}[Y_{11} Y_{12} Y_{21} Y_{13} Y_{14} Y_{33}|\xi, \eta]] \\
         &= \mathbb{E}[\lambda^6 f(\xi_1)^4 f(\xi_2) f(\xi_3) g(\eta_1)^2 g(\eta_2) g(\eta_3)^2 g(\eta_4)] \\
         &= \lambda^6 F_4 G_2^2.
     \end{split}
 \end{equation*}
 \begin{equation*}
     \begin{split}
         \mathbb{E}[Y_{11} Y_{21} Y_{22} Y_{13} Y_{33} Y_{34}] &= \mathbb{E}[\mathbb{E}[Y_{11} Y_{21} Y_{22} Y_{13} Y_{33} Y_{34}|\xi, \eta]] \\
         &= \mathbb{E}[\lambda^6 f(\xi_1)^2 f(\xi_2)^2 f(\xi_3)^2 g(\eta_1)^2 g(\eta_2) g(\eta_3)^2 g(\eta_4)] \\
         &= \lambda^6 F_2^3 G_2^2.
     \end{split}
 \end{equation*}
 Injecting these expressions into~\eqref{eq:intermediate_lemma_motif}, we find the correct expression for $\mathbb{E}[h(Y_{(1,2;1,2)}) h(Y_{(1,3;3,4)})]$, which concludes the proof.
\end{proof}

\begin{lemma}
Let $Y$ be a RCE matrix generated by the Bernoulli-BEDD model, with density $\lambda$ and degree functions $f$ and $g$. For $p \ge 1$ and $q \ge 1$, let $h_{p,q}$ be the quadruplet kernel defined by
\begin{equation*}
    h_{p,q}(Y_{(i_1,...,i_{p};j_1,...,j_{q})}) = \prod_{u = 1}^{p} \prod_{v = 1}^{q} Y_{i_u j_v}.
\end{equation*}
Then $\mathbb{E}[h_{p,q}(Y_{(1,...,p;1,...,q)})] = \lambda^{pq} F_q^p G_p^q$.
\label{lem:product_pq_kernel}
\end{lemma}

\begin{proof}
  Direct derivation gives
  \begin{equation*}
      \begin{split}
          \mathbb{E}[h_{p,q}(Y_{(1,...,p;1,...,q)})] &= \mathbb{E}[\prod_{i = 1}^{p} \prod_{j = 1}^{q} Y_{ij}] \\
          &= \mathbb{E}[\mathbb{E}[\prod_{i = 1}^{p} \prod_{j = 1}^{q} Y_{i j}|\xi, \eta]] \\
          &= \mathbb{E}[\prod_{i = 1}^{p} \prod_{j = 1}^{q} \mathbb{E}[Y_{i j}|\xi, \eta]] \\
          &= \mathbb{E}[\prod_{i = 1}^{p} \prod_{j = 1}^{q} \lambda f(\xi_i) g(\eta_j)] \\
          &= \lambda^{pq} \prod_{i = 1}^{p} \mathbb{E}[f(\xi_i)^q] \prod_{j = 1}^{q} \mathbb{E}[g(\eta_j)^p] \\
          &= \lambda^{pq} F_q^p G_p^q.
      \end{split}
  \end{equation*}
\end{proof}

\section{Some $U$-statistics written with matrix operations \label{app:matrix_operation}}

We denote for all $k \in \mathbb{N}$, $Y^{\odot k}$ the matrix (or vector) $Y$ elevated to the element-wise power $k$, i.e. $Y^{\odot k}_{ij} = Y_{ij}^k$ for all $i$ and $j$. 

Following formula~\eqref{eq:ustat_matrix_operation}, we write all the quadruplet $U$-statistics considered in the examples described in Sections~\ref{subsec:est} and \ref{subsec:comp} as simple operations on matrices. $U^{h_1}_{N}$ and $U^{h_2}_{N}$ are already given in these sections and
\begin{equation*}
\begin{split}
    U^{h_3}_{N} &= \frac{1}{m_N(m_N-1) n_N } \left[ |Y_NY_N^T|_1 - \Trace(Y_NY_N^T)\right], \\
    U^{h_4}_{N} &= \frac{1}{m_N n_N(n_N-1)} \left[ |\tilde{Y}_N^T\tilde{Y}_N|_1 - \Trace(\tilde{Y}_N^T\tilde{Y}_N)\right], \\
    U^{h_5}_{N} &= \frac{1}{m_N n_N} |Y_N|_1, \\
    U^{h_6}_{N} &= \frac{1}{m_N n_N(n_N-1)} \left[ |\tilde{Y}_N^TY_N|_1 - \Trace(\tilde{Y}_N^TY_N)\right].
\end{split}
\end{equation*}
where $\Trace$ is the trace operator and $\tilde{Y}$ is defined by $\tilde{Y}_{ij} = Y_{ij}^2 - Y_{ij}$.

The motif-counting $U$-statistic of Section~\ref{subsec:motifs} can be written as 
\begin{equation*}
\begin{split}
    U^{h_7}_N &= \frac{1}{m_N(m_N-1) n_N(n_N-1)} \left[ |Y_N^TY_NY_N^T|_1 - |(Y_N^{\odot 2})^T Y_N|_1 - |Y_N^{\odot 2} Y_N^T|_1 + \Trace(Y_N^{\odot 2} Y_N^T) \right. \\
    & \left. - \Trace(Y_N^TY_NY_N^TY_N) + |(Y_N^{\odot 2})^T Y_N^{\odot 2}|_1 + |Y_N^{\odot 2} (Y_N^{\odot 2})^T|_1 - \Trace((Y_N^{\odot 2})^T Y_N^{\odot 2}) \right].
\end{split}
\end{equation*}
The kernels $h_{p,q}$ are not quadruplet kernels, but they can also be simply computed if $p=1$ or $q=1$. We define respectively $r(Y_N)$ and $c(Y_N)$ the vector of row sums (degrees) and the vector of column sums (degrees) of the matrix $Y_N$. For all $1 \le i \le m_N$, $r(Y_N)_i = \sum_{j=1}^{n_N} Y_{ij}$ and for all $1 \le j \le n_N$, $c(Y_N)_j = \sum_{i=1}^{m_N} Y_{ij}$.
\begin{equation*}
\begin{split}
    U^{h_{1,p}}_{N} &= \left[m_N\binom{n_N}{p}\right]^{-1} \sum_{i=1}^{m_N} \binom{r(Y_N)_i}{p}, \\
    U^{h_{q,1}}_{N} &= \left[\binom{m_N}{q} n_N\right]^{-1} \sum_{j=1}^{n_N} \binom{c(Y_N)_j}{q}.
\end{split}
\end{equation*}
We also notice that 
\begin{equation*}
\begin{split}
    U^{h_{1,1}}_{N} &= U^{h_5}_N, \\
    U^{h_{1,2}}_{N} &= U^{h_1}_N, \\
    U^{h_{2,1}}_{N} &= U^{h_3}_N.
\end{split}
\end{equation*}

\end{appendix}

\section*{Acknowledgements}

The author thanks St\'ephane Robin (Sorbonne Universit\'e), Sophie Donnet (INRAE) and Fran\c{c}ois Massol (CNRS) for many fruitful discussions and insights. This work was funded by a grant from R\'egion \^Ile-de-France and by the grant ANR-18-CE02-0010-01 of the French National Research Agency ANR (project EcoNet).

\bibliographystyle{imsart-nameyear}
\bibliography{biblio.bib}

\end{document}